\documentclass[11pt, leqno]{amsart}

\usepackage{graphicx}
\usepackage{amssymb,amsmath,amsthm,amscd,MnSymbol}
\usepackage{mathrsfs}
\usepackage{lscape}
\usepackage{enumerate}
\usepackage[usenames,dvipsnames]{color}
\usepackage[colorlinks=true, pdfstartview=FitV,
 linkcolor=blue,citecolor=blue,urlcolor=blue]{hyperref}
\usepackage{tikz}
\tikzset{dynkdot/.style={circle,draw,scale=.38}}

\usepackage[all]{xy}

\allowdisplaybreaks[3]

\definecolor{darkred}{rgb}{0.7,0,0} 
\newcommand{\defn}[1]{{\color{darkred}\emph{#1}}} 

\numberwithin{equation}{section}

\theoremstyle{plain}
\newtheorem{lemma}{Lemma}[section]
\newtheorem{proposition}[lemma]{Proposition}
\newtheorem{theorem}[lemma]{Theorem}

\theoremstyle{definition}
\newtheorem{remark}[lemma]{Remark}
\newtheorem{conjecture}[lemma]{Conjecture}
\newtheorem{example}[lemma]{Example}
\newtheorem{algorithm}[lemma]{Algorithm}

\newtheorem{definition}[lemma]{Definition}
\newtheorem{corollary}[lemma]{Corollary}

\newcommand{\soc}{{\operatorname{soc}}}

\newcommand{\g}{\mathfrak{g}}
\newcommand{\h}{\mathfrak{h}}

\newcommand{\Q}{\mathbb{Q}}
\newcommand{\C}{\mathbb{C}}
\newcommand{\Z}{\mathbb{Z}}
\newcommand{\N}{\mathsf{N}}

\newcommand{\cmA}{\mathsf{A}}
\newcommand{\seteq}{\mathbin{:=}}

\newcommand{\al}{\alpha}

\newcommand{\bi}{\mathbf{i}}
\newcommand{\bj}{\mathbf{j}}

\newcommand{\PP}{ \textbf{\textit{P}}}

\newcommand{\RR}{ \textbf{\textit{R}}}
\newcommand{\Qd}{ \mathscr{Q} }
\newcommand{\PPi}{\PP_{\lf \Qd \rf}}
\newcommand{\lan}{\langle}
\newcommand{\ran}{\rangle}
\newcommand{\ko}{ \textbf{k}}

\newcommand{\ii}{ \textbf{\textit{i}}}
\newcommand{\jj}{ \textbf{\textit{j}}}

\newcommand{\ut}{ \textbf{\textit{t}}}

\newcommand{\prt}[2]{ \left( \begin{matrix} #1 \\ #2 \end{matrix}\right) }

\newcommand{\be}{\beta}
\newcommand{\ga}{\gamma}

\newcommand{\shp}{\hspace{-0.4ex}+\hspace{-0.4ex}}

\newcommand{\PR}{\Phi^+}

\newcommand{\redex}{{\widetilde{w}}}
\newcommand{\redez}{{\widetilde{w}_0}}
\newcommand{\um}{\underline{m}}

\newcommand{\us}{\underline{s}}
\newcommand{\un}{\underline{n}}
\newcommand{\up}{\underline{p}}

\newcommand{\tb}{\mathtt{b}}
\newcommand{\rl}{\mathsf{Q}}
\newcommand{\wt}{{\rm wt}}
\newcommand{\rev}{{\rm rev}}

\newcommand{\gdist}{{\rm gdist}}

\newcommand{\rds}{{\rm rds}}

\newcommand{\uUp}{\Upsilon}
\newcommand{\lf}{[\hspace{-0.3ex}[}
\newcommand{\rf}{]\hspace{-0.3ex}]}
\newcommand{\htau}{\ut\hspace{-1.5ex}-}

\newcommand{\wUp}{ \widehat{\Upsilon}}

\newcommand{\Rnorm}[1]{R^{\rm{norm}}_{#1}}
\newcommand{\Runiv}[1]{R^{\rm{univ}}_{#1}}

\newcommand{\Ca}{\mathcal{C}}
\newcommand{\Rren}[1]{R^{{\rm ren}}_{#1}}
\newcommand{\rmat}[1]{{\mathbf r}_{\mspace{-2mu}\raisebox{-.5ex}{${\scriptstyle{#1}}$}}}
\newcommand{\cqm}{\C[[q^{1/m}]]\;q^{1/m}}

\newcommand{\tens}{\mathop\otimes}

\newlength{\mylength}
\newenvironment{blue}
{\relax\color{blue}}
{\hspace*{.5ex}\relax}

\newcommand{\beb}{\begin{blue}}
\newcommand{\eb}{\end{blue}}

\title[Twisted AR-quivers and applications]{Twisted and folded Auslander-Reiten quivers and applications to the representation theory of quantum affine algebras}

\author[S.-j. Oh, U. R. Suh]{Se-jin Oh$^\ddagger$, Uhi Rinn  Suh$^\dagger$}

\address{Department of Mathematics, Ewha Womans University, Seoul 03760, Korea}
\email{sejin092@gmail.com}
\address{Department of Mathematical Sciences and Research institute of Mathematics, Seoul National University, GwanAkRo 1, Gwanak-Gu, Seoul 08826, Korea}
\email{uhrisu@gmail.com}

\thanks{$^\dagger$ This work was supported by the New Faculty Startup Fund from Seoul National University, NRF Grant \# 2016R1C1B1010721 and \# 2019R1F1A1059363.}
\thanks{$^\ddagger$ This work was supported by the Ministry of Education of the Republic of Korea and the National Research Foundation of Korea (NRF-2019R1A2C4069647)}


\begin{document}

\begin{abstract}
In this paper, we introduce twisted and folded AR-quivers of type $A_{2n+1}$, $D_{n+1}$, $E_6$ and $D_4$
associated to (triply) twisted Coxeter elements. Using the quivers of type $A_{2n+1}$ and $D_{n+1}$, we describe the denominator formulas and Dorey's rule for quantum affine algebras $U'_q(B^{(1)}_{n+1})$ and $U'_q(C^{(1)}_{n})$,
which are important information of representation theory of quantum affine algebras.
More precisely, we can read the denominator formulas for $U'_q(B^{(1)}_{n+1})$ (resp. $U'_q(C^{(1)}_{n})$) using certain statistics on any folded AR-quiver of type $A_{2n+1}$ (resp. $D_{n+1}$) and
Dorey's rule for $U'_q(B^{(1)}_{n+1})$ (resp. $U'_q(C^{(1)}_{n})$) applying the notion of minimal pairs in a twisted AR-quiver.
By adopting the same arguments, we propose the conjectural denominator formulas and Dorey's rule for $U'_q(F^{(1)}_{4})$ and $U'_q(G^{(1)}_{2})$.
\end{abstract}

\maketitle
\tableofcontents

\section*{Introduction}

The Auslander-Reiten (AR) quiver of an Artin algebra is a quiver whose vertices correspond to  indecomposable modules of the algebra and whose arrows correspond to {\it irreducible morphisms} between the modules. In particular, for a Dynkin quiver $Q$ of finite type $X=ADE$ and the corresponding path algebra $\C Q$,  Gabriel \cite{Gab80} showed that the corresponding AR quiver $\Gamma_Q$ consists of the set of vertices identified with the set $\Phi^+$ of positive roots.

On the other hand, for a given Dynkin quiver $Q$, there are the corresponding Coxeter element $\phi_Q$,  the convex partial order $\prec_Q$  and the set $[Q]$ of reduced expressions of  the longest element $w_0$ in the Weyl group $W$.
Interestingly,  the AR-quiver $\Gamma_Q$ is closely related to $\phi_Q$,  $\prec_Q$ and  $[Q]$ as follows (see Section \ref{subsec:Coxeter} and \ref{Subsec:AR-properties} for details).

$\bullet$ $\Gamma_Q$ is completely determined by the Coxeter element $\phi_Q$ or the Dynkin quiver $Q$.

$\bullet$ $\Gamma_Q$ can be understood as the Hesse diagram of $\prec_Q$ defined on $\Phi^+$. In other words,
$$ \text{for $\al,\be \in \PR$}, \quad \al \prec_Q \be \iff \text{ there exists a path from $\be$ to $\al$ in $\Gamma_Q$.}$$

$\bullet$ A reduced  expression $\redez$ of $w_0$ which is adapted to $Q$ can be obtained by reading the AR-quiver $\Gamma_Q$ {\it properly}. The set of reduced expressions adapted to $Q$ forms the commutation class $[Q]$.

Note that the family of  adapted commutation classes $[Q]$ is called the adapted $r$-cluster point $\lf Q\rf$. Here, the notion $r$-cluster point implies that these classes are related to each other by so called reflection functors. 

\vskip 2mm

Moreover, in \cite{Oh14A,Oh14D,Oh15E}, the first named author investigated that certain statistics of AR-quivers provide some significant information in the representation theories of quantum groups $U_q(X)$, quantum affine algebras $U_q'(X^{(1)})$ and KLR-algebras. Especially, he read the denominator formulas and Dorey's rule by constructing a simple algorithm for labeling $\Gamma_Q$ which depends only on its shape. More precisely, he showed the followings$\colon$

({\bf A}) Connections between Dorey's rule and AR-quivers :
Dorey \cite{Dor91} described relations between three-point couplings in the simply laced affine Toda field theories (ATFTs) and Lie theories. Afterwards, Chari-Pressley \cite{CP96} interpreted the phenomenon in terms  of finite dimensional integrable $U_q'(Y^{(1)})$-modules ($Y=A_n, B_n, C_n, D_n$), which is now referred to as Dorey's rule.
For type $A_n$ and $D_n$, they crucially used Coxeter elements of the corresponding type. Inspired from their work, Dorey's rule for $U_q'(A^{(1)}_{n})$, $U_q'(D^{(1)}_{n})$ and $U_q'(E_n^{(1)})$ can be interpreted in terms of coordinates of AR quivers \cite{Oh14A,Oh14D,Oh15E}. That is, for $i$-th fundamental representation $V(\varpi_i)$, the condition that
\begin{equation}\label{Eqn:Dorey}
{\rm Hom}( V(\varpi_k)_z ,V(\varpi_i)_x\otimes V(\varpi_j)_y)\neq 0
 \end{equation}
is equivalent to (i) $x=(-q)^a$, $y= (-q)^b$, $z=(-q)^c$ and
(ii)  $\al, \be,\ga \in \Phi^+$ whose coordinates are $(i,a)$, $(j,b)$, $(k,c)$ in an AR-quiver $\Gamma_Q$ satisfies $\al+\be=\ga.$ (See also \cite{KKKO14D}  and \cite{Oh14} for the analogous results related to type $A_n^{(2)}$ and $D_n^{(2)}$.)

\medskip

({\bf B})  Connections between denominator formulas and AR-quivers :   With new statistics introduced in \cite{Oh15E},
one can {\it read} the denominator formulas $d_{k,l}(z)$ for $U_q'(A^{(t)}_{n})$ (resp. $U_q'(D^{(t)}_{n})$) $(t=1,2)$ from {\it any} $\Gamma_Q$ of type $A_n$ (resp. $D_n$),
which {\it control} their representation theory (see Theorem \ref{Thm: basic properties}).
In detail, the first named author introduced the distance polynomial $D_{k,l}(z)$ of any $\Gamma_Q$ which directly follows from the statistics of $\Gamma_Q$.
Using distance polynomials, the denominator formula $d_{k,l}(z)$ between fundamental representations $V(\varpi_k)$
and $V(\varpi_l)$ over $U_q'(A_n^{(1)})$ (resp. $U_q'(D_n^{(1)})$) can be described as
$$ d_{k,l}(z) = D_{k,l}(z) \times (z-(-q)^{h^\vee})^{\delta_{k^*,l}} $$
where $h^\vee$ is the dual Coxeter number of $A_n$ (resp. $D_n)$.
Note that the denominator formulas between fundamental representations over classical quantum affine algebras were calculated in \cite{AK,DO94,KKK13b,Oh14}.

\medskip

In addition, another interesting application of AR-quivers to the representation theory of $U_q'(X^{(1)})$ is found by  Hernandez-Leclerc in \cite{HL11}. They introduced the category $\Ca_Q$ of $U_q'(X^{(1)})$-modules whose definition depends on the coordinate system of $\Gamma_Q$ and proved that each $\Ca_Q$ provides the categorification of negative part $U_q^-(X)$ of $U_q(X)$ and the dual PBW-basis of
$U_q^-(X)$ associated to the commutation class $[Q]$. In \cite{KKKO14D}, Kang-Kashiwara-Kim-Oh introduced the category $\Ca^{(2)}_Q$ of $U_q'(X_{n}^{(2)})$-modules ($X=A$ or $D$) and proved that
$\Ca^{(2)}_Q$ plays the same role of $\Ca^{(1)}_Q$ in the sense of categorification, by using Dorey's rules for $U_q'(X_{n}^{(2)})$ and AR-quiver $\Gamma_Q$ crucially.

\medskip

The main purpose of this paper is finding analogous results to ({\bf A}) and ({\bf B})  for type $BCFG$. In order to do this, we focus on the fact that  Chari-Pressley \cite{CP96} considered a twisted Coxeter element $\widehat{\phi}$ of type $A_{2n+1}$, $D_{n+1}$ to see Dorey's rule for $U_q'(B^{(1)}_n)$ and $U_q'(C^{(1)}_n)$.
On the other hand,  the authors \cite{OS15} defined combinatorial AR-quiver $\Upsilon_{[\ii_0]}$ for {\it any} commutation class $[\ii_0]$ of reduced expressions of $w_0$,
which is a generalization of $\Gamma_Q$. Indeed,  $\Upsilon_{[\ii_0]}$ reflects the convex partial order $\prec_{[\ii_0]}$ induced from the commutation class $[\ii_0]$.
Hence combinatorial AR-quivers are good options to substitute AR-quivers in ({\bf A}) and ({\bf B}). Indeed, our main results are the followings.

$\bullet$ For each (triply) twisted Coxeter element, we associate a commutation class $[\ii_0]$ of $w_0$ of types $A_{2n-1}$, $D_{n+1}$ and $E_6$ ($D_4$). Also the commutation classes arising from (triply) twisted Coxeter elements
are reflection equivalent.

$\bullet$ For each commutation class $[\ii_0]$ associated to a (triply) twisted Coxeter element, we {\it fold} the combinatorial AR-quiver $\Upsilon_{[\ii_0]}$ via Dynkin diagram automorphisms (see Section \ref{Sec:twisted Coxeter}). Then we get the folded AR-quiver  $\widehat{\Upsilon}_{[\ii_0]}$. Using $\widehat{\Upsilon}_{[\ii_0]}$ instead of $\Gamma_Q$, we can find analogous results to ({\bf A}) and ({\bf B}) for type $BC$. Thus
one can say that commutation classes associated to twisted Coxeter elements are deeply related to the representation theory of quantum affine algebras of type $B_n^{(1)}$ and $C_n^{(1)}$.

$\bullet$ By the same argument, we can find the conjectural formulas of Dorey's rule and denominator formulas for $U_q'(F^{(1)}_4)$ and $U_q'(G^{(1)}_2)$ (see the Appendix \ref{Appendix}). Indeed, there is no article written about Dorey's rule or denominator formulas for $U_q'(F^{(1)}_4)$ and $U_q'(G^{(1)}_2)$ for the best of authors' knowledge. We expect our conjectural formulas can give reasonable suggestions.

$\bullet$ For each commutation class $[\ii_0]$ associated to a (triply) twisted Coxeter element, we introduce new subcategories $\mathscr{C}_{[\ii_0]}$ for quantum affine algebras of untwisted non-simply laced types (such as
$B_n^{(1)}$, $C_n^{(1)}$, $F_4^{(1)}$ and $G_2^{(1)}$) which can be understood as twisted analogues of $\Ca_Q$ (see Definition \ref{def: [ii0] module category}).

\medskip

In order to achieve the goals, for a non-trivial Dynkin diagram automorphism of type $A_{2n+1}, D_{n+1}, E_6, D_4$, we study a class  $[\ii_0]$ of reduced expressions  associated to a (triply) twisted Coxeter element and the corresponding combinatorial AR-quiver $\Upsilon_{[\ii_0]}$. More precisely, we mainly discuss about reduced expressions in an $r$-cluster point $\lf \Qd \rf$ (resp. $\lf \mathfrak{Q} \rf$), called $($triply$)$ twisted adapted cluster point, which contains all reduced expressions related to (triply) twisted Coxeter elements. If we call the combinatorial AR-quiver  $\Upsilon_{[\ii_0]}$ for $[\ii_0]\in \lf \Qd\rf$ or $\lf \mathfrak{Q} \rf$ by a twisted AR-quiver, the crucial part of this paper deals with shapes and labeling of twisted AR-quivers. The followings are main steps to see the relations between twisted AR-quivers and denominator formulas or Dorey's rule:

$\bullet$ Twisted AR-quivers $\Upsilon_{[\ii_0]}$ of type $A_{2n+1}$ and $D_{n+1}$ can be obtained from AR-quivers $\Gamma_Q$ of type $A_{2n}$ and $A_n$ by simple surgeries (In Section \ref{Sec:Charac_twistedAR}, Algorithm \ref{Rem:surgery A} and Algorithm \ref{Alg surgery D}).

$\bullet$ $\Upsilon_{[\ii_0]}$ is foldable via the corresponding Dynkin diagram automorphism so that we can obtain folded AR-quiver $\widehat{\Upsilon}_{[\ii_0]}$ (Section \ref{Sec:coordinate of twisted AR}).

$\bullet$ $\Upsilon_{[\ii_0]}$ and $\widehat{\Upsilon}_{[\ii_0]}$ have natural coordinate systems (Section \ref{Sec:coordinate of twisted AR}).

$\bullet$ Labels in $\Upsilon_{[\ii_0]}$ and $\widehat{\Upsilon}_{[\ii_0]}$ of type $A_{2n+1}$ and $D_{n+1}$ are completely determined by the shape of quiver (Section \ref{Sec:label_twistedAR}).

As consequences, we can read the denominator formulas and Dorey's rule of type $U_q'(B^{(1)}_{n+1})$ and $U_q'(C^{(1)}_n)$ from any $\widehat{\Upsilon}_{[\ii_0]}$ of type $A_{2n+1}$ and $D_{n+1}$ by the algorithm for labeling $\widehat{\Upsilon}_{[\ii_0]}$. Also, Algorithm \ref{Rem:surgery A} and Algorithm \ref{Alg surgery D} explain the similarities of denominator formulas for classical untwisted quantum affine algebras (see \eqref{eq: interesting interpretation} and \eqref{eq: CDCD}).

\medskip

In Section \ref{comb_AR}, we introduce $r$-cluster points of (classes of) reduced expressions of $w_0$. In Section \ref{Sec:Adapted_AR}, we review results on reduced expressions in the adapted $r$-cluster point $\lf \Delta \rf$ and  AR-quivers focusing on properties which are useful in the applications ({\bf A}) and ({\bf B}). In Section \ref{Sec:twisted Coxeter}, we review twisted Coxeter elements and introduce the twisted adapted $r$-cluster point $\lf \Qd\rf$. For detailed properties of a reduced expression $\ii_0$ in  $\lf \Qd\rf$ and the combinatorial AR-quiver $\Upsilon_{[\ii_0]}$, we explain in Section \ref{Sec:Charac_twistedAR} to Section  \ref{Sec:label_twistedAR}.

In Section \ref{Sec:Charac_twistedAR}, we find the cardinalities of twisted adapted $r$-cluster points of type $A_{2n+1}$ and type $D_{n+1}$ and show twisted AR-quivers of type $A_{2n+1}$ and $D_{n+1}$ can be obtained from AR-quivers of type $A_{2n}$ and $A_n$ by simple surgeries. In Section \ref{Sec:coordinate of twisted AR},  we consider the folded AR-quiver $ \widehat{\Upsilon}_{[\ii_0]}$ of a twisted AR-quiver $\Upsilon_{[\ii_0]}$ and give coordinates to twisted AR-quivers and folded AR-quivers. In Section \ref{Sec:label_twistedAR}, we show how to find labels of twisted and folded AR-quivers by observing their shape only.

From Section \ref{sec: Quantum affine and application}, we focus on applications of twisted and folded AR-quivers to the representation theory of quantum affine algebras. In Section \ref{sec: Quantum affine and application}, we review basic notions in quantum affine algebras and their representation theories related to $R$-matrices, denominator formulas and Dorey's rule.  In Section \ref{Sec:app_AR}, we introduce terms on sequences of positive roots, including Distance polynomials and minimal pairs. Here we used properties of twisted and folded AR-quivers in Section \ref{Sec:Charac_twistedAR} to Section  \ref{Sec:label_twistedAR}. In Section \ref{Sec: Application}, we show  denominator formulas and Dorey's rule can be obtained by statistics of $\Upsilon_{[\ii_0]}$ stated in the previous sections.  In addition, we give conjectural formulas of Dorey's rule and denominators  for $U_q'(F_4^{(1)})$ and $U_q'(G_2^{(1)})$ in Appendix \ref{Appendix}.
In Appendix  \ref{Appendix:twisted Dynkin quiver}, we introduce a (triply) twisted Dynkin quiver and adapted reduced expressions to a (triply) twisted Dynkin quiver.  Here we show  all reduced expressions adapted
to a (triply) twisted  Dynkin quiver are in a commutation class in $\lf \Qd\rf$ (resp. $\lf \mathfrak{Q}\rf$).  This explains the motivation of the notion ``twisted adapted" reduced expressions.

\medskip

In \cite{KO17,OT18}, the first named author and his collaborators
proved that $\mathscr{C}_{[\ii_0]}$ gives a categorification of
$U_q^-(X)$ ($X=A_{2n+1}, D_{n+1}$, $E_6$ and $D_4$) and the dual PBW-basis of $U_q^-(X)$
associated to the (triply) twisted adapted class $[\ii_0]$ by using results in the previous versions of this paper (\cite{OS16A,OS16B}).
Hence we can observe mysterious categorical relations
between quantum affine algebras
\begin{itemize}
\item $U_q'(A^{(t)}_{2n+1})$ $(t=1,2)$ and $U_q'(B^{(1)}_{n+1})$,
\item $U_q'(D^{(t)}_{n+1})$ $(t=1,2)$ and $U_q'(C^{(1)}_{n})$,
\item $U_q'(E^{(t)}_{6})$ $(t=1,2)$ and $U_q'(F^{(1)}_{4})$,
\item $U_q'(D^{(t)}_{4})$ $(t=1,2,3)$, $U_q'(C^{(1)}_{3})$ and $U_q'(G^{(1)}_{2})$,
\end{itemize}
with the results in \cite{KKK13b,KKKO14D}  (see also \cite{HO18} in the aspect of quantum cluster algebras).
Note that such observation was initiated in \cite{FH11A}.
In \cite{OT18}, the conjectures in this paper on denominator formulas and Dorey's rule for $U'_q(F^{(1)}_{4})$ and $U'_q(G^{(1)}_{2})$ are also proved by the first author and Scrimshaw.

\subsection*{Acknowledgments}
We are thankful to the anonymous referee for making many useful comments.

\section{Reduced expressions and combinatorial AR-quivers} \label{comb_AR}

In this section, we review notions related to reduced expressions of the longest element in a Weyl group and introduce an $r$-cluster point, which consists of commutation classes of reduced expressions. Also, we recall the combinatorial AR-quiver associated to a class of reduced expressions, which is a generalization of the AR-quiver (Section \ref{Sec:Adapted_AR}) associated to a class of adapted reduced expressions.

\subsection{$r$-cluster points}
Let us consider the finite type Dynkin diagram $\Delta_n$ of rank $n$, labeled by the index set $I_n$. Let $W_n$ be a Weyl group,
generated by the set of simple reflections $\Pi_n:=\{ s_i \ | \ i \in I_n
\}$ and ${}_{n}w_0$ be the longest element
of $W_n$. We usually drop $n$ if there is no danger of confusion. Let us denote by $\PR$ the set of positive roots associated to $\Delta$ and denote $\N:= |\Phi^+|$.   

Let $\langle I \rangle$ be the free monoid generated by $I$ and $\langle I \rangle_{{\mathrm r}}$ be the set of reduced words of $\langle I \rangle$
in the sense of Weyl group representation. We usually denote by $\bi$ for words and by $\ii$ for the corresponding  reduced words.

We say that two reduced words $\ii$ and $\jj$ representing $w \in W$ are \defn{commutation equivalent}, denoted by $\ii
\sim \jj$, if $\ii$ is obtained from $\jj$ by applying the commutation relation $s_as_b=s_bs_a$, where $a$ and $b$ are non-adjacent vertices in $\Delta$. We denote by $[\ii]$ the commutation class of $\ii$ under the equivalence relation $\sim$.

Fix a Dynkin diagram $\Delta$ of finite type.
For a commutation class $[\ii_0]$ representing $w_0$, we say that an index $i$ is a \defn{sink $($resp. {\it source}$)$ of
$[\ii_0]$} if there is a reduced word $\jj_0 \in [\ii_0]$ of $w_0$ starting with $i$  (resp. ending with $i$).

\begin{definition} \label{Def:inv}
Let  $^*$ be the involution on $I$ induced by $w_0$, that is
$$w_0(\alpha_i)=-\alpha_{i^*}$$
for each simple root $\alpha_i$.
\end{definition}

The following proposition is a well-known to experts (see for example \cite{OS15}).

\begin{proposition}
 For the reduced word $\ii_0=i_1 i_2 \cdots i_{\N-1} i_\N$ of $w_0$, the word $\ii'_0= i_{\N}^* i_1 i_2 \cdots i_{\N-1}$
is again a reduced word of $w_0$ such that $[\ii'_0] \ne [\ii_0]$.
Similarly, $\ii''_0= i_2 \cdots i_{\N-1} i_\N i^*_1$ is again a
reduced word of $w_0$ such that $[\ii''_0] \ne [\ii_0]$.
\end{proposition}



By applying the above proposition, we can obtain new reduced words for $w_0$ by applying the operations so called
 \defn{reflection functors}, from a reduced word for $w_0$.
The right action of reflection functor $r_i$ on $[\ii_0]$ is defined by
$$[\ii_0]\, r_i =
\left\{
\begin{array}{ll}
[ i_2 \cdots i_\N i^*] & \text{ if $i$ is a sink and } \ii'_0=i_{ \ }   i_2 \cdots i_\N \in [\ii_0],\\
\ [\ii_0] &  \text{ if $i$ is not a sink of $[\ii_0]$}.
\end{array}
\right.
$$
Similarly, the left action of reflection functor $r_i$ on $[\ii_0]$ is defined by
$$
r_i \, [\ii_0]=
\left\{
\begin{array}{ll}
[ i^* i_1 \cdots i_{\N-1} ] &  \text{ if $i$ is a source and } \ii'_0= i_1 \cdots i_{\N-1} i \in [\ii_0],\\
\ [\ii_0] &  \text{ if $i$ is not a source of $[\ii_0]$}.
\end{array}
\right.
$$

For the word $\mathbf{w}= i_1 \cdots i_k$, the right  (resp. left) action of the reflection functor  $r_{\mathbf{w}}$ is defined by
$$ [\ii_0] \, r_{\mathbf{w}}=[\ii_0] r_{i_1} \cdots r_{i_k} \qquad (\text{resp. } r_{\mathbf{w}}[\ii_0]=r_{i_k}\cdots r_{i_1} [\ii_0]).$$


\begin{definition}
Let  $[\ii_0]$ and $[\ii'_0]$ be two commutation classes representing $w_0$. We say that $[\ii_0]$ and $[\ii'_0]$ are
\defn{reflection equivalent} and write $[\ii_0]\overset{r}{\sim} [\ii'_0]$ if $[\ii_0']$ can be obtained from $[\ii_0]$ by a
sequence of reflection functors. The family of commutation classes
$\lf \ii_0 \rf \seteq \{\, [\ii'_0]\, |\, [\ii'_0]\overset{r}{\sim}[\ii_0]\, \}$ is called an \defn{$r$-cluster point} (see \cite{OS15}).
\end{definition}

Note that for a reduced expression  $\ii_0= i_1 i_2\cdots i_\N$  and a word $\mathbf{w}=i_\N i_{\N-1}\cdots i_1 i_\N^* i_{\N-1}^* \cdots i^*_2$, we have
\[ [\ii_0] r_{i_1}=[i_2 i_3\cdots i_\N i^*_1] =r_{\mathbf{w}} [\ii_0].\]

\smallskip

In this paper, we deal with two types of  $r$-cluster points: adapted $r$-cluster points (Definition \ref{Def:adapted cluster}) and twisted adapted $r$-cluster points (Section \ref{Sec:twisted Coxeter}).  An adapted $r$-cluster point  consists  of adapted reduced expressions. This type of $r$-cluster points are well-investigated. On the other hand, twisted adapted $r$-cluster points  are newly investigated  in this paper
for the best of authors' knowledge.

\subsection{Convex orders and combinatorial AR-quivers}
Note that, for any reduced expression $\ii_0$ of $w_0$, we can define a total order $<_{\ii_0}$ on $\Phi^+$ as follows:
\begin{equation} \label{compatible reading}
 \be^{\ii_0}_k \seteq s_{i_1}s_{i_2} \cdots s_{i_{k-1}}(\alpha_{i_k}) \ (1 \le k \le \N) \quad \text{ and } \quad \be^{\ii_0}_k <_{\ii_0} \be^{\ii_0}_l \iff k <l.
\end{equation}
Interestingly, the order $<_{\ii_0}$ is convex in the following sense (see ~\cite{Papi94,Zel87}): 
We say that an order $\prec$ on $\Phi^+$ is \defn{convex} if it satisfies the following property: For $\al,\be \in \Phi^+$ satisfying $\al+\be \in \Phi^+$, we have either
\begin{align*}
\al \prec \al+\be \prec \be \quad\text{ or }\quad \be \prec \al+\be \prec \al.
\end{align*}

By considering  $<_{\ii'_0}$ for all $\ii_0' \in [\ii_0]$, the convex partial order $\prec_{[\ii_0]}$ on $\Phi^+$ is defined as follows:
\begin{align}
\al \prec_{[\ii_0]}  \be \quad \text{ if } \al <_{\ii_0'} \be \text{ for all } \ii_0' \in [\ii_0].
\end{align}

In \cite{OS15}, the authors introduced the \defn{combinatorial Auslander-Reiten quiver}  $\Upsilon_{[\ii_0]}$
for {\it any} commutation class $[\ii_0]$ of any finite type.

\begin{algorithm} \label{Alg_AbsAR}
The quiver $\Upsilon_{\ii_0}=(\Upsilon^0_{\ii_0},
\Upsilon^1_{\ii_0})$ associated to $\ii_0$ is constructed by the
following algorithm:
\begin{enumerate}
\item[{\rm (Q1)}] $\Upsilon_{\ii_0}^0$ consists of $\N$ vertices labeled by $\beta^{\ii_0}_1, \cdots, \beta^{\ii_0}_{\N}$ as in~\eqref{compatible reading}.
 \item[{\rm (Q2)}] There is an arrow from $\beta^{\ii_0}_k$ to $\beta^{\ii_0}_j$ for $1\leq j<k \leq \N$ if the followings hold:
\begin{enumerate}
\item[{\rm (Ar1)}] two indices $i_k$ and $i_j$ are distinct and connected in the Dynkin diagram,
\item[{\rm (Ar2)}] for $j'$ such that $j<j'<k$, we have $i_{j'}\neq i_j, i_k$.
\end{enumerate}
\item[{\rm (Q3)}] Assign the color $m_{jk}=-(\alpha_{i_j}, \alpha_{i_k})$ to each arrow $\beta^{\ii_0}_k\to \beta^{\ii_0}_j$ in {\rm (Q2)}; that is,
$\beta^{\ii_0}_k \xrightarrow{m_{jk}} \beta^{\ii_0}_j$.  Replace
$\xrightarrow{1}$ by $\rightarrow$,  $\xrightarrow{2}$ by
$\Rightarrow$ and  $\xrightarrow{3}$ by  $\Rrightarrow$.
\end{enumerate}
\end{algorithm}

For a given reduced expression $\ii_0=(i_1 i_2 i_3 \cdots i_{\N})$ of the longest element $w_0\in W$, the \defn{residue} of the vertex  $\beta^{\ii_0}_k\in \Upsilon^0_{\ii_0}$ is $i_k$.

\begin{theorem} \cite{OS15} \label{thm: OS14}
Let us choose any commutation class $[\ii_0]$.
\begin{enumerate}
\item[{\rm (1)}] If $\ii_0 \sim \ii_0'$, then $\Upsilon_{\ii_0} \simeq \Upsilon_{\ii'_0}$ as quivers. Hence $\Upsilon_{[\ii_0]}$ is well-defined.
\item[{\rm (2)}] $\al \prec_{[\ii_0]} \be$ if and only if there exists a path from $\be$ to $\al$ in $\Upsilon_{[\ii_0]}$.
\item[{\rm (3)}] Each $\ii'_0 \in [\ii_0]$ can be obtained by reading the residue of every vertex in a way compatible with the opposite directions of arrows.
\end{enumerate}
\end{theorem}

\begin{remark}
Conventionally, we assume residues of vertices  in $\Upsilon_{[\ii_0]}$ increase from the north to the south. Also, every arrow in a $\Upsilon_{[\ii_0]}$ points South-East or North-East direction.
\end{remark}

\begin{example} \label{D4 non-adapted D-1}
Let $\ii_0 = (123124123124)$ be a reduced word of $w_0$ of type $D_4$. We can draw $\Upsilon_{[\ii_0]}$ as follows:
$$\scalebox{0.84}
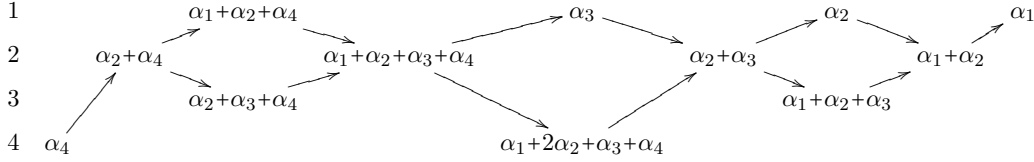
{\xymatrix@C=1ex@R=1ex{
1&&& \alpha_1 \shp \alpha_2 \shp \alpha_4\ar@{->}[dr] && \alpha_3 \ar@{->}[dr]&& \alpha_2 \ar@{->}[dr] && \alpha_1\\
2&& \alpha_2 \shp \alpha_4 \ar@{->}[ur]\ar@{->}[dr] && \alpha_1 \shp
\alpha_2 \shp \alpha_3 \shp \alpha_4 \ar@{->}[ur]\ar@{->}[ddr]
 && \alpha_2 \shp \alpha_3 \ar@{->}[ur]\ar@{->}[dr]&& \alpha_1 \shp \alpha_2 \ar@{->}[ur]\\
3& && \alpha_2 \shp \alpha_3 \shp \alpha_4 \ar@{->}[ur]&& && \alpha_1 \shp \alpha_2 \shp \alpha_3 \ar@{->}[ur] \\
4& \alpha_4 \ar@{->}[uur] &&  && \alpha_1 \shp 2\alpha_2 \shp
\alpha_3 \shp \alpha_4 \ar@{->}[uur]&& }}
$$
\end{example}

\begin{remark} \label{Rem:AR vs comb AR}
As one can expect, a combinatorial AR-quiver is a generalization of an AR-quiver, which will be reviewed carefully in Section \ref{Sec:Adapted_AR}. More precisely, if $[\ii_0]$ is a class of adapted reduced expressions then the corresponding $\Upsilon_{[\ii_0]}$ is the same as the corresponding AR-quiver $\Gamma_Q$ (see Section \ref{Sec:Adapted_AR}).
\end{remark}

\begin{definition} \cite[Definition 1.6]{Oh14A}  \label{Def:sectional} Fix any class $[\jj_0]$ of $w_0$ of any finite type.
\begin{enumerate}[(i)]
\item A  pair $(\al,\be)$ of positive roots is \defn{sectional} in $\Upsilon_{[\jj_0]}$ if there exists a path in $\Upsilon_{[\jj_0]}$ between them
consisting of $d(i,j)$-arrows, where $i$ is the residue of $\al$, $j$ is the residue of $\be$ and $d(i,j)$ denotes the number of edges between $i$ and $j$ in Dynkin diagram.
\item A full subquiver $\rho$ of $\Upsilon_{[\jj_0]}$ is \defn{sectional} if every pair $(\al,\be)$ in $\rho$ is
sectional.
\item A connected subquiver $\rho$ in $\Upsilon_{[\jj_0]}$ is called an \defn{$S$-sectional} (resp. \defn{$N$-sectional}) path if it is a concatenation of downward (upward) arrows,
and there is no longer connected path consisting of downward arrows (resp. upward arrows) containing $\rho$.
\end{enumerate}
\end{definition}

We write $N$-path (resp. $S$-path) instead of $N$-sectional path (resp. $S$-sectional path) for brevity.

\section{Adapted reduced expressions and  AR-quivers} \label{Sec:Adapted_AR}

 The Auslander-Reiten (AR) theory, which is closely related to adapted reduced expressions, have been studied well. In this section, we briefly review properties of adapted reduced expressions, AR-quivers and their applications. For the precise references, we mainly refer \cite{ARS,Gab80,Oh14A,Oh15E}

In this section, we assume the set $\Phi^+$ of positive roots  and Dynkin diagram $\Delta$ are of type $X_n$, $X=A,D,E$.

\subsection{Coxeter elements and the adapted $r$-cluster point $\lf \Delta \rf$}\label{subsec:Coxeter} Let us consider a Dynkin quiver $Q$, which is obtained by orienting all edges of the Dynkin diagram $\Delta$.
For a sink $i$ of $Q$, we denote by $i Q$ the quiver obtained from $Q$ by reversing the orientation of each arrow incidents with $i$ in $Q$.
For a reduced word $\ii$, we say that $\ii=i_1 i_2 \cdots i_l$ is \defn{adapted to} $Q$ if
\begin{align*}
\text{ $i_k$ is a sink of $ i_{k-1} \cdots i_2 i_1 \ Q$ for all $1 \le k \le l$.}
\end{align*}

Also, recall that a \defn{Coxeter element } $\phi=s_{i_1}s_{i_2}\cdots s_{i_n}$ is a product of simple reflections, where $i_1 i_2 i_3\cdots i_n$ is a rearrangement of $1\,2\, 3\,\cdots n.$

The following theorem shows that Dynkin quivers and Coxeter elements are closely related to classes of reduced expressions.

\begin{theorem} \label{Thm:adapted_corres}\
\begin{enumerate}
\item[{\rm (1)}]  The set of all reduced expressions of $w_0$ adapted to $Q$ forms a commutation class $[Q]$.
\item[{\rm (2)}] If $Q \neq Q'$ then we have $[Q] \neq [Q']$.
\item[{\rm (3)}] The set of commutation classes $[Q]$ forms an $r$-cluster point $\lf \Delta \rf$.
\item[{\rm (4)}] For a Dynkin quiver $Q$, there exists unique adapted Coxeter element denoted by $\phi_Q$. Conversely, for each Coxeter element $\phi$, there exists unique Dynkin quiver $Q$ such that $\phi=\phi_Q.$
\end{enumerate}
\end{theorem}

By the previous theorem, there are one-to-one and onto correspondences between the set of Dynkin quivers, adapted classes, and Coxeter elements by the maps
\begin{equation}\label{eq:correspondence}
Q \longleftrightarrow [Q] \longleftrightarrow \phi_Q.
\end{equation}

Also, by Theorem \ref{Thm:adapted_corres} (3), we get the following definition.

\begin{definition} \label{Def:adapted cluster}
The $r$-cluster point $\lf \Delta \rf$ consisting of all commutation classes $[Q]$ is called the \defn{adapted $r$-cluster point.}
\end{definition}

By counting the number of Dynkin quivers, we have
$$
\text{the number of commutation classes in $\lf \Delta \rf$ is $2^{n-1}$.}
$$

\subsection{AR-quivers and their properties} \label{Subsec:AR-properties}
For each Dynkin quiver $Q$ of type $ADE$, let us denote by $\Gamma_Q$ the AR-quiver which is associated to $Q$ and
whose vertices are labeled by $\Phi^+$. In this subsection, we focus on its relations with $\prec_{[Q]}$ and $[Q]$.
Recall the Coxeter element $\phi_Q$ of $[Q]$. Then $\Gamma_Q$ can be constructed by using only $\phi_Q$:

(A)  For any reduced expression $s_{i_1}s_{i_2}\cdots s_{i_n}$ of $\phi_Q$, the subset
$$  \Phi(\phi_Q) \seteq \{ \be^{\phi_Q}_1=\al_{i_1}, \be^{\phi_Q}_2=s_{i_1}(\al_{i_2}),\ldots,\be^{\phi_Q}_n =s_{i_1}\cdots s_{i_{n-1}}(\al_{i_n}) \}.$$
of  $\Phi^+$ is well-defined.
Furthermore, the index $i_k$ on $\be^{\phi_Q}_k$ is also well-assigned.

(B) The \defn{height function} $\xi$ associated to $Q$ is a map on $Q$ satisfying $\xi(j)=\xi(i)+1$ if there exists an arrow $i \to j$ in $Q$. Note that the connectedness of
$Q$ implies the uniqueness of $\xi$ up to constant.

\medskip

Thus we can assign $\be^{\phi_Q}_k$ ($k=1, \cdots, n$)  to $(i_k,\xi(i_k)) \in I \times \Z$ which does depend only on $Q$ and hence $\phi_Q$:

\begin{algorithm} \label{Alg:AR}
The AR-quiver $\Gamma_Q$, whose set of vertices is also identified with a subset of  $I \times \Z$,
can be constructed by the following injective map $\Omega_Q:\PR \to I \times \Z$ in an inductive way
$($cf. \cite{HL11}$):$
\begin{enumerate}
\item[{\rm (1)}] $\Omega_Q(\be^{\phi_Q}_k) \seteq (i_k,\xi(i_k))$ for $k=1, \cdots, n$.
\item[{\rm (2)}] If $\Omega_Q(\be)$ is already assigned as $(i,p)$ and $\phi_Q(\be) \in \Phi^+$, then $$\Omega_Q(\phi_Q(\be))=(i,p-2).$$
\item[{\rm (3)}] For $(i,p),(j,q) \in {\rm Im}(\Omega_Q)$, there exists an arrow  $(i,p) \to (j,q)$ if and only if $i$ and $j$ are adjacent in $\Delta$ and
$$p-q=-1.$$
\end{enumerate}
\end{algorithm}
For $\be$ with $\Omega_Q(\be)=(i,p)$, we call $i$ the \defn{residue} of $\beta$ with respect to $[Q]$, and $(i,p)$ the  \defn{coordinate of $\be$} in $\Gamma_Q$.

\begin{proposition} \cite{OS15}
For each $[Q]$, we have
$$ \Gamma_Q \simeq \Upsilon_{[Q]} \quad \text{ as quivers}.$$
\end{proposition}

The AR-quiver $\Gamma_Q$ satisfies the \defn{additive property} with respect to arrows and $\phi_Q$ in the following sense:
For $\al \in \PR$ with $\be=\phi_Q(\al) \in \PR$, we have
\begin{align} \label{eq: additive}
 \al + \be = \displaystyle\sum_{\eta \in  {}_{\be} Q_\al } \eta,
\end{align}
where
\begin{align} \label{eq: al be Q}
{}_{\be} Q_\al =\{ \eta \in \PR \ | \ \text{there exists a path $\be\to \eta \to \al$  in $\Gamma_Q$} \}.
\end{align}


\begin{example}
The AR-quiver $\Gamma_Q$ associated to $\xymatrix@R=3ex{ *{ \circ }<3pt> \ar@{<-}[r]_<{1}  &*{\circ}<3pt>
\ar@{->}[r]_<{2}  &*{\circ}<3pt>
\ar@{<-}[r]_<{3} &*{\circ}<3pt>
\ar@{<-}[r]_<{4}  & *{\circ}<3pt> \ar@{-}[l]^<{\ \ 5}
}$ with the height function $\xi$ such that $\xi(1)=0$ is given as follows:
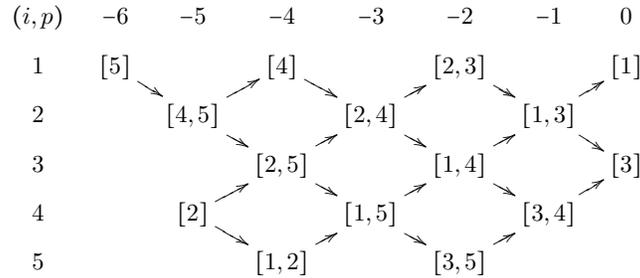
\[ \scalebox{0.84}{\ \xymatrix@C=2ex@R=1ex{
( i,p ) &-6&-5&-4&-3&-2&-1&0\\
1& [5]\ar@{->}[dr] & & [4]\ar@{->}[dr] & &[2,3] \ar@{->}[dr] & &[1]  \\
2&&[4,5]\ar@{->}[ur]\ar@{->}[dr]& &[2,4] \ar@{->}[ur]\ar@{->}[dr]&& [1,3] \ar@{->}[ur]\ar@{->}[dr]\\
3&&& [2,5]\ar@{->}[ur]\ar@{->}[dr] && [1,4]\ar@{->}[ur]\ar@{->}[dr] && [3]\\
4&&  [2] \ar@{->}[ur]\ar@{->}[dr]  && [1,5] \ar@{->}[ur] \ar@{->}[dr] && [3,4] \ar@{->}[ur]\\
5&&&[1,2] \ar@{->}[ur] &&[3,5] \ar@{->}[ur]}}
\]
Here $[a,b]$ $(1 \le a,b \le 5)$ denotes the positive root $\sum_{k=a}^b \al_k$ and $[a]:= \al_a$.
\end{example}

When we want to find $\Gamma_Q$ without labeling, Proposition \ref{Prop:AR} can be an easier alternative method to Algorithm \ref{Alg:AR}.

\begin{proposition} \cite{B99, Gab80, R80} \label{Prop:AR}
For the dual Coxeter number $\mathsf{h}^\vee$ associated to $Q$ and $i\in I$, let
\[ r^Q_i= \frac{\mathsf{h}^\vee+a^Q_i-b^Q_i}{2}\]
where $a^Q_i$ is the number of arrows in $Q$ between $i$ and $i^*$ directed toward $i$, and $b^Q_i$ is the number of arrows in $Q$ between $i$ and $i^*$
directed toward $i^*$. Then the number of vertices in $\Gamma_Q \cap \left( \{i\} \times \Z \right)$ is $r^Q_i$ and
\[ \Gamma_Q \cap \left( \{i\} \times \Z \right)= \{ \, (i, \xi(i)-2k)\, | \, k=0, \cdots, r_i-1\}.\]
\end{proposition}

\begin{remark} \label{rem: boundary}
It is known that
\begin{enumerate}
\item[{\rm (i)}] the right boundary of $\Gamma_Q$, the full subquiver consisting of $\Phi(\phi_Q)$, is isomorphic to
$Q$ as a quiver,
\item[{\rm (ii)}] the left boundary of $\Gamma_Q$,  the full subquiver consisting of $\{ \phi_Q^{r^Q_i-1}(\be^{\phi_Q}_i) \}$, is isomorphic to
$Q^*$ as a quiver, where $Q^*$ is the quiver obtained from $Q$ by replacing each vertex $i$ with $i^*$.
\end{enumerate}
\end{remark}

Now, the following theorem shows a combinatorial AR-quiver can be understood as a generalization of an AR-quiver (see Remark \ref{Rem:AR vs comb AR}), in the sense that $\Gamma_Q$ is a visualization of $\prec_{Q} \seteq \prec_{[Q]}$.

\begin{theorem} \cite{B99,R96} \hfill
\begin{enumerate}
\item[{\rm (1)}] $\al \prec_Q \be$ if and only if there exists a path from $\be$ to $\al$ inside of $\Gamma_Q$.
\item[{\rm (2)}] By reading residues of vertices in a way {\it compatible with}  opposite directions of  arrows in $\Gamma_Q$, we can obtain all reduced words $\ii_0 \in [Q]$.
\end{enumerate}
\end{theorem}

With the above theorem, we can extend the correspondences in \eqref{eq:correspondence} to the one-to-one and onto correspondences between the set of Dynkin quivers, adapted commutation classes, Coxeter elements, AR-quivers and associated convex orders:
\begin{equation}\label{eq: 2n-1many 2}
Q \longleftrightarrow [Q] \longleftrightarrow \phi_Q \longleftrightarrow \Gamma_Q \longleftrightarrow \prec_Q.
\end{equation}

On the other hand, relations between adapted commutation classes can be explained by reflection functors.
The reflection functor $r_i: [Q] \mapsto [Q]r_i$ for a sink $i$ of $[Q]$ can be understood by the map from $\Gamma_Q$ to   $\Gamma_{[Q]r_i}$. The functor can be described
using coordinates and the dual Coxeter number $\mathsf{h}^\vee$ as follows:
\begin{algorithm} \label{alg: Ref Q}
Let $\mathsf{h}^\vee$ be the dual Coxeter number associated to $Q$ and  $\al_i$ $(i \in I)$ be a sink of $\Gamma_{Q}$.
\begin{enumerate}
\item[{\rm (A1)}] Remove the vertex $(i,p)$ such that $\Omega_Q(\al_i)=(i,p)$ and the arrows adjacent to $(i,p)$.
\item[{\rm (A2)}] Add the vertex $(i^*,p-\mathsf{h}^\vee)$ and the arrows to all $(j,p-\mathsf{h}^\vee+1) \in \Gamma_Q$ for $j$ adjacent to $i^*$ in $\Delta$.
\item[{\rm (A3)}] Label the vertex $(i^*,p-\mathsf{h}^\vee)$ with $\al_i$ and change the labels $\be$ to $s_i(\be)$ for all $\be \in \Gamma_Q \setminus \{\al_i\}$.
\end{enumerate}
\end{algorithm}

\subsection{Labeling of  (combinatorial) AR-quivers of type $A_n$}
In this subsection, we briefly review results in \cite{Oh14A} and \cite{OS15} regarding labels of (combinatorial) AR-quivers.

Recall that, for every $1 \le a \le b \le n$,
$\beta = \sum_{a \le k \le b} \alpha_k$ is a positive root in $\Phi^+_{A_n}$ and every positive root in $\Phi^+_{A_n}$ is of the form.
Thus we frequently identify $\beta \in \Phi^+$ (and hence a vertex in $\Upsilon_{[\jj_0]}$) with the segment $[a,b]$.
For $\beta=[a,b]$, we say $a$ is the \defn{first component} of $\beta$ and $b$ is the \defn{second component} of $\beta$. If $\beta$ is simple, we write $\beta=[a]$.

\begin{proposition} \label{pro: section shares} \cite[Proposition 4.5]{OS15}
Fix any class $[\jj_0]$ of $w_0$ of type $A_n$.
Let $\rho$ be an $N$-path $($resp. $S$-path$)$ in $\Upsilon_{[\jj_0]}$. Then every positive root contained in
$\rho$ has the same first $($resp. second$)$ component.
\end{proposition}

\begin{theorem} \label{thm: labeling GammaQ}
\cite[Corollary 1.12]{Oh14A} Fix any Dynkin quiver $Q$ of type $A_n$.
For $1 \le i \le n$, the AR-quiver $\Gamma_Q$ contains an  $N$-path with $(n-i)$-arrows  exactly once whose vertices share $i$ as the first component.
At the same time, $\Gamma_Q$ contains an $S$-path with $(i-1)$-arrows  exactly once whose vertices share $i$ as the second component.
\end{theorem}


\medskip

With the above theorem, we can label the vertices of $\Gamma_Q$ without computation like \eqref{compatible reading}.

\section{Twisted Coxeter elements and twisted adapted $r$-cluster point(s)} \label{Sec:twisted Coxeter}

From Section \ref{Sec:twisted Coxeter} to Section \ref{Sec:label_twistedAR}, we shall introduce new $r$-cluster points of type $A_{2n+1}$, $D_{n+1}$, $E_6$, $D_4$, called the $($triply$)$ twisted adapted $r$-cluster point(s) and show properties of classes in the (triply) twisted $r$-cluster point(s). In particular, in this section, we define the (triply) twisted adapted $r$-cluster point(s) using the following Dynkin diagram automorphisms $\vee$, which yield Dynkin diagrams of non-simply laced types$\colon$
\begin{subequations}
\begin{align}
\big( A_{2n+1}: \xymatrix@R=0.5ex@C=4ex{ *{\circ}<3pt> \ar@{-}[r]_<{1 \ \ }  &*{\circ}<3pt>
\ar@{-}[r]_<{2 \ \ }  &   {}
\ar@{.}[r] & *{\circ}<3pt> \ar@{-}[r]_>{\,\,\,\ 2n} &*{\circ}<3pt>\ar@{-}[r]_>{\,\,\,\, 2n+1} &*{\circ}<3pt> }, \ i^{\vee} = 2n+2-i \big) &\longleftrightarrow B_{n+1} \label{eq: B_n} \\
   \left( D_{n+1}:
\raisebox{1em}{\xymatrix@R=0.5ex@C=4ex{
& & &  *{\circ}<3pt>\ar@{-}[dl]^<{ \  n} \\
*{\circ}<3pt> \ar@{-}[r]_<{1 \ \ }  &*{\circ}<3pt>
\ar@{.}[r]_<{2 \ \ } & *{\circ}<3pt> \ar@{.}[l]^<{ \ \ n-1}  \\
& & &   *{\circ}<3pt>\ar@{-}[ul]^<{\quad \ \  n+1}   \\
}}, \ i^{\vee} = \begin{cases} i & \text{ if } i \le n-1, \\ n+1 & \text{ if } i = n, \\ n & \text{ if } i = n+1. \end{cases} \right)  &\longleftrightarrow  C_n \label{eq: C_n} \\
  \left( E_{6}:
\raisebox{2em}{\xymatrix@R=3ex@C=4ex{
& & *{\circ}<3pt>\ar@{-}[d]_<{\quad \ \  6} \\
*{\circ}<3pt> \ar@{-}[r]_<{1 \ \ }  &
*{\circ}<3pt> \ar@{-}[r]_<{2 \ \ }  &
*{\circ}<3pt> \ar@{-}[r]_<{3 \ \ }  &
*{\circ}<3pt> \ar@{-}[r]_<{4 \ \ }  &
*{\circ}<3pt> \ar@{-}[l]^<{ \ \ 5 } }}, \begin{cases} 1^{\vee}=5, \ 5^{\vee}=1, \\ 2^{\vee}=4, \ 4^{\vee}=2, \\ 3^{\vee}=3, \ 6^{\vee}=6. \end{cases}  \right) & \longleftrightarrow F_4 \label{eq: F_4} \\
   \left( D_{4}:\raisebox{1em}{
\xymatrix@R=0.5ex@C=4ex{
& &   *{\circ}<3pt>\ar@{-}[dl]^<{ \ 3} \\
*{\circ}<3pt> \ar@{-}[r]_<{1 \ \ }  &*{\circ}<3pt>
\ar@{-}[l]^<{2 \ \ }   \\
& &    *{\circ}<3pt>\ar@{-}[ul]^<{\quad \ \  4} \\
}}, \ \begin{cases} 1^{\vee}=3, \ 3^{\vee}=4, \
4^{\vee}=1, \\ 2^{\vee}=2. \end{cases} \right)  &\longleftrightarrow G_2 \label{eq:
G_2}
\end{align}
\end{subequations}


\subsection{Twisted Coxeter elements} As we showed in Section \ref{Sec:Adapted_AR}, a Coxeter element has  information about the corresponding adapted commutation class. In a similar sense, a twisted Coxeter element  introduced in this subsection gives rise to a commutation class, which will be called a twisted adapted class.  Note that  a twisted Coxeter element is closely related to a  twisted Dynkin quiver defined in Appendix \ref{Appendix:twisted Dynkin quiver} via the correspondence \eqref{Eqn:twisted DQuiver-Coxeter}.

\begin{remark}
We can view the Weyl group $W$ as a
subgroup of $GL(\C\Phi)$ generated by the set of simple reflections $\{ s_i \ | \ i \in I \}$.
In this case, we use the term ``reduced expression" instead of ``reduced word". Moreover, we sometimes abuse the notation
$\ii$ to represent reduced expressions.
\end{remark}

Let $\sigma \in GL(\C\Phi)$ be a linear transformation of finite order
which preserves a base $\Pi$ of $\Phi$. Then $\sigma$ preserves
$\Phi$ itself and normalizes $W$ and so $W$ acts by conjugation on the
coset $W\sigma$. 


\begin{definition} \label{def: twisted Coxeter} \hfill
\begin{enumerate}
\item Let $\{ \Pi_{1}, \ldots,\Pi_{k} \}$ be all orbits of $\Pi$ in
$\Phi$ with respect to $\sigma$. For each $r\in \{1, \cdots, k \}$, choose  $\alpha_{i_r} \in \Pi_{r}$ arbitrarily,
and let $s_{i_r} \in W$ denote the corresponding reflection. Let $w$
be the product of $s_{i_1}, \ldots , s_{i_k}$ in any order. The
element $w\sigma \in W\sigma$ thus obtained is called a \defn{$\sigma$-Coxeter element}.
\item For $\vee$ in \eqref{eq: B_n}, \eqref{eq: C_n}, \eqref{eq: F_4},  $\vee$-Coxeter element is called a \defn{twisted Coxeter element}.
\item For $\vee$ or $\vee^2$ in \eqref{eq: G_2}, $\vee$-Coxeter element is called a \defn{triply twisted Coxeter element}.
\end{enumerate}
\end{definition}


\begin{example}
Take $\vee$ in \eqref{eq: B_n} for $A_{5}$. There are
$12$ distinct twisted Coxeter elements:
\begin{align*}
& s_1s_2s_3\vee, \ s_2s_1s_3\vee, \ s_3s_1s_2\vee, \ s_3s_2s_1\vee, \ s_5s_2s_3\vee, \ s_3s_2s_5\vee, \\
& s_1s_4s_3 \vee, \ s_3s_1s_4\vee, \ s_5s_4s_3\vee, \ s_4s_5s_3\vee,
\ s_3s_5s_4\vee, \ s_3s_4s_5\vee.
\end{align*}
\end{example}

\begin{example}
Take $\vee$ in \eqref{eq: C_n} for $D_{4}$:
There are $8$ distinct twisted Coxeter elements:
\begin{align*}
& s_1s_2s_3\vee, \ s_1s_3s_2\vee, \ s_2s_1s_3\vee, \ s_3s_2s_1\vee, \ s_1s_2s_4\vee, \ s_1s_4s_2\vee, \ s_2s_1s_4\vee, \ s_4s_2s_1\vee.
\end{align*}
\end{example}

\begin{remark}
If there is no danger of confusion, we simply denote by $i_1 i_2 \cdots i_k \vee$ the twisted Coxeter element $s_{i_1} s_{i_2}\cdots s_{i_k} \vee$.
\end{remark}

\begin{proposition}\label{prop: number tCox elts} \hfill
\begin{enumerate}
\item[{\rm (1)}] The number of twisted Coxeter elements of type $A_{2n+1}$ associated to \eqref{eq: B_n} is $4\times 3^{n-1}.$
\item[{\rm (2)}] The number of twisted Coxeter elements of type $D_{n+1}$ associated to \eqref{eq: C_n} is $2^n.$
\item[{\rm (3)}] The number of twisted Coxeter elements of type $E_6$ associated to \eqref{eq: F_4} is $24.$
\item[{\rm (4)}] The number of triply twisted Coxeter elements of type $D_4$  associated to \eqref{eq: G_2} is $12$.
\end{enumerate}
\end{proposition}

\begin{proof}
\begin{enumerate}
\item Suppose the number of twisted Coxeter elements of type $A_{2n+1}$ is $4\times 3^{n-1}.$ Then it is enough to show that a twisted Coxeter element of type $A_{2n+1}$ induces three distinct twisted Coxeter elements of type $A_{2n+3}$. Take a twisted Coxeter element $s_{i_1}s_{i_2}\cdots s_{i_{n+1}} \vee$ of type $A_{2n+1}$. If $1 \in \{i_1, i_2, \cdots, i_{n+1}\}$ then it induces three twisted Coxeter elements of type $A_{2n+3}$:
\begin{equation} \label{coxeter->twisted coxeter}
 (2n+3) (i_1+1) (i_2+1)\cdots (i_{n+1}+1) \vee,\ 1  (i_1+1) (i_2+1)\cdots (i_{n+1}+1) \vee,  \ (i_1+1)(i_2+1)\cdots (i_{n+1}+1) 1 \vee.
 \end{equation}
Note that, since $2n+1 \nin  \{i_1, i_2, \cdots, i_{n+1}\}$, $(2n+3)$ commutes with $(i_1+1), (i_2+1), \cdots, (i_{n+1}+1)$. Hence any twisted Coxeter element of the form  $(i_1+1) (i_2+1)\cdots (i_k+1) (2n+3)(i_{k+1}+1)(i_{n+1}+1)\vee$ is the same as the first  twisted Coxeter element in \eqref{coxeter->twisted coxeter}.  On the other hand, observe that there is $i_{k'}$, $k'=1,2, \cdots, n+1$, such that $i_{k'}=1$. Hence any twisted Coxeter element of the form $(i_1+1) (i_2+1)\cdots (i_k+1) (1) (i_{k+1}+1)(i_{n+1}+1)\vee$ for $k<k'$ is the same as the second twisted Coxeter element in \eqref{coxeter->twisted coxeter} and any twisted Coxeter element of the form $(i_1+1) (i_2+1)\cdots (i_k+1) (1) (i_{k+1}+1)(i_{n+1}+1)\vee$ for $k\geq k'$ is the same as the third twisted Coxeter element in \eqref{coxeter->twisted coxeter}.

Otherwise,  $2n+1 \in \{i_1, i_2, \cdots, i_{n+
1}\}$ and it induces
\[1 (i_1+1)(i_2+1)\cdots (i_{n+1}+1) \vee,\, (2n+3) (i_1+1)(i_2+1)\cdots (i_{n+1}+1) \vee,  \, (i_1+1)(i_2+1)\cdots (i_{n+1}+1)(2n+3) \vee.  \]
\item Note that a twisted Coxeter element $i_1 i_2 \cdots i_n \vee$  of type $D_{n+1}$ has $k\in \{1,2, \cdots, n\}$ such that  $i_k=n\text{ or } n+1$. From the twisted Coxeter element, we get the Coxeter element of type $A_n$
\[ s_{i_1} s_{i_2} \cdots s_n \cdots s_{i_n},\]
replacing $s_{i_k}$ by $s_n$.
Conversely, a Coxeter element $s_{i_1} s_{i_2} \cdots s_{i_n}$  of type $A_n$ with $s_{i_k}= s_n$ induces two distinct twisted Coxeter elements of type $D_{n+1}$
\[ i_1 i_2 \cdots i_k \cdots i_n \vee, \quad  i_1 i_2 \cdots (i_k+1) \cdots i_n \vee.\]
Since we know  the number of Coxeter elements of type $A_n$ is the same as the number of Dynkin quiver, which is $2^{n-1},$ we proved (2).
\end{enumerate}
The remaining cases can be checked directly.
\end{proof}

\subsection{(Triply) Twisted adapted $r$-cluster point $\lf \Qd \rf$ (resp. $\lf \mathfrak{Q}\rf$)} \label{Subsec:twisted adapted cluster}
Now, we introduce a special $r$-cluster point, called the $($triply$)$ twisted $r$-cluster point, associated to a particular (triply) twisted Coxeter element and $\vee$.

Note that, for each word $\bj=j_1 \cdots j_s$ in $\left< I \right>$ and $k \in \Z_{\ge 0}$, we denote
\begin{align} \label{eq: vee def}
& (j_1 \cdots j_s)^\vee \seteq j^\vee_1 \cdots j^\vee_s \text{ and }
(j_1 \cdots j_s)^{k \vee} \seteq  ( \cdots ((j_1 \cdots j_s \underbrace{ )^\vee )^\vee \cdots )^\vee}_{ \text{ $k$-times} }.
\end{align}

\noindent {\bf $\bullet$ A reduced expression associated to a twisted Coxeter element, type $A_{2n+1}$ case}:
Let us fix  the twisted Coxeter element $1\, 2\, \cdots\, n+1 \vee$ of type $A_{2n+1}$ and
consider the related word $\bi^{\natural}_0$ of $W$ of type $A_{2n+1}$:
\begin{align} \label{eq: can word A}
\bi^{\natural}_0 = \prod_{k=0}^{2n} (1\ 2\ 3\cdots (n+1))^{k\vee} \qquad \text{for $\vee$ in \eqref{eq: B_n}.}
\end{align}
\noindent
Note that the expression in~\eqref{eq: can word A} does not correspond to the one-line notation of symmetric group but the word in $\langle I \rangle$.

\begin{proposition} \label{prop: ii_0 red A}
The word $\bi^{\natural}_0$ in \eqref{eq: can word A} is a reduced expression of $w_0$ of type $A_{2n+1}$ which is not adapted to any Dynkin quiver. Hence
it can be denoted by $\ii^{\natural}_0$ instead of $\bi^{\natural}_0$.
\end{proposition}

\begin{proof}
Let us recall that the Weyl group of type $A_{2n+1}$ is the symmetric group $\mathfrak{S}_{2n+2}$ and $w_0$
satisfies $w_0(k)= 2n+3-k$ for $k=1, \cdots, 2n+2$.

Denote  $\ut=s_1s_2\cdots s_{n}s_{n+1}s_{2n+1}s_{2n}\cdots s_{n+1}$ and
$\htau=s_1s_2\cdots s_{n}s_{n+1}$. Then one can  check that
\begin{align*}
& \htau = \left( \begin{matrix} 1& 2 & \ldots & n+1 &n+2& n+3 & \ldots & 2n+1 & 2n+2 \\ 2 & 3 & \ldots & n+2 & 1 & n+3 & \ldots & 2n+1 & 2n+2  \end{matrix} \right), \\
& \ut = \left( \begin{matrix} 1& 2 & \ldots & n+1 & n+2 & n+3 & \ldots & 2n+1 & 2n+2 \\ 2 & 3 & \ldots & 2n+2 & n+2 & 1 & \ldots & 2n & 2n+1  \end{matrix} \right),
\end{align*}
with the two-line notation of symmetric group. Hence
$$ \bi^{\natural}_0= \ut^n\htau = \left( \begin{matrix} 1& 2 & \ldots & 2n+1 & 2n+2 \\ 2n+2 & 2n+1 & \ldots & 2 & 1   \end{matrix} \right),$$
which is the same as $w_0$. Note that $\bi^{\natural}_0$ is reduced since the length of $\bi^{\natural}_0$
is the same as $|\PR|$.
\end{proof}

\noindent {\bf $\bullet$ A reduced expression associated to a twisted Coxeter element, type $D_{n+1}$ case}:
As in the $A_{2n+1}$ case, let us consider the twisted Coxeter element $1\, 2\, \cdots \, n \vee$ and the related word $\bi^{\natural}_0$ of $W$ of type $D_{n+1}$:
\begin{align} \label{eq: can word D}
\bi^{\natural}_0=\prod_{k=0}^{n} (1\ 2\ \cdots \ n)^{k\vee} \qquad \text{for $\vee$ in \eqref{eq: C_n}.}
\end{align}

\begin{proposition}
The word $\bi^{\natural}_0$ in \eqref{eq: can word D} is a reduced expression of $w_0$ of type $D_{n+1}$ which is not adapted to any Dynkin quiver. Hence
it can be denoted by $\ii^{\natural}_0$ instead of $\bi^{\natural}_0$.
\end{proposition}

\begin{proof} Recall that
\[\Phi^+_{D_{n+1}}  =\{ \varepsilon_{a_1}-\varepsilon_{a_2}, \ \varepsilon_{b_1}+\varepsilon_{b_2} \, | \, 1\leq a_1 < a_2 \leq n+1 ,  \ 1\leq b_1  < b_2 \leq n+1 \}.\]
We denote the positive roots by
\begin{equation} \label{Eqn:root_D}
\lan a_1, -a_2 \ran, \ \lan b_1, b_2 \ran,
\end{equation}
 respectively.
By defining
\[ \beta_{p,q}^{\bi^{\natural}_0}= \prod_{k=0}^{p-2}(s_1\ s_2\ \cdots \ s_n)^{k\vee} (s_1 s_2 \cdots s_{q-1})^{(p-1)\vee} (\alpha_{q^{(p-1)\vee}}) \text{ for } p\in \{ 1, \cdots, n+1\},\ q=\{1, \cdots, n\},     \]
one can check that $\beta_{1, q}^{\bi^{\natural}_0}= \lan 1, -q-1\ran$, $ \beta_{n+1, q}^{\bi^{\natural}_0}= \lan q, n+1 \ran$
and for  $2\leq p \leq n$
\[ \beta_{p,q}^{\bi^{\natural}_0}=\left\{ \begin{array}{ll}  \lan p, -q-p \ran & \text{ if }p+q \leq n+1, \\
\lan p+q-n-1, p \ran & \text{ if } p+q>n+1.
\end{array}\right. \]
Since $\{\beta_{p,q}^{\bi^{\natural}_0}\}= \Phi^+$, the word $\ii^{\natural}_0$ is a reduced expression of $w_0$.
\end{proof}

\noindent {\bf $\bullet$ A reduced expression associated to a twisted Coxeter element, type $E_{6}$ case}:
Let us consider the twisted Coxeter element $1\, 2\, 6\, 3\vee$ and
the related word $\bi^{\natural}_0$ of $W$ of type $E_{6}$:
\begin{align} \label{eq: can word E}
\bi^{\natural}_0=\prod_{k=0}^{8} (1 \ 2 \ 6 \ 3)^{k\vee} \qquad \text{for $\vee$ given in \eqref{eq: F_4}.}
\end{align}

Then one can check the following proposition:

\begin{proposition}
The word $\bi^{\natural}_0$ in \eqref{eq: can word E} is a reduced expression of $w_0$ of type $E_{6}$ which is not adapted to any Dynkin quiver. Hence
it can be denoted by $\ii^{\natural}_0$ instead of $\bi^{\natural}_0$. 
\end{proposition}

\subsubsection{Twisted adapted $r$-cluster points}
\begin{definition} \label{def: twisted cluster}
Let $\ii^\natural_0$ be the reduced expression in (\ref{eq: can word A}), (\ref{eq: can word D}) or (\ref{eq: can word E}).
\begin{enumerate}
\item The $r$-cluster point $\lf \Qd \rf \seteq \lf \ii^{\natural}_0 \rf$ is called the \defn{twisted adapted $r$-cluster point} of type $A_{2n+1}$, $D_{n+1}$ or $E_6$. 
\item A class $[\ii_0]\in \lf \Qd \rf$ is called a \defn{twisted adapted class}  of type $A_{2n+1}$, $D_{n+1}$ or $E_6$.
\end{enumerate}
\end{definition}

\begin{remark} \hfill
\begin{enumerate}
\item Regarding the notion of {\it twisted adapted classes}, we introduce twisted Dynkin quivers and adapted reduced expressions to a twisted Dynkin quiver, in Appendix \ref{Appendix:twisted Dynkin quiver}.
For type $D_{n+1}$, a class in $\lf \Qd \rf$ is adapted to a twisted Dynkin quiver.
However, for type $A_{2n+1}$ and $E_6$   cases, there are classes  in $\lf \Qd \rf$ which are not adapted to a twisted Dynkin quiver (see Remark \ref{rem: twisted correspondence}.)
\item Recall that every commutation class associated to a Coxeter element belongs to the unique $r$-cluster point called the adapted $r$-cluster point. In Section \ref{Sec:Charac_twistedAR}, we show every
commutation class associated to a twisted Coxeter element belongs to $\lf \Qd \rf$. However, there exists a commutation class in $\lf \Qd \rf$ of type $A_{2n+1}$ (resp. $E_6$) which is not related to a twisted Coxeter element.
\end{enumerate}
\end{remark}

\subsubsection{Triply twisted adapted $r$-cluster points} For type $D_4$, we consider the following two words $\bi_0$ and $\bj_0$  of $W$ of type $D_{4}$:
\begin{align} \label{eq: can word D tri}
\bi^{\dagger}_0 =\prod_{k=0}^{5} (2\ 1)^{k\vee} \quad \text{and} \quad \bi^{\ddagger}_0 =\prod_{k=0}^{5} (2\ 1)^{2k\vee} \qquad \text{for $\vee$ in \eqref{eq: G_2}.}
\end{align}

Then one can check the following proposition:

\begin{proposition}
The words $\bi^{\dagger}_0$ and $\bi^{\ddagger}_0$ in \eqref{eq: can word D tri} are reduced expressions of $w_0$ of type $D_{4}$.
Hence we denote them by $\ii^{\dagger}_0$ and $\ii^{\ddagger}_0$ instead of $\bi^{\dagger}_0$ and $\bi^{\ddagger}_0$ which are not adapted to any Dynkin quiver.
\end{proposition}

\begin{definition} \label{def: triply twisted cluster} \hfill
\begin{enumerate}
\item The $r$-cluster points $\lf \mathfrak{Q} ^{\dagger} \rf \seteq \lf \ii^{\dagger}_0 \rf$ and $\lf \mathfrak{Q}^{\ddagger} \rf \seteq \lf \ii^{\ddagger}_0 \rf$ are called the \defn{triply twisted adapted $r$-cluster points}.
\item A class $[\ii_0]\in \lf \mathfrak{Q} \rf \seteq \lf \mathfrak{Q}^{\dagger} \rf \bigsqcup \lf \mathfrak{Q}^{\ddagger} \rf$ is called a \defn{triply twisted adapted class}.
\end{enumerate}
\end{definition}

\begin{remark}
In Appendix \ref{Appendix:twisted Dynkin quiver}, we show a class $[\ii_0]$ of type $D_4$ is triply twisted adapted if and only if it is adapted to a triply twisted Dynkin quiver.
\end{remark}

\section{Characterizations of (triply) twisted adapted classes} \label{Sec:Charac_twistedAR}
As we saw in the previous section, we can construct the (triply) twisted adapted $r$-cluster point
from a particular (triply) twisted Coxeter element. In this section, we shall show that  we can construct
$\lf \Qd \rf$ (resp. $\lf \mathfrak{Q} \rf$) from {\it any} (triply) twisted Coxeter element and count the number of twisted adapted classes in $\lf \Qd \rf$ (resp. $\lf \mathfrak{Q} \rf$).

Also, we introduce algorithms of finding the shapes of $\Upsilon_{[\ii_0]}$ for $[\ii_0] \in \lf \Qd \rf$ (cf. Proposition \ref{Prop:AR} for $\Gamma_Q$).

\subsection{Type $A_{2n+1}$}
Consider the monoid homomorphism
\[ \PP: \langle  I_{2n+1} \rangle \to \langle I_{2n} \rangle \ \
\text{ such that } \
 \PP(i)= \begin{cases}
i &\text{ if } 1 \le i \le n, \\
i-1 &\text{ if } n+2 \le i \le 2n+1, \\
{\rm id} &\text{ if } i=n+1.
\end{cases}\]

The following proposition can be proved by using the argument in the proof of Proposition \ref{prop: ii_0 red A}.

\begin{lemma}\label{prop: projection by P A case} Recall $\ii^\natural_0=\prod_{k=0}^{2n} (1\ 2\ 3\cdots n+1)^{k\vee}$ in \eqref{eq: can word A}.
\begin{enumerate}
\item[{\rm (1)}] $\PP(\ii^\natural_0)= (1\ 2\ 3\cdots n\ 2n\ 2n-1\cdots n+1)^n(1\, 2\, 3\cdots n)$ is a reduced expression of the longest element ${}_{2n}w_0$ of $A_{2n}$
and adapted to the Dynkin quiver which has only one source at the vertex $n+1$;
\begin{align} \label{eq:one-direct quiver A}
Q^\natural = \raisebox{0.3em}{\scalebox{0.84}{\xymatrix@R=3ex{ \circ
\ar@{<-}[r]_<{ \ 1} &  \circ
\ar@{<-}[r]_<{ \ 2}  & \cdots &  \circ
\ar@{<-}[r]_<{ \ n} &\circ \ar@{->}[r]_<{ \ n+1}
&\circ \ar@{->}[r]_<{ \ n+2} & \cdots
& \circ \ar@{<-}[l]^<{\ \ \ \ \ \ 2n} }}}.
\end{align}
\item[{\rm (2)}]  For $\ii_0 \in [\ii^\natural_0]$, we have $[\PP(\ii_0)] = [\PP(\ii^\natural_0)].$
\end{enumerate}
\end{lemma}

For a word $\bi$ and $J \subset I$, we define a subword $\bi_{|J}$ of $\bi=i_1\cdots i_l$ as follows:
$$  \bi_{|J} \seteq i_{t_1} \cdots i_{t_s} \text{ such that } \left\{
\begin{array}{l}
i_{t_x} \in J \text{ and } 1\leq t_x < t_y\leq l \text{ for all } 1 \le x<y < s,\\
\text{if } t\not \in \{t_1, t_2, \cdots, t_s\} \text{ then } i_t\not \in J.
\end{array}\right.$$

\begin{lemma} \label{Prop: Twisted A character}
For any $[\ii_0] \in \lf \Qd \rf$, $[\ii_0]$ satisfies the following properties:
\begin{enumerate}
\item[{\rm (1)}] There is $n+1$ between every adjacent $n$ and $n+2$ in $\ii_0$.
\item[{\rm (2)}] Let $J=\{n,n+1, n+2\} \in I_{2n+1}.$ We have $\PP(\ii_{0|J})= (n\ n+1)^n n \ \text{ or } \ (n+1 \ n)^n n+1.$
\end{enumerate}
\end{lemma}

\begin{proof}
We know the following facts:
\begin{enumerate}
\item[{\rm (i)}] For $\ii^\natural_0$ in \eqref{eq: can word A},  any reduced expression $\jj_0$ in $[\ii^\natural_0]$ satisfies $\jj_{0|J}=(n\ n+1\ n+2\ n+1)^n(n\ n+1)$.
\item[{\rm (ii)}]  $n^\vee= n+2$, $(n+1)^\vee= n+1$ and $(n+2)^\vee=n$.
\end{enumerate}
Hence $\ii_{0|J}$ is one of the followings:
\begin{eqnarray} &&
  \parbox{95ex}{
\begin{itemize}
\item $(n\ n+1\ n+2\ n+1)^n(n\ n+1)$,
\item $(n+1\ n+2\ n+1)(n\ n+1\ n+2\ n+1)^{n-1}(n\ n+1\ n^\vee)=(n+1\ n+2)( n+1 \ n\ n+1\ n+2)^n$,
\item $n+2(n+1\ n\ n+1\ n+2)^n (n+1)^\vee =(n+2 \ n+1\ n\ n+1)^n(n+2\  n+1),$
\item $( n+1\ n\ n+1)(n+2 \ n+1\ n\ n+1)^{n-1}(n+2\  n+1\ n)=(n+1\ n) (n+1\ n+2\ n+1\ n)^n.$
\end{itemize}
}\label{eq: 4cases}
\end{eqnarray}
We can check that every case in (\ref{eq: 4cases}) satisfies (1) and (2). Hence our assertions follow.
\end{proof}

\begin{remark}
In \eqref{eq: 4cases}, one can observe that $n+1$ is a sink or a source (but not both) for any $[\ii_0]\in \lf \Qd \rf$.
\end{remark}

\begin{lemma} \label{prop: well def on class A} \hfill
\begin{enumerate}
\item[{\rm (1)}] If $\ii'_0 ,\ii''_0  \in  [\ii_0] \in \lf \Qd \rf$ then $[\PP(\ii'_0)]=[\PP(\ii''_0)].$
Hence we can denote
\[ \PP([\ii'_0]):= [\PP(\ii'_0)].\]
\item[{\rm (2)}]  For $[\ii_0] \in \lf \Qd \rf$, we have $[\PP(\ii_0)] \in \lf \Delta \rf$ of type $A_{2n}.$
\end{enumerate}
\end{lemma}

\begin{proof}
(1) is similar to Lemma \ref{prop: projection by P A case} (2).\\
(2) By (1), we can see that $i\in I_{2n+1}\backslash\{n+1\}$ is a sink (resp. source) of $\ii_0$ if and only if $\PP(i)\in I_{2n}$ is a sink (resp. source) of $\PP(\ii_0).$ Also, $\PP(i^\vee) = (\PP(i))^\vee$. Hence If we regard $r_{id}$ as the identity map then
\[ \PP([\ii_0] \cdot r_i) = [\PP(\ii_0)]\cdot r_{\PP(i)}.\]
In Lemma \ref{prop: projection by P A case}, we showed $\PP([\ii^\natural_0])$ is adapted to the quiver $Q^\natural$ of $A_{2n}$ in~\eqref{eq:one-direct quiver A}. Since all adapted reduced expressions consist of $\lf \Delta \rf$, we proved (2).
\end{proof}

\begin{example} For the twisted adapted reduced expression $\ii^\natural_0= 1\, 2\, 3\, 5\, 4\, 3\, 1\, 2 \,3\, 5\, 4\, 3\, 1\, 2\, 3$ of type $A_5$, we have
$$\PP([\ii^\natural_0]) =1 \ 2 \ 4 \ 3 \ 1 \ 2 \ 4 \ 3 \ 1 \ 2$$
which is a reduced expression of ${}_4w_0$ and adapted to
$Q = \xymatrix@R=3ex{ *{ \circ }<3pt> \ar@{<-}[r]_<{1}  &*{\circ}<3pt>
\ar@{<-}[r]_<{2}  &*{\circ}<3pt>
\ar@{->}[r]_<{3} &*{\circ}<3pt>
\ar@{-}[l]^<{ \ \ 4} }$.
\end{example}

By Lemma \ref{prop: well def on class A}, if we restrict $\PP$ to a map on reduced expressions in $\lf \Delta \rf$  then the map can be considered as a map between classes in $\lf \Qd \rf$ of $A_{2n+1}$ and $\lf \Delta \rf$ of $A_{2n}$. Hence we use the following notation.

\begin{definition} 
The map from  $\lf \Qd \rf$ of type $A_{2n+1}$ to $\lf \Delta\rf$ of type $A_{2n}$, induced from $\PP$, is denoted by
\[ \PPi\ : \ \lf \Qd \rf \to \lf \Delta \rf, \quad [\ii_0]\mapsto [\PP(\ii_0)]=: \PPi([\ii_0]).  \]
\end{definition}

\begin{lemma} \label{prop:P onto A}
For a given Dynkin quiver $Q$ of $A_{2n}$, there are at least two distinct classes $[\ii'_0], [\ii''_0]\in \lf \Qd \rf$ such that $\PPi([\ii'_0])=\PPi([\ii''_0])=[Q].$
\end{lemma}

\begin{proof}
Consider the monoid homomorphism
\[ \RR:\langle I_{2n} \rangle \to \langle I_{2n+1} \rangle \ \text{
such that}  \ \
 i \mapsto \left\{ \begin{array}{ll} i & \text{ if }i=1, \cdots, n-1, \\ i+1 & \text{ if } i=n+2, \cdots, 2n, \\ n \ n+1 & \text{ if } i=n, \\ n+2\ n+1 & \text{ if } i=n+1. \end{array}\right. \]
Then
\begin{enumerate}[(i)]
\item $\PP\circ \RR= {\rm id}$ and $[\RR(i)\RR(j)]= [\RR(j)\RR(i)]$ for $i,j$ such that $|i-j|\geq 2$,
\item if $\RR(i) \ii$ is a reduced expression of $w_0$ then $[\RR(i) \ii]\cdot r_{\RR(i)}=[ \ii  \RR(i^\vee)]$.
\end{enumerate}

Let us consider the reduced expression ${}_{2n}\ii^\natural_0=(1\ 2\ \cdots n \ 2n \  2n-1\ \cdots n+1)^n  (1\ 2\ \cdots n)$ of the longest element
${}_{2n}w_0$ of $A_{2n}$. We shall show the following statement.
For a class of reduced expression
 \begin{equation} \label{assumption on A_2n}
 [{}_{2n}\ii'_0]=[{}_{2n}\ii^\natural_0] \cdot r_{\bi}=[i_1i_2\cdots i_l], \quad (\bi:\text{word})
 \end{equation}
of ${}_{2n}w_0$, we can induce a reduced expression $[\ii'_0]$ of $w_0$ such that
\begin{equation} \label{claim on A_2n+1}
[\ii'_0]= [\ii^\natural_0]\cdot r_{\RR(\bi)}=[\RR({}_{2n}\ii'_0)]= [\RR(i_1\cdots i_l)].
\end{equation}

Now, we use an induction on the length of $\bi$. We know $\ii^\natural_0= \RR({}_{2n}\ii^\natural_0)$ is a reduced expression of $\redez$. Suppose \eqref{claim on A_2n+1} is true for $\bi$ with length $\ell$ and take the length $\ell+1$ word $\bi'= \bi \, i$ for a sink $i=i_1$ of $[{}_{2n}\ii'_0]$. Then, by (i) and (ii), there is a reduced expression in  $[\ii'_0]$ which starts with $\RR(i)$. Moreover, we have
 \[ \ [\ii'_0]\cdot r_{\RR(i_1)}=[\RR(i_2\cdots i_l) \ \RR(i^\vee_1) ]=[ \RR(i_2\cdots i_l \ i^\vee_1)]. \]
 In other words, \eqref{claim on A_2n+1} is true for any word $\bi$ with length $\ell+1$.
As a conclusion, for any word $\bi$ consisting of $\{1, \cdots, 2n\}$,
\[\ [\ii^\natural_0]\cdot r_{\RR(\bi)}=[\RR(i_1\cdots i_l)] \text{ satisfies }\PPi([\ii^\natural_0]\cdot r_{\RR(\bi)})=  [i_1 \cdots i_l] =[{}_{2n}\ii^\natural_0] \cdot r_{\bi}.\]
Since every reduced expression adapted to a quiver consists of one $r$-cluster point, we proved that $\PPi$ is an onto map.

In addition, since a class of reduced expression $[\ii'_0]$ of $w_0$ in the image of $\RR$  has $n+1$ as a source,  $r_{n+1}\cdot[\ii'_0]$ has $n+1$ as a sink and
these two are distinct classes. The following equation is easy to check:
\[ \PPi([\ii'_0])= \PPi(r_{n+1}\cdot [\ii'_0]).\]
Hence we proved the lemma.
\end{proof}

\begin{lemma} \label{prop: 2to1}
The map $\PPi$ is a two-to-one and onto map. More precisely, for each Dynkin quiver $Q$ of $A_{2n}$,
there is a unique class $[\ii_0]$ of reduced expressions in $\lf \Qd \rf$ satisfying $\PPi([\ii_0])=[Q]$ which has $n+1$ as a source $($resp. sink$)$.
\end{lemma}

\begin{proof}
Suppose $\ii'_0$ and $\ii''_0$ are two distinct reduced expressions in $\lf \Qd \rf$ such that
\begin{enumerate}[(i)]
\item both $[\ii'_0]$ and $[\ii''_0]$ have $n+1$ as a source,  
\item $\PPi([\ii'_0])=\PPi([\ii''_0])=[Q]$ for a quiver $Q$ of type $A_{2n}.$
\end{enumerate}

As we saw in Lemma \ref{Prop: Twisted A character}, we have
\[ \ii'_{0|J}=\ii''_{0|J}=(\RR(n)\RR(n+1))^n\RR(n) \text{ or } (\RR(n+1)\RR(n))^n\RR(n+1)\]
for $J=\{n,n+1, n+2\}$. Since  $s_{n+1}s_{j}= s_{j}s_{n+1}$  for any $j\in I_{2n+1}\setminus J,$ using these commutation relations,
we can assume that every $n+1$ exists right after $n$ or $n+2$ in $\ii'_0$ and $\ii''_0$, that is $\ii'_0$ and $\ii''_0$ appear as images of $\RR$. Also,  by letting $s_{\RR(n)}= s_n s_{n+1}$ and  $s_{\RR(n+1)}= s_{n+2} s_{n+1}$, we can check that
$\RR$ preserves commutation relations.  Hence
the assumption   $\PPi([\ii'_0])=\PPi([\ii''_0])=[Q]$ implies  $[\ii'_0]=[\ii''_0]$, since $ \PP \circ \RR={\rm id}$.
In other words, there is a unique class $[\ii'_0]$ of reduced expressions in $\lf \Qd \rf$ satisfying
$$ \text{ $\PPi([\ii'_0])=[Q]$ which has $n+1$ as a source.}$$
Similarly, if $[\ii'_0]$ and $[\ii''_0]$ such that   $\PPi([\ii'_0])=\PPi([\ii''_0])=[Q]$ do not satisfy (i) but (i') below
\begin{enumerate}[(i')]
\item both $[\ii'_0]$ and $[\ii''_0]$ have $n+1$ as a sink,
\end{enumerate}
then we can show $[\ii'_0]=[\ii''_0]$.  Hence there is a unique class $[\ii'_0]$ of reduced expressions in $\lf \Qd \rf$ satisfying
$$ \text{$\PPi([\ii'_0])=[Q]$ which has $n+1$ as a sink (resp. source).}$$

Recall that we showed $\PPi$ is an onto map in Lemma \ref{prop:P onto A}. So we proved the lemma.
\end{proof}

\begin{theorem} \label{thm: 22n}
The number of classes in $\lf \Qd \rf$ is $2^{2n}.$
\end{theorem}

\begin{proof}
Since there are $(2n-1)$-many arrows in a Dynkin quiver of $A_{2n}$ and each arrow has two possible directions, the number of  Dynkin quivers is $2^{2n-1}$. By Lemma \ref{prop: 2to1}, the number of classes in $\lf \Qd \rf$ is $2 \times 2^{2n-1}=2^{2n}.$
\end{proof}

Now we focus on classes of reduced expressions related to twisted Coxeter elements. Consider
$\vee: i\mapsto 2n+1-i$ for $i\in I_{2n}$ and let $i_1 i_2 \cdots i_n \vee$ be a twisted Coxeter element of $W_{2n}$. Then it is well-known that
$i_1\  i_2 \cdots i_n\ i^\vee_1\ i^\vee_2\ \cdots i^\vee_n$ is a Coxeter element of $W_{2n}$ and hence there exist a unique $Q$ of type $A_{2n}$
such that $i_1\  i_2 \cdots i_n\ i^\vee_1\ i^\vee_2\ \cdots i^\vee_n$ is adapted to $Q$. Moreover, one can prove the following lemma
by using induction on $n$.

\begin{lemma} \label{Prop:twisted 2n}
For a twisted Coxeter element $i_1 i_2 \cdots i_n \vee$ of $W_{2n}$, we have
 \begin{equation} \label{Eqn:3_1104}
\prod_{k=0}^{2n} (i_1\ i_2\ i_3\cdots i_n)^{k\vee}
 \end{equation}
is a reduced expression of ${}_{2n}w_0$ adapted to a Dynkin quiver $Q'$ of type $A_{2n}.$
\end{lemma}

\begin{theorem} \label{thm:twisted longest}
Let $i_1\ i_2\ \cdots i_{n+1}\vee$ be a twisted Coxeter element of $A_{2n+1}.$ Then
\begin{equation}\label{Eqn: 2n+1 twsited}
 \ii_0 = \prod_{k=0}^{2n} (i_1\ i_2\ \cdots i_{n+1})^{k\vee} \ \text{is a reduced expression of $w_0$ and $[\ii_0]\in \lf \Qd \rf.$}
 \end{equation}
\end{theorem}

\begin{proof}
Any twisted Coxeter element $i_1\ i_2\ \cdots i_{n+1}\vee$  satisfies only one of followings:
\[ \text{ (i) $\bigl(i_1\ i_2\ \cdots i_{n+1} \ (n+1) \bigr)$ is not reduced,} \quad
\text{ (ii) $\bigl((n+1) \ i_1\ i_2\ \cdots i_{n+1}\bigr)$ is not reduced.}\]
Note that, for $k_1$ and $k_2$ such that $i_{k_1}= n\text{ or } n+2$ and $i_{k_2}=n+1$, if $k_1<k_2$ then the twisted Coxeter element is in the case (i) and if $k_1>k_2$ then it is in the case (ii).

For the case (i), we can assume that $i_{n+1}=n+1$. Then our assertion follows from the facts that
\begin{itemize}
\item $[\PP(\prod(i_1 \  i_2 \ \cdots  i_{n+1}))^{k\vee}] \in \lf \Delta \rf$ is a reduced expression of type $A_{2n}$ by Lemma \ref{Prop:twisted 2n},
\item $[\ii_0]=[\RR\circ \PP( \prod(i_1 \  i_2 \ \cdots  i_{n+1})^{k\vee})]\in \lf \Qd \rf$  is a reduced expression of type $A_{2n+1}$ by the proof of Lemma \ref{prop:P onto A}.
\end{itemize}
The case (ii) can be proved in the similar way.
\end{proof}

\begin{remark}
Note that there are only $4\times 3^{n-1}$-many twisted Coxeter elements of type $A_{2n+1}$ while there are $2^{2n}$-many twisted adapted classes. Thus there are twisted adapted classes which are not associated to any twisted Coxeter element in the sense of Theorem \ref{thm:twisted longest}.
\end{remark}

Now, using the properties above, we shall derive an algorithm (Algorithm \ref{Rem:surgery A}) to get the combinatorial AR-quiver associated to a twisted adapted class. Recall that, by Lemma \ref{prop: 2to1}, we have two distinct twisted adapted classes in $\lf \Qd \rf$ which are obtained from each $Q$ of type $A_{2n}$. In other words, $\PPi^{-1}(Q)$  consists of two commutation classes denoted by
\begin{align} \label{eq: two classes}
\begin{cases}
[Q^<] & \text{ if $n+1$ is a source of the class}, \\
[Q^>] & \text{ if $n+1$ is a sink of the class}.
\end{cases}
\end{align}
The procedure of finding $\ii'_0$ in $[Q^<]$ or  $[Q^>]$ is as follows.
\begin{enumerate}[(i)]
\item[{\rm (i)}] Take any ${}_{2n}\ii'_0 \in [Q]$.
\item[{\rm (ii)}] Substitute $i\in \{n+1, \cdots, 2n\}$ in ${}_{2n}\ii'_0$ by $i^+=i+1$.
\item[{\rm (iii)}] Between each adjacent $n$ and $n+2$, insert $n+1$.
\item[{\rm (iv)}] Insert another $n+1$ at the end or at the beginning $($not both$)$ of the sequence obtained in {\rm (iii)}.
\end{enumerate}

\begin{example} \label{Ex: twisted Coxeter}
Consider the Dynkin quiver $Q$ of $A_6$ :
\[ Q = {\xymatrix@C=5ex{ \circ
\ar@{->}[r]_<{  1 \ \ \ } &  \circ
\ar@{<-}[r]_<{  2 \ \ \ }  &  \circ
\ar@{->}[r]_<{  3 \ \ \ } &\circ \ar@{->}[r]_<{ 4 \ \ \ }
&\circ \ar@{<-}[r]_<{5 \ \ \ }
& \circ \ar@{-}[l]^<{\ \ \ \ 6} }}. \]
The commutation class adapted to $Q$ is
\[ [Q]= \left[\prod_{k=0}^{6} (5\ 4\ 6)^{k\vee}\right] .\]
Then one can compute that
$$\PPi^{-1}([Q])=\left\{ \ [Q^<] \seteq \left[\prod_{k=0}^{6} (6\ 5\ 7 \ 4)^{k\vee}\right], \ [Q^>]\seteq\left[\prod_{k=0}^{6} (4\ 6\ 5 \ 7)^{k\vee}\right],
 \ \right\}, $$
where the elements are commutation classes  associated to Coxeter elements $4 \ 6\ 5\ 7 \vee$ and $6\ 5\ 4\ 7 \vee$ in the sense of Theorem \ref{thm:twisted longest}.
The AR-quiver $\Gamma_Q$ and combinatorial AR-quivers $\Upsilon_{[Q^<]}$, $\Upsilon_{[Q^>]}$  can be depicted as follows:

\vskip -1em

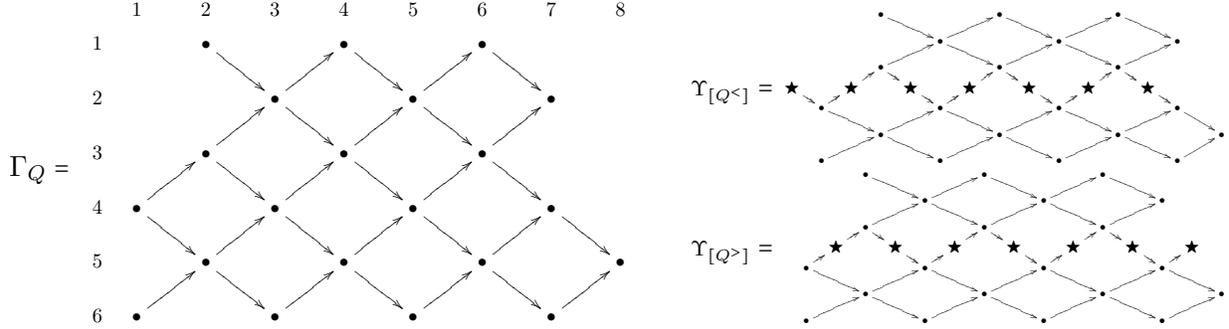
\[ \ \ \ \Gamma_Q = \raisebox{5.7em}{ \scalebox{0.7}{\xymatrix@C=2ex@R=1ex{
& 1 && 2 && 3 && 4 && 5  && 6 && 7 && 8 &&  \\
1 & &&  \bullet\ar@{->}[ddrr] && &&  \bullet\ar@{->}[ddrr]  && && \bullet\ar@{->}[ddrr]  && && \\
& &&&&& &&&& &&&& &&\\
2 & && &&\bullet\ar@{->}[ddrr] \ar@{->}[uurr]  && && \bullet \ar@{->}[ddrr]\ar@{->}[uurr] && && \bullet &&  \\
&&&&&& &&&&&&&&\\
3 & && \bullet \ar@{->}[ddrr] \ar@{->}[uurr] && && \bullet\ar@{->}[ddrr] \ar@{->}[uurr]  && && \bullet\ar@{->}[ddrr] \ar@{->}[uurr]  && && \\
&&&&&& &&&& &&&& &&\\
4 & \bullet\ar@{->}[ddrr] \ar@{->}[uurr]  && && \bullet\ar@{->}[ddrr] \ar@{->}[uurr]  && && \bullet \ar@{->}[ddrr] \ar@{->}[uurr] && && \bullet\ar@{->}[ddrr]   &&\\
&&&&&&&&&&&&&&\\
5 & && \bullet\ar@{->}[ddrr] \ar@{->}[uurr]  && && \bullet\ar@{->}[ddrr] \ar@{->}[uurr]  && && \bullet\ar@{->}[ddrr] \ar@{->}[uurr]  && && \bullet  \\
 &&&&&& &&&&&&&&&&\\
6 & \bullet \ar@{->}[uurr]  && && \bullet \ar@{->}[uurr]  && && \bullet  \ar@{->}[uurr] && && \bullet  \ar@{->}[uurr] &&
}}}
\begin{matrix}
{\scriptstyle \Upsilon_{[Q^<]}} =  \raisebox{2.8em}{\scalebox{0.45}{\xymatrix@C=2ex@R=0.5ex{
& &&  \bullet\ar@{->}[ddrr] && &&  \bullet\ar@{->}[ddrr]  && && \bullet\ar@{->}[ddrr]  && && \\
& &&&&& &&&& &&&& &&\\
& && &&\bullet\ar@{->}[ddrr] \ar@{->}[uurr]  && && \bullet \ar@{->}[ddrr]\ar@{->}[uurr] && && \bullet &&  \\
&&&&&& &&&&&&&&\\
 & && \bullet \ar@{->}[dr] \ar@{->}[uurr] && && \bullet\ar@{->}[dr] \ar@{->}[uurr]  && && \bullet\ar@{->}[dr] \ar@{->}[uurr]  && && \\
 \bigstar\ar@{->}[dr]&&\bigstar\ar@{->}[ur]&&\bigstar\ar@{->}[dr]& &\bigstar\ar@{->}[ur]&&\bigstar\ar@{->}[dr]& &\bigstar\ar@{->}[ur]&&\bigstar\ar@{->}[dr]& &&\\
& \bullet\ar@{->}[ddrr] \ar@{->}[ur]  && && \bullet\ar@{->}[ddrr] \ar@{->}[ur]  && && \bullet \ar@{->}[ddrr] \ar@{->}[ur] && && \bullet\ar@{->}[ddrr]   &&\\
&&&&&&&&&&&&&&\\
& && \bullet\ar@{->}[ddrr] \ar@{->}[uurr]  && && \bullet\ar@{->}[ddrr] \ar@{->}[uurr]  && && \bullet\ar@{->}[ddrr] \ar@{->}[uurr]  && && \bullet  \\
 &&&&&& &&&&&&&&&&\\
& \bullet \ar@{->}[uurr]  && && \bullet \ar@{->}[uurr]  && && \bullet  \ar@{->}[uurr] && && \bullet  \ar@{->}[uurr] }}} \\
{\scriptstyle \Upsilon_{[Q^>]}} = \raisebox{2.8em}{\scalebox{0.45}{\xymatrix@C=2ex@R=0.5ex{
& &&  \bullet\ar@{->}[ddrr] && &&  \bullet\ar@{->}[ddrr]  && && \bullet\ar@{->}[ddrr]  && && \\
&&&& &&&& &&&& &&\\
& && &&\bullet\ar@{->}[ddrr] \ar@{->}[uurr]  && && \bullet \ar@{->}[ddrr]\ar@{->}[uurr] && && \bullet &&  \\
&&&&&& &&&&&&&&\\
 & && \bullet \ar@{->}[dr] \ar@{->}[uurr] && && \bullet\ar@{->}[dr] \ar@{->}[uurr]  && && \bullet\ar@{->}[dr] \ar@{->}[uurr]  && && \\
 &&\bigstar\ar@{->}[ur]&&\bigstar\ar@{->}[dr]& &\bigstar\ar@{->}[ur]&&\bigstar\ar@{->}[dr]& &\bigstar\ar@{->}[ur]&&\bigstar\ar@{->}[dr]& &  \bigstar&\\
& \bullet\ar@{->}[ddrr] \ar@{->}[ur]  && && \bullet\ar@{->}[ddrr] \ar@{->}[ur]  && && \bullet \ar@{->}[ddrr] \ar@{->}[ur] && && \bullet\ar@{->}[ddrr] \ar@{->}[ur]  &&\\
&&&&&&&&&&&&\\
& && \bullet\ar@{->}[ddrr] \ar@{->}[uurr]  && && \bullet\ar@{->}[ddrr] \ar@{->}[uurr]  && && \bullet\ar@{->}[ddrr] \ar@{->}[uurr]  && && \bullet  \\
 &&&&&& &&&&&&&&&&\\
& \bullet \ar@{->}[uurr]  && && \bullet \ar@{->}[uurr]  && && \bullet  \ar@{->}[uurr] && && \bullet  \ar@{->}[uurr] &&
}}}
\end{matrix}
\]
\end{example}

\begin{example} \label{Ex: non twisted Coxeter}
Consider the Dynkin quiver
$ Q = {\xymatrix@C=3ex{ \circ
\ar@{->}[r]_<{ \ 1} &  \circ
\ar@{->}[r]_<{ \ 2}  &  \circ
\ar@{->}[r]_<{ \ 3} &\circ \ar@{->}[r]_<{ \ 4}
&\circ \ar@{<-}[r]_<{ \ 5}
& \circ \ar@{-}[l]^<{\ \ \ \ \ \ 6} }}$ of $A_6$.
The commutation class adapted to $Q$ is
\[ [Q]= [5\ 6\ 4\ 3\ 2\ 1\ 5\ 6\ 4\ 3\ 2\ 1\ 5\ 6\ 4\ 3\ 5\ 6\ 4\ 5\ 6] .\]
Then $\PPi^{-1}( [Q] )$ is
\begin{equation*}
\begin{aligned}
& \left\{ \ [Q^{<}] \seteq [ \ 6\ 7 \ 4\ 5\ 4 \ 3\ 2\ 1\ 6\ 7\ 4\ 5\ 4\ 3\ 2\ 1\ 6\ 7\ 4\ 5\ 4\ 3\ 6\ 7\ 4\ 5\ 6\ 7  \ ] \right., \\
& \left. \ \ \ \ \ \   [Q^{>}] \seteq [ \ 6\ 7 \ 5\ 4\ 3 \ 4\ 2\ 1\ 6\ 7\  5\ 4\ 3\ 4\ 2\ 1\ 6\ 7\ 5\ 4\ 3\ 4\ 6\ 7\ 5\ 4 \ 6\ 7  \ ]  \right\}.
\end{aligned}
\end{equation*}
These two classes of reduced expressions are {\it not} associated to any twisted Coxeter element.
The AR-quiver $\Gamma_Q$ and combinatorial AR-quivers related to $\PPi^{-1}( [Q] )$ can be depicted as follows:
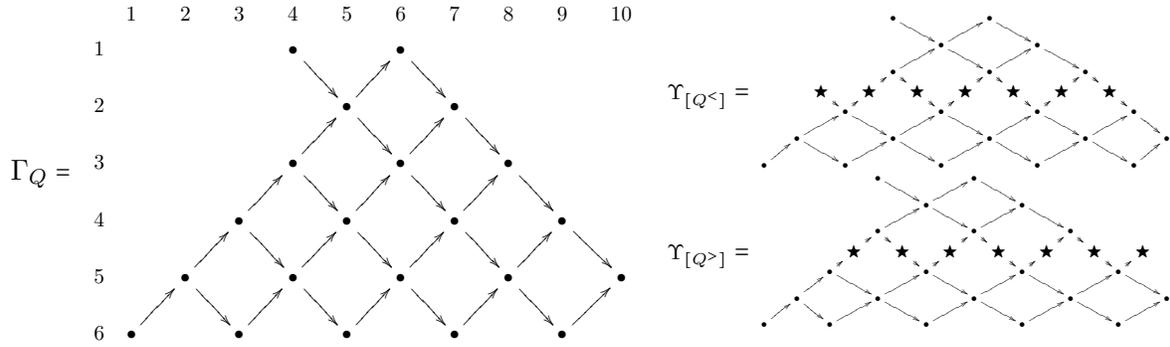
\[ \ \ \ \Gamma_Q = \raisebox{5.7em}{ \scalebox{0.73}{\xymatrix@C=1ex@R=1ex{
& 1 && 2 && 3 && 4 && 5  && 6 && 7 && 8 && 9 && 10 \\
1 && & &&  &&\bullet\ar@{->}[ddrr]  && && \bullet\ar@{->}[ddrr]  && &&  && && \\
& &&&&& &&&& &&&& &&\\
2  && & && && &&\bullet\ar@{->}[ddrr] \ar@{->}[uurr]  && && \bullet \ar@{->}[ddrr] && && &&  \\
&&&&&& &&&&&&&&\\
3 && & && && \bullet \ar@{->}[ddrr] \ar@{->}[uurr] && && \bullet\ar@{->}[ddrr] \ar@{->}[uurr]  && && \bullet\ar@{->}[ddrr]  && && \\
&&&&&& &&&& &&&& &&\\
4 && & && \bullet\ar@{->}[ddrr] \ar@{->}[uurr]  && && \bullet\ar@{->}[ddrr] \ar@{->}[uurr]  && && \bullet \ar@{->}[ddrr] \ar@{->}[uurr] && && \bullet\ar@{->}[ddrr]   &&\\
&&&&&&&&&&&&&&\\
5 && &  \bullet\ar@{->}[ddrr] \ar@{->}[uurr]  && && \bullet\ar@{->}[ddrr] \ar@{->}[uurr]  && && \bullet\ar@{->}[ddrr] \ar@{->}[uurr]  && && \bullet\ar@{->}[ddrr] \ar@{->}[uurr]  && && \bullet  \\
 &&&&&& &&&&&&&&&&\\
6  &  \bullet \ar@{->}[uurr]  & && & \bullet \ar@{->}[uurr]  && && \bullet \ar@{->}[uurr]  && && \bullet  \ar@{->}[uurr] && && \bullet  \ar@{->}[uurr] &&
}}}  \  \ \
\begin{matrix}
{\scriptstyle \Upsilon_{[Q^<]}} = \raisebox{2.8em}{\scalebox{0.45}{\xymatrix@C=1ex@R=0.5ex{
&  & &&  &&\bullet\ar@{->}[ddrr]  && && \bullet\ar@{->}[ddrr]  && &&  && && \\
&&&& &&&& &&&& &&\\
& & && && &&\bullet\ar@{->}[ddrr] \ar@{->}[uurr]  && && \bullet \ar@{->}[ddrr] && && &&  \\
&&&&&& &&&&&&&&\\
& & && && \bullet \ar@{->}[dr] \ar@{->}[uurr] && && \bullet\ar@{->}[dr] \ar@{->}[uurr]  && && \bullet\ar@{->}[dr]  && && \\
&&&\bigstar\ar@{->}[dr]&& \bigstar \ar@{->}[ur]  &&\bigstar\ar@{->}[dr]&& \bigstar \ar@{->}[ur] &&\bigstar\ar@{->}[dr]&& \bigstar \ar@{->}[ur] &&\bigstar\ar@{->}[dr]\\
& & && \bullet\ar@{->}[ddrr] \ar@{->}[ur]  && && \bullet\ar@{->}[ddrr] \ar@{->}[ur]  && && \bullet \ar@{->}[ddrr] \ar@{->}[ur] && && \bullet\ar@{->}[ddrr]   &&\\
&&&&&&&&&&&&\\
& &  \bullet\ar@{->}[ddrr] \ar@{->}[uurr]  && && \bullet\ar@{->}[ddrr] \ar@{->}[uurr]  && && \bullet\ar@{->}[ddrr] \ar@{->}[uurr]  && && \bullet\ar@{->}[ddrr] \ar@{->}[uurr]  && && \bullet  \\
 &&&&&& &&&&&&&&&&\\
  \bullet \ar@{->}[uurr]  & && & \bullet \ar@{->}[uurr]  && && \bullet \ar@{->}[uurr]  && && \bullet  \ar@{->}[uurr] && && \bullet  \ar@{->}[uurr] }}} \\
{\scriptstyle \Upsilon_{[Q^>]}} = \raisebox{2.8em}{\scalebox{0.45}{\xymatrix@C=1ex@R=0.5ex{
&  & &&  &&\bullet\ar@{->}[ddrr]  && && \bullet\ar@{->}[ddrr]  && &&  && && \\
&&&& &&&& &&&& &&\\
& & && && &&\bullet\ar@{->}[ddrr] \ar@{->}[uurr]  && && \bullet \ar@{->}[ddrr] && && &&  \\
&&&& &&&&&&&&\\
& & && && \bullet \ar@{->}[dr] \ar@{->}[uurr] && && \bullet\ar@{->}[dr] \ar@{->}[uurr]  && && \bullet\ar@{->}[dr]  && && \\
&&&&& \bigstar \ar@{->}[ur]  &&\bigstar\ar@{->}[dr]&& \bigstar \ar@{->}[ur] &&\bigstar\ar@{->}[dr]&& \bigstar \ar@{->}[ur] &&\bigstar\ar@{->}[dr]&& \bigstar\\
& & && \bullet\ar@{->}[ddrr] \ar@{->}[ur]  && && \bullet\ar@{->}[ddrr] \ar@{->}[ur]  && && \bullet \ar@{->}[ddrr] \ar@{->}[ur] && && \bullet\ar@{->}[ddrr]\ar@{->}[ur]   &&\\
&&&&&&&&&&&&\\
& &  \bullet\ar@{->}[ddrr] \ar@{->}[uurr]  && && \bullet\ar@{->}[ddrr] \ar@{->}[uurr]  && && \bullet\ar@{->}[ddrr] \ar@{->}[uurr]  && && \bullet\ar@{->}[ddrr] \ar@{->}[uurr]  && && \bullet  \\
&&&& &&&&&&&&&&\\
 \bullet \ar@{->}[uurr]  & && & \bullet \ar@{->}[uurr]  && && \bullet \ar@{->}[uurr]  && && \bullet  \ar@{->}[uurr] && && \bullet  \ar@{->}[uurr] &&
}}}
\end{matrix}
\]
\end{example}

\begin{algorithm} \label{Rem:surgery A}
As we can see in Example \ref{Ex: twisted Coxeter} and Example \ref{Ex: non twisted Coxeter}, the combinatorial AR-quiver $\Upsilon_{[\ii_0]}$ for  $[\ii_0]\in \lf \Qd \rf$ of type $A_{2n+1}$ can be constructed from the AR-quiver $\Gamma_Q$ for  $[Q]\ \seteq {\PPi([\ii_0])}$ of type $A_{2n}$ by the following {\it surgery} (see
Example \ref{Ex: twisted Coxeter} Example \ref{Ex: non twisted Coxeter}).
\begin{enumerate}
\item we put a vertex on every arrow between a residue $n$ vertex and a residue $n+1$ vertex,
\item since vertices obtained in (1) have  residue $n+1$, the original  residue $m$ for $m\geq n+1$ should be renamed by the  residue  $m+1$,
\item break every arrow with a new vertex added in (1) into two arrows : one is from the residue $n$ vertex to the residue $n+1$ vertex and the other one is from the residue $n+1$ vertex to the residue  $n+2$ vertex  (resp. one is from the residue $n+2$ vertex to the residue $n+1$ vertex and the other one is from the  residue $n+1$ vertex to the residue $n$ vertex),
\item  if $\alpha_{n+1}$ is a source (resp. sink) in $\Upsilon_{[\ii_0]}$, add a new vertex at residue $n+1$ and an arrow
to make the new vertex as a source (resp. sink) of $\Upsilon_{[\ii_0]}$.
\end{enumerate}
\end{algorithm}

\begin{definition} \label{Def:4.15_0822}
For $[\ii_0] \in \lf \Qd \rf$ such that $\PPi([\ii_0])=[Q]$,
we denote by $\Gamma_Q \cap \Upsilon_{[\ii_0]}$ the set of all vertices in $\Upsilon_{[\ii_0]}$
whose residues are contained in $I_{2n+1} \setminus \{ n+1 \}$ and by $\Upsilon_{[\ii_0]} \setminus \Gamma_Q$ the set of all vertices in
$\Upsilon_{[\ii_0]}$ whose residues are $n+1$.
\end{definition}

\begin{remark} \label{Rem:4.16_0822}
By Algorithm \ref{Rem:surgery A}, we sometimes identify a vertex in $\Gamma_Q \cap \Upsilon_{[\ii_0]}$ with a vertex in $\Gamma_Q$ also, and call it \defn{induced}.
Also,we call a subquiver $\rho$ in $\Upsilon_{[\ii_0]}$ \defn{induced} if it contains a vertex in $\Gamma_Q \cap \Upsilon_{[\ii_0]}$ and a subquiver $\rho$
in $\Upsilon_{[\ii_0]}$ \defn{totally induced} if all vertices $\rho$ are contained in $\Gamma_Q \cap \Upsilon_{[\ii_0]}$.
\end{remark}

\subsection{Type $D_{n+1}$} For $\vee$ in (\ref{eq: C_n}), consider the map
\[ \mathfrak{p}^{D_{n+1}}_{A_n}: \{ \text{ twisted Coxeter elements of $D_{n+1}$ }\} \to \{ \text{ Coxeter elements of $A_n$ }\} \]
such that $ i_1 \ i_2 \ \cdots\ i_n\vee \mapsto  \left\{ \begin{array}{ll}  i_1 \ i_2 \ \cdots\ i_n & \text{ if } i_t=n, \\  (i_1 \ i_2 \ \cdots\ i_n)^\vee & \text{ if } i_t=n+1, \end{array} \right.$
for $t$ satisfying $i_t\in \{n, n+1\}.$

\begin{proposition} \label{prop: 2 to 1 Dt to A}
The map $ \mathfrak{p}^{D_{n+1}}_{A_n}$ is a two-to-one and onto map.
\end{proposition}

\begin{proof}
Suppose $i_1 \ i_2 \ \cdots\ i_n $ can be considered as a Coxeter element of $A_n$. Then both $[i_1 \ i_2 \ \cdots\ i_n]\vee $ and  $[(i_1 \ i_2 \ \cdots\ i_n)^\vee]\vee$ are twisted Coxeter elements of $D_{n+1}$. Thus our
assertion follows.
\end{proof}

Recall that there is an involution $*$ on the index set $I^{D}_{n+1}$ such that  $ w_0(\alpha_i)=-\alpha_{i^*}.$
If $n+1$ is odd then $* : i\mapsto \left\{ \begin{array}{ll} i & \text{ if } i\neq n, n+1, \\ i+(-1)^{\delta_{n+1, i}} & \text{ if } i=n, n+1. \end{array} \right.$ On the other hand,
if $n+1$ is even then $ i^*=i , \text{ for } i\in I.$

\begin{proposition} \label{prop: reverse direction for fullfiling}
Let $[\ii_0]$ be a twisted adapted class of type $D_{n+1}$. Then there is a twisted Coxeter element $i_1\, i_2\,  \cdots \, i_n \vee$ such that
\[\,  [\ii_0]=\left[ \prod_{k=0}^{n} (i_1\ i_2\ \cdots \ i_n)^{k\vee}\right].  \]
\end{proposition}

\begin{proof}
Suppose $[\ii'_0]$ is a class of reduced expressions of $w_0$ such that   $[\ii'_0]=\left[ \prod_{k=0}^{n} (i'_1\ i'_2\ \cdots \ i'_n)^{k\vee}\right]$ for a twisted Coxeter element $i'_1\ i'_2\ \cdots \ i'_n \vee$. Then
\begin{enumerate}[(i)]
\item If $i$ is a sink of $[\ii'_0]$ then there is  a reduced expression $\jj=j_1\,  j_2\, \cdots\,  j_n$ such that
$$[j_1 \ j_2\ \cdots\  j_n]=[i'_1\ i'_2\ \cdots \ i'_n], \quad j_1=i \quad \text{ and } \quad [\ii_0]= \big[\prod_{k=0}^{n} (j_1 \ j_2\ \cdots \ j_n)^{k\vee}\big].$$
\item $[\ii''_0] = [\ii'_0] r_{i'_1}$ has the form of
\[  [\ii''_0]= \left[ \prod_{k=0}^{n} (i'_2\ i'_3 \ \cdots \ i'_n\ i'^\vee_1)^{k\vee}\right]\]
and $i'_2\ i'_3 \ \cdots \ i'_n\ i'^\vee_1 \vee$ is a twisted Coxeter element.
\end{enumerate}
Since $[\ii_0^\natural]=\left[ \prod_{k=0}^{n} (1\ 2\ \cdots \ n)^{k\vee}\right]\in \lf \Qd\rf$,  for any $[\ii_0]\in \lf \Qd \rf$,  there is a word $\textbf{w}$ such that $ [\ii_0]= [\ii_0^\natural] r_{\textbf{w}}.$ Hence $[\ii_0]=\left[ \prod_{k=0}^{n} (i_1\ i_2\ \cdots \ i_n)^{k\vee}\right]$  for a twisted Coxeter element $i_1\, i_2\,  \cdots \, i_n \vee$.
\end{proof}

\begin{proposition} \label{prop: one direction for fullfiling}
For a twisted Coxeter element $ i_1 i_2  \cdots i_n \vee$ of $D_{n+1}$, the word  \[ \prod_{k=0}^{n} (i_1\ i_2\ \cdots \ i_n)^{k\vee} \ \text{ is a reduced expression of $w_0$.} \]
\end{proposition}

\begin{proof} 
Recall that if $\phi_{Q'}= i'_1 \, i'_2 \cdots i'_n$ is the Coxeter element associated to the Dynkin quiver $Q'$ of type $A_n$ then $\phi_{Q''}=  i'_2\, i'_3 \cdots i'_n\, i'_1$ is the Coxeter element associated to the Dynkin quiver $Q''$ where $r_{i'_1} Q'=Q''$. Hence if  ${}_n{\ii'}_0^A$ is the reduced expression of type $A_n$ associated to the Coxeter elements $i'_1\, i'_2\,  \cdots \, i'_n$ then  ${}_n{\ii''}_0^A$ such that $[{}_n{\ii''}_0^A]=[{}_n{\ii'}_0^A]r_{i_1}$ is  associated to $\phi_{Q''}$.

On the other hand, in the proof of Proposition \ref{prop: reverse direction for fullfiling}, we noted that if $[\ii'_0]\in \lf \Qd \rf$ of type $D_{n+1}$ is associated to the twisted Coxeter element $i_1' \ i_2' \ \cdots \ \ i_n' \vee$ then $[\ii''_0]=  [\ii'_0] r_{i'_1}$ is associated to $i_2' \ i_3' \ \cdots \ \ i_n' \ {i'}_1^\vee \vee.$ Since $\vee$ is the map switching $n$ and $n+1$, by induction, we have the following statement.

 For   $[{}_n\ii_0^A]\in \lf \Delta\rf$  associated to $\phi_Q= 1 \, 2 \, \cdots \, n$, let a word $\textbf{w}\in \left< I_n \right>$ have the property
\[ \, [{}_n\ii_0^A] r_{\textbf{w}} = [{}_n\ii_0^{'A}] \text{ where }{}_n\ii_0^{'A} \text{ is associated to } i'_1\, i'_2\, \cdots \, i'_n.\]
Then we have $\textbf{w}_{n+1}\in \left< I_{n+1} \right>$ such that
\begin{itemize}
\item we get $\textbf{w}$ from $\textbf{w}_{n+1}$ by the map $\left< I_{n+1} \right>\to \left< I_{n} \right>$ such that $i\mapsto i$  for $i\in I_n$ and $n+1 \mapsto n$.
\item  for classes in $\lf \Qd \rf$ of type $D_{n+1}$,
\[\, [\ii_0^\natural]r_{\textbf{w}_{n+1}}= [\prod_{k=0}^{n} (1\ 2\ \cdots \ n)^{k\vee}] r_{\textbf{w}_{n+1}} = [\prod_{k=0}^{n} (j'_1\ j'_2\ \cdots \ j'_n)^{k\vee}]\,  \]
where $\mathfrak{p}^{D_{n+1}}_{A_n}(j'_1\ j'_2\ \cdots \ j'_n \vee)=  {i'_1}\cdots {i'_n}$ and
\[  [\prod_{k=0}^{n} (j'_1\ j'_2\ \cdots \ j'_n)^{k\vee}] r_{j'_1}r_{j'_2}\cdots r_{j'_n}=  [\prod_{k=0}^{n} ({j'}^\vee_1\ {j'}^\vee_2\ \cdots \ {j'}^\vee_n)^{k\vee}].\]
\end{itemize}
Now, since $\mathfrak{p}^{D_{n+1}}_{A_n}$ is a surjective two-to-one map such that $(\mathfrak{p}^{D_{n+1}}_{A_n})^{-1} ({i'_1} \, {i'_2} \, \cdots \, {i'_n})= \{ j'_1\ j'_2\ \cdots \ j'_n \vee, {j'}^\vee_1\ {j'}^\vee_2\ \cdots \ {j'}^\vee_n \vee\}$, we conclude that
\begin{enumerate}[(a)]
\item the word $\prod_{k=0}^{n} [(i_1\ i_2\ \cdots \ i_n)^{k\vee}]$ associated to the twisted Coxeter element $ i_1 i_2  \cdots i_n \vee$ is obtained by applying a reflection functor to $[\ii_0^\natural]$.
\item by (i), $\prod_{k=0}^{n} (i_1\ i_2\ \cdots \ i_n)^{k\vee}$ is a reduced expression of $w_0$.
\end{enumerate}
\end{proof}




\begin{theorem} \label{Thm:D_twisted adapted class}
The twisted adapted $r$-cluster point $\lf \Qd \rf$ of type $D_{n+1}$ consists of the classes of the form
 \begin{equation}\label{Eqn:twisted D reduced expression}
[\prod_{k=0}^{n} (i_1\ i_2\ \cdots \ i_n)^{k\vee}],
\end{equation}
where $i_1\ i_2\ \cdots \ i_n\vee$ is a twisted Coxeter element of type $D_{n+1}$. In addition, the number of classes in the twisted adapted $r$-cluster point $\lf \Qd \rf$ of type $D_{n+1}$ is $2^{n}.$
\end{theorem}

\begin{proof}
The first assertion follows from Proposition \ref{prop: reverse direction for fullfiling} and \ref{prop: one direction for fullfiling}. Also, since distinct twisted Coxeter elements of type $D_{n+1}$ give rise to reduced expressions with distinct combinatorial AR-quivers, the number of classes in $\lf \Qd \rf$ of type $D_{n+1}$ is the same as the number of twisted Coxeter elements  of type $D_{n+1}$. Hence the number of classes in $\lf \Qd \rf$  of type $D_{n+1}$ is $2^n$.
\end{proof}

\begin{remark}
Analogous to the case of adapted commutation classes, there is the natural one-to-one correspondence between the set of twisted Coxeter elements and twisted commutation classes of type $D_{n+1}$. It follows by the fact that the number of elements in both sets is $2^n$.
\end{remark}

By Proposition \ref{prop: one direction for fullfiling} and  Proposition \ref{prop: reverse direction for fullfiling},
we can consider $\mathfrak{p}_{A_n}^{D_{n+1}}$ as a two-to-one onto map between twisted adapted classes of type
$D_{n+1}$ and adapted classes of type $A_n$, i.e., $\lf \Qd \rf \twoheadrightarrow \lf \Delta\rf$. Thus, from now on, we use the notation
\begin{equation}
\mathfrak{p}_{A_n}^{D_{n+1}}: \lf \Qd \rf \twoheadrightarrow \lf \Delta\rf
\end{equation}
for the map between commutation classes.

\begin{example} \label{ex: Twisted D first}
For $[\ii_0]=[\prod_{k=0}^4 (2\, 1\, 3\, 5)^{k\vee}] \in \lf \Qd \rf$, one can check that $\mathfrak{p}_{A_n}^{D_{n+1}}([\ii_0])=[Q]$ where
$$ \quad Q = {\xymatrix@C=5ex{ \circ  \ar@{->}[r]_<{  1} &  \circ \ar@{<-}[r]_<{  2}  &  \circ \ar@{<-}[r]_<{  3} & \circ \ar@{-}[l]^<{\ \ \ \ 4} }}.$$
The combinatorial AR-quiver $\Upsilon_{[\ii_0]}$ can be depicted as follows:
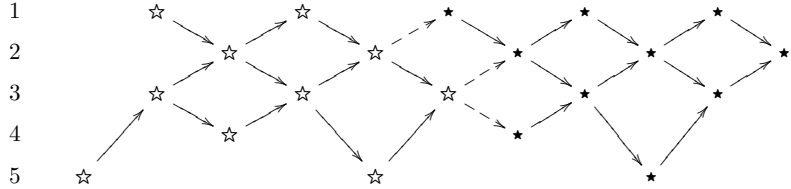
\begin{equation*}
 \scalebox{0.8}{\xymatrix@C=4ex@R=1ex{
1&&\medstar\ar@{->}[dr]  && \medstar\ar@{->}[dr]  &&  \filledstar \ar@{->}[dr]  && \filledstar\ar@{->}[dr] && \filledstar \ar@{->}[dr]\\
2&&& \medstar\ar@{->}[dr]\ar@{->}[ur]  &&\medstar\ar@{->}[dr]\ar@{-->}[ur]  && \filledstar\ar@{->}[dr]\ar@{->}[ur]  && \filledstar\ar@{->}[dr]\ar@{->}[ur] && \filledstar \\
3& & \medstar\ar@{->}[dr]\ar@{->}[ur]  &&\medstar\ar@{->}[ddr]\ar@{->}[ur]  &&\medstar\ar@{-->}[dr]\ar@{-->}[ur]  && \filledstar\ar@{->}[ddr]\ar@{->}[ur] && \filledstar\ar@{->}[ur] \\
4&&& \medstar\ar@{->}[ur]  &&&& \filledstar \ar@{->}[ur] \\
5&\medstar\ar@{->}[uur]  &&&& \medstar\ar@{->}[uur]  &&&& \filledstar \ar@{->}[uur]
}}
\end{equation*}
\end{example}

By Remark \ref{rem: boundary} and Proposition \ref{prop: reverse direction for fullfiling},  $\Upsilon_{[\ii'_0]}$ for $[\ii'_0] \in \lf \Qd \rf$ can be considered as
concatenation of $\Gamma_{Q^*}$ and $\Gamma_{Q}$ for $Q$ of type $A_n$, where $\mathfrak{p}_{A_n}^{D_{n+1}}([\ii'_0])=[Q]$.

\begin{remark} \label{rem surgery D}
In the above example, $\Gamma_Q$ is isomorphic to the full subquiver of $\Upsilon_{[\ii_0]}$ consisting of $\filledstar$'s and
$\Gamma_{Q^*}$ is isomorphic to the full subquiver of $\Upsilon_{[\ii_0]}$ consisting of $\medstar$'s:
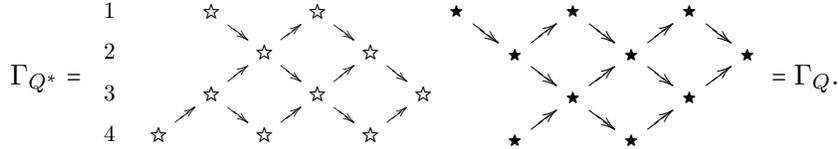
$$
 \Gamma_{Q^*} =   \raisebox{2.5em}{\scalebox{0.8}{ \xymatrix@C=2ex@R=1ex{
1 && \medstar \ar@{->}[dr]  && \medstar \ar@{->}[dr]  \\
2 &&&\medstar \ar@{->}[dr]\ar@{->}[ur]  &&\medstar \ar@{->}[dr] \\
3 & & \medstar\ar@{->}[dr]\ar@{->}[ur]  &&\medstar\ar@{->}[dr]\ar@{->}[ur]&&\medstar\\
4 &\medstar\ar@{->}[ur]&& \medstar\ar@{->}[ur] && \medstar\ar@{->}[ur]
}}}
     \raisebox{2.5em}{\scalebox{1}{\xymatrix@C=2ex@R=1ex{
 \filledstar \ar@{->}[dr]  &&\filledstar\ar@{->}[dr] &&\filledstar \ar@{->}[dr]\\
& \filledstar\ar@{->}[dr]\ar@{->}[ur]  && \filledstar\ar@{->}[dr]\ar@{->}[ur] &&\filledstar \\
  && \filledstar\ar@{->}[dr]\ar@{->}[ur] &&\filledstar\ar@{->}[ur] \\
&\filledstar\ar@{->}[ur] && \filledstar\ar@{->}[ur]
}}} = \Gamma_Q.
$$
\end{remark}

\begin{remark} \label{Rem:typeDtwisted_typeA_two_copy}
As quivers, two AR-quivers $\Gamma_{Q^*}$ and $\Gamma_Q$ are isomorphic to each other. More explicitly, $\Gamma_{Q^*} \simeq \Gamma_Q$ by  the map which relates a vertex in $\Gamma_{Q^*}$ with residue $i$  to a vertex in $\Gamma_Q$ with residue $i^*=n+1-i$. This fact will be used crucially in Section \ref{subsec:label_D}.
\end{remark}

We can derive an algorithm of finding twisted adapted AR-quiver of type $D_{n+1}$ by Remark \ref{rem surgery D}.

\begin{algorithm} \label{Alg surgery D}
We can draw $\Upsilon_{[\ii_0]}$ for $[\ii_0] \in \lf \Qd \rf$ from the AR-quivers $\Gamma_{Q^*}$ and $\Gamma_{Q}$ of type $A_n$ as follows:
\begin{enumerate}
\item Draw $\Gamma_{Q^*}$ and $\Gamma_{Q}$ and juxtapose them.
\item Draw arrows from rightmost vertices in $\Gamma_{Q^*}$ to leftmost vertices in $\Gamma_{Q}$ if their residues are adjacent to each other in $\Delta_{A_n}$.
\item Change residues of vertices from $n$ to $n+1$ if they correspond to $n+1$ in $\ii_0$.
\end{enumerate}
\end{algorithm}

\begin{example} The two to one and onto map $\mathfrak{p}^{D_{n+1}}_{A_n}$ from $\lf \Qd \rf$ of type $D_{n+1}$
to $\lf \Delta \rf$ of type $A_n$ can be described as follows: For $Q = {\xymatrix@C=3ex{ \circ  \ar@{->}[r]_<{  1} &  \circ \ar@{<-}[r]_<{  2}  &  \circ \ar@{<-}[r]_<{  3} & \circ \ar@{-}[l]^<{\ \ \ \ 4} }}$,
we can construct a quiver $\Gamma_{Q^*} \overset{+}{\sqcup} \Gamma_{Q}$ concatenating $\Gamma_{Q^*}$ and $\Gamma_{Q}$ (see Algorithm \ref{Alg surgery D}):
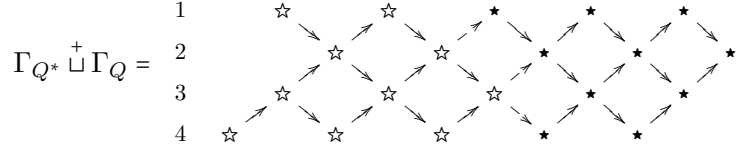
\begin{equation*}
\Gamma_{Q^*} \overset{+}{\sqcup} \Gamma_{Q}=   \raisebox{2.1em}{ \scalebox{0.8}{\xymatrix@C=2ex@R=1ex{
1&&\medstar\ar@{->}[dr]  && \medstar\ar@{->}[dr]  &&  \filledstar \ar@{->}[dr]  && \filledstar\ar@{->}[dr] && \filledstar \ar@{->}[dr]\\
2&&& \medstar\ar@{->}[dr]\ar@{->}[ur]  &&\medstar\ar@{->}[dr]\ar@{-->}[ur]  && \filledstar\ar@{->}[dr]\ar@{->}[ur]  && \filledstar\ar@{->}[dr]\ar@{->}[ur] && \filledstar \\
3& & \medstar\ar@{->}[dr]\ar@{->}[ur]  &&\medstar\ar@{->}[dr]\ar@{->}[ur]  &&\medstar\ar@{-->}[dr]\ar@{-->}[ur]  && \filledstar\ar@{->}[dr]\ar@{->}[ur] && \filledstar\ar@{->}[ur] \\
4&\medstar\ar@{->}[ur]&& \medstar\ar@{->}[ur]  &&\medstar\ar@{->}[ur]&& \filledstar \ar@{->}[ur] && \filledstar \ar@{->}[ur]\\
}}}
\end{equation*}
Then we have two distinct combinatorial AR-quivers in $(\mathfrak{p}^{D_{n+1}}_{A_n})^{-1}([Q])$:
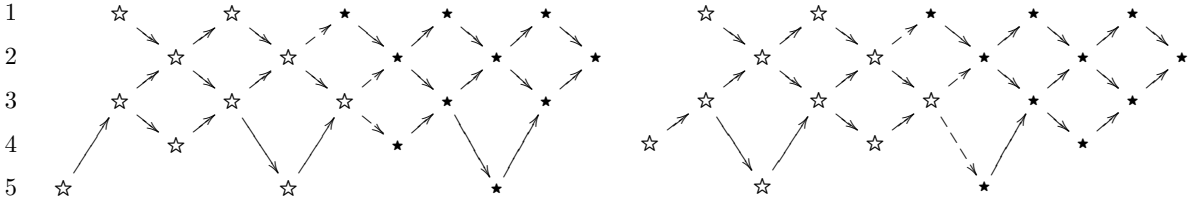
\begin{equation*}
  \raisebox{2.5em}{ \scalebox{0.85}{\xymatrix@C=2ex@R=1ex{
1&&\medstar\ar@{->}[dr]  && \medstar\ar@{->}[dr]  &&  \filledstar \ar@{->}[dr]  && \filledstar\ar@{->}[dr] && \filledstar \ar@{->}[dr]\\
2&&& \medstar\ar@{->}[dr]\ar@{->}[ur]  &&\medstar\ar@{->}[dr]\ar@{-->}[ur]  && \filledstar\ar@{->}[dr]\ar@{->}[ur]  && \filledstar\ar@{->}[dr]\ar@{->}[ur] && \filledstar \\
3& & \medstar\ar@{->}[dr]\ar@{->}[ur]  &&\medstar\ar@{->}[ddr]\ar@{->}[ur]  &&\medstar\ar@{-->}[dr]\ar@{-->}[ur]  && \filledstar\ar@{->}[ddr]\ar@{->}[ur] && \filledstar\ar@{->}[ur] \\
4&&& \medstar\ar@{->}[ur]  &&&& \filledstar \ar@{->}[ur] \\
5&\medstar\ar@{->}[uur]  &&&& \medstar\ar@{->}[uur]  &&&& \filledstar \ar@{->}[uur]
}}} \ \
  \raisebox{2.5em}{ \scalebox{0.85}{\xymatrix@C=2ex@R=1ex{
&\medstar\ar@{->}[dr]  && \medstar\ar@{->}[dr]  &&  \filledstar \ar@{->}[dr]  && \filledstar\ar@{->}[dr] && \filledstar \ar@{->}[dr]\\
&& \medstar\ar@{->}[dr]\ar@{->}[ur]  &&\medstar\ar@{->}[dr]\ar@{-->}[ur]  && \filledstar\ar@{->}[dr]\ar@{->}[ur]  && \filledstar\ar@{->}[dr]\ar@{->}[ur] && \filledstar \\
 & \medstar\ar@{->}[ddr]\ar@{->}[ur]  &&\medstar\ar@{->}[dr]\ar@{->}[ur]  &&\medstar\ar@{-->}[ddr]\ar@{-->}[ur]  && \filledstar\ar@{->}[dr]\ar@{->}[ur] && \filledstar\ar@{->}[ur] \\
\medstar\ar@{->}[ur]  &&&& \medstar\ar@{->}[ur]  &&&& \filledstar \ar@{->}[ur] &&&& \\
&& \medstar\ar@{->}[uur]  &&&& \filledstar \ar@{->}[uur]  &&&\\
}}}
\end{equation*}
by assigning residues $n$ and $n+1$ (resp. $n+1$ and $n$) to vertices in the last row of $\Gamma_{Q^*} \overset{+}{\sqcup} \Gamma_{Q}$ alternatingly from the right.
\end{example}

\subsection{Type $E_6$} Recall that
\begin{equation}\label{Eqn:E_6 twisted reduced}
 \ii^\natural_0 =\prod_{k=0}^{8} (1\ 2\ 6\ 3)^{k\vee}.
 \end{equation}

By using the example below and reflection functors, we can check that
there are $32$ distinct twisted adapted classes in $\lf \Qd \rf=\lf \ii^\natural_0 \rf$ while there are only $24$ distinct twisted Coxeter elements. Here the number $32$
coincides with the number of distinct Dynkin quivers of type $E_6$.

\begin{example} \label{ex:E6 twist}
The combinatorial AR-quiver $\Upsilon_{[\ii^\natural_0]}$ can be drawn as follows:  ${\scriptstyle\prt{a_1a_2a_3}{a_4a_5a_6}} \seteq \displaystyle \sum_{i=1}^6 a_i \alpha_i$

\vskip -0.5em

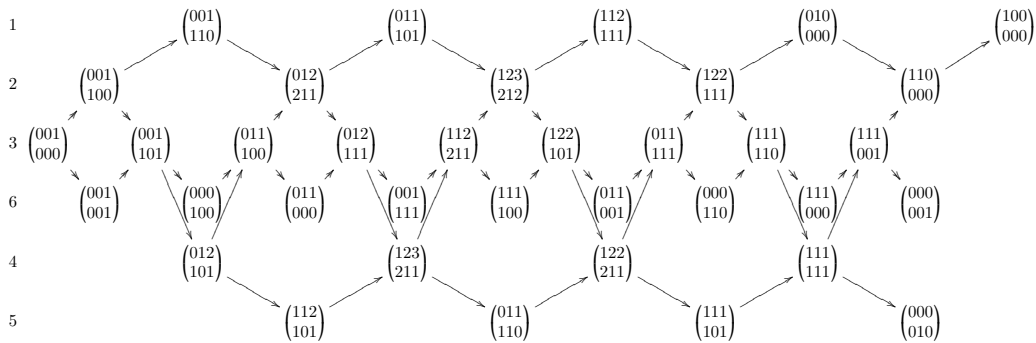
$$\scalebox{0.58}{\xymatrix@C=0.1ex@R=1.3ex{
1 &&&& {\scriptstyle\prt{001}{110}} \ar@{->}[drr] &&&& {\scriptstyle\prt{011}{101}}\ar@{->}[drr] &&&& {\scriptstyle\prt{112}{111}}\ar@{->}[drr]
&&&& {\scriptstyle\prt{010}{000}} \ar@{->}[drr] &&&&&&& {\scriptstyle\prt{100}{000}} \\
2 && {\scriptstyle\prt{001}{100}} \ar@{->}[urr]\ar@{->}[dr] &&&& {\scriptstyle\prt{012}{211}}\ar@{->}[urr]\ar@{->}[dr] &&&&{\scriptstyle\prt{123}{212}}
\ar@{->}[urr]\ar@{->}[dr]
&&&&{\scriptstyle\prt{122}{111}} \ar@{->}[urr]\ar@{->}[dr] &&&& {\scriptstyle\prt{110}{000}}\ar@{->}[urrrrr]\\
3 & {\scriptstyle\prt{001}{000}}\ar@{->}[ur]\ar@{->}[dr] && {\scriptstyle\prt{001}{101}} \ar@{->}[dr]\ar@{->}[ddr] && {\scriptstyle\prt{011}{100}} \ar@{->}[ur]\ar@{->}[dr]
&& {\scriptstyle\prt{012}{111}} \ar@{->}[dr]\ar@{->}[ddr] && {\scriptstyle\prt{112}{211}}\ar@{->}[ur]\ar@{->}[dr] &&{\scriptstyle\prt{122}{101}}\ar@{->}[ddr]\ar@{->}[dr]
&&{\scriptstyle\prt{011}{111}}\ar@{->}[ur]\ar@{->}[dr] &&
{\scriptstyle\prt{111}{110}}\ar@{->}[dr]\ar@{->}[ddr] && {\scriptstyle\prt{111}{001}}\ar@{->}[ur]\ar@{->}[dr]\\
6 && {\scriptstyle\prt{001}{001}}\ar@{->}[ur] && {\scriptstyle\prt{000}{100}} \ar@{->}[ur] && {\scriptstyle\prt{011}{000}} \ar@{->}[ur] &&
{\scriptstyle\prt{001}{111}}\ar@{->}[ur] &&{\scriptstyle\prt{111}{100}}\ar@{->}[ur] && {\scriptstyle\prt{011}{001}} \ar@{->}[ur]&& {\scriptstyle\prt{000}{110}}
\ar@{->}[ur] &&{\scriptstyle\prt{111}{000}}\ar@{->}[ur]  &&{\scriptstyle\prt{000}{001}} \\
4 &&&& {\scriptstyle\prt{012}{101}}\ar@{->}[drr]\ar@{->}[uur] &&&& {\scriptstyle\prt{123}{211}}\ar@{->}[drr]\ar@{->}[uur]
&&&& {\scriptstyle\prt{122}{211}}\ar@{->}[drr]\ar@{->}[uur] &&&& {\scriptstyle\prt{111}{111}} \ar@{->}[drr]\ar@{->}[uur]\\
5 &&&&&& {\scriptstyle\prt{112}{101}}\ar@{->}[urr]  &&&& {\scriptstyle\prt{011}{110}} \ar@{->}[urr]
&&&& {\scriptstyle\prt{111}{101}} \ar@{->}[urr]  &&&& {\scriptstyle\prt{000}{010}}
}}
$$
\end{example}

\subsection{Triply twisted adapted classes of type $D_4$} Recall that
\begin{equation}\label{Eqn:D_4 twisted reduced}
 \ii^\dagger_0 =\prod_{k=0}^{5} (2\ 1)^{k\vee} \quad \text{and} \quad \ii^\ddagger_0 =\prod_{k=0}^{5} (2\ 1)^{2k\vee}.
 \end{equation}
Each triply twisted adapted class $[\ii_0'] \in \lf\mathfrak{Q}\rf$
consists of a unique reduced expression and there are $6$ distinct twisted adapted classes in each triply twisted adapted $r$-cluster point.
Recall $12$ is the number of distinct triply twisted Coxeter elements.
\begin{example} \label{ex:D4 triply twist}
The combinatorial AR-quiver $\Upsilon_{[ \ii^\dagger_0 ]}$ can be drawn as follows:
$$\scalebox{0.67}
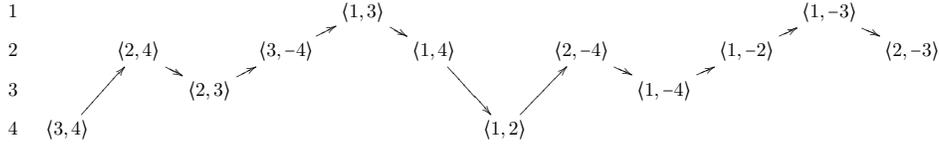
{\xymatrix@C=2ex@R=1ex{
1 &&&&& \lan 1, 3 \ran \ar@{->}[dr] &&&&&& \lan 1, -3 \ran \ar@{->}[dr]\\
2 && \lan 2, 4 \ran \ar@{->}[dr] && \lan 3, -4 \ran\ar@{->}[ur] && \lan 1, 4 \ran \ar@{->}[ddr]&& \lan 2, -4 \ran \ar@{->}[dr]&& \lan 1, -2 \ran \ar@{->}[ur] && \lan 2, -3 \ran\\
3 &&& \lan 2, 3 \ran\ar@{->}[ur] &&&&&& \lan 1, -4 \ran\ar@{->}[ur]\\
4 & \lan 3, 4 \ran \ar@{->}[uur] &&&&&& \lan 1, 2 \ran \ar@{->}[uur]
}}
$$
\end{example}

\section{Twisted AR-quivers and folded AR-quivers} \label{Sec:coordinate of twisted AR}

\begin{definition}
A combinatorial AR-quiver $\Upsilon_{[\ii_0]}$ associated to a (triply) twisted adapted class $[\ii_0] \in \lf \Qd \rf$ (resp. $[\ii_0] \in \lf \mathfrak{Q} \rf$) is called a \defn{$($triply$)$ twisted AR-quiver}.
\end{definition}

Recall, in Algorithm \ref{Alg:AR}, the coordinate system on $\Gamma_Q$ is useful to indicate vertices.
Since twisted AR-quivers $\Upsilon_{[\ii_0]}$ have similar patterns with $\Gamma_Q$ (see Algorithm \ref{Rem:surgery A} and Algorithm \ref{Alg surgery D}) it is worth to introduce coordinate systems on twisted AR-quivers. In this section, we introduce coordinate systems on $\Upsilon_{[\ii_0]}$ and, using the coordinates, define folded AR-quivers.

\subsection{Coordinate system on a (triply) twisted AR-quiver}
Let us fix an automorphism $\vee$ on the Dynkin diagram of type $X$ where $ X=A_{2n+1},D_{n+1},E_6$  and let $I$ be the index set of type $X$.
Then we can consider its folded type $\widehat{X}$ and the corresponding orbit index set $\widehat{I}=\{ \hat{i} \ | \ i \in I \}$ of $I$.
If we choose $\vee$ as one of \eqref{eq: B_n}, \eqref{eq: C_n}, \eqref{eq: F_4} and \eqref{eq: G_2}, then $\widehat{X}$ is one of $B_{n+1}$, $C_n$, $F_4$ or $G_2$.
We denote by $\widehat{\Pi}=\{ \alpha_{\hat{i}} \ | \ \hat{i} \in \widehat{I} \}$ the set of simple roots of type $\widehat{X}$.

Now we can give a coordinate system on $\Upsilon_{[\ii_0]}$ by using root system of type $\widehat{X}$. To do this,
$$\text{we fix the length $|\alpha_{\hat{i}_s}|$ of the longest root as $1$. }$$

\begin{definition} \label{def: length on arrow}
Let  $[\ii_0]\in \lf \Qd \rf$ or $\lf \mathfrak{Q}\rf$. For an arrow $\mathsf{a}$ between a vertex of residue $i$ and a vertex of residue $j$ in $\Upsilon_{[\ii_0]}$ , we assign the \defn{length} $\ell(\mathsf{a})$ that is the minimum of $|\alpha_{\hat{i}}|^2$ and $|\alpha_{\hat{j}}|^2$:
$$ \ell(\mathsf{a}) \seteq \min\{ |\alpha_{\hat{i}}|^2, \ |\alpha_{\hat{j}}|^2 \}.$$
\end{definition}

Using the length of an arrow, we can naturally define a  coordinate system on $\Upsilon_{[\ii_0]}$ for $[\ii_0]\in \lf \Qd \rf$. Precisely, for
\begin{align} \label{def: d}
\frac{1}{\mathsf{d}} \seteq \min\{ |\alpha_{\hat{i}}|^2 \ | \ \hat{i} \in \widehat{I} \},
\end{align}
we  assign a coordinate $(i,p) \in I \times \frac{1}{\mathsf{d}}\Z$ to a vertex $v$,
where $i$ is the residue of $v$ and $p$ is a number induced from lengths of arrows.
For $\beta \in \PR$, we denote by $\Omega_{[\ii_0]}(\be) \in I \times \frac{1}{\mathsf{d}}\Z$ the \defn{coordinate} of $\be$ in $\Upsilon_{[\ii_0]}$.

\begin{example} \label{ex: Q<Q>}
The coordinate systems for $\Upsilon_{[Q^<]}$ and $\Upsilon_{[Q^>]}$ in Example \ref{Ex: twisted Coxeter} are given as follows:
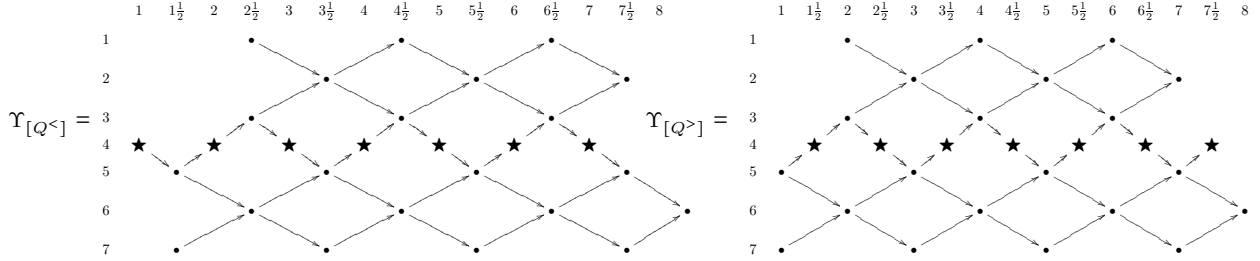
\[ {\scriptstyle \Upsilon_{[Q^<]} } = \raisebox{4em}{\scalebox{0.5}{\xymatrix@C=2ex@R=1ex{
& 1 & 1\frac{1}{2} &2 & 2\frac{1}{2} &3 & 3\frac{1}{2} &4 & 4\frac{1}{2} &5 & 5\frac{1}{2} &
6 & 6\frac{1}{2} &7 & 7\frac{1}{2} &8 \\
1 &&&&  \bullet\ar@{->}[ddrr] && &&  \bullet\ar@{->}[ddrr]  && && \bullet\ar@{->}[ddrr]  && && \\
& &&&&& &&&& &&&& &&\\
2 && && &&\bullet\ar@{->}[ddrr] \ar@{->}[uurr]  && && \bullet \ar@{->}[ddrr]\ar@{->}[uurr] && && \bullet &&  \\
&&&&&& &&&&&&&&\\
3 & & && \bullet \ar@{->}[dr] \ar@{->}[uurr] && && \bullet\ar@{->}[dr] \ar@{->}[uurr]  && && \bullet\ar@{->}[dr] \ar@{->}[uurr]  && && \\
4& \bigstar\ar@{->}[dr]&&\bigstar\ar@{->}[ur]&&\bigstar\ar@{->}[dr]& &\bigstar\ar@{->}[ur]&&\bigstar\ar@{->}[dr]& &\bigstar\ar@{->}[ur]&&\bigstar\ar@{->}[dr]& &&\\
5 && \bullet\ar@{->}[ddrr] \ar@{->}[ur]  && && \bullet\ar@{->}[ddrr] \ar@{->}[ur]  && && \bullet \ar@{->}[ddrr] \ar@{->}[ur] && && \bullet\ar@{->}[ddrr]   &&\\
&&&&&&&&&&&&&&\\
6 && && \bullet\ar@{->}[ddrr] \ar@{->}[uurr]  && && \bullet\ar@{->}[ddrr] \ar@{->}[uurr]  && && \bullet\ar@{->}[ddrr] \ar@{->}[uurr]  && && \bullet  \\
 &&&&&& &&&&&&&&&&\\
7 && \bullet \ar@{->}[uurr]  && && \bullet \ar@{->}[uurr]  && && \bullet  \ar@{->}[uurr] && && \bullet  \ar@{->}[uurr] &&
}}} \hspace{-4ex} {\scriptstyle \Upsilon_{[Q^>]} } = \raisebox{4em}{ \scalebox{0.5}{\xymatrix@C=2ex@R=1ex{
 & 1 & 1\frac{1}{2} &2 & 2\frac{1}{2} &3 & 3\frac{1}{2} &4 & 4\frac{1}{2} &5 & 5\frac{1}{2} &
6 & 6\frac{1}{2} &7 & 7\frac{1}{2} &8 \\
1 & &&  \bullet\ar@{->}[ddrr] && &&  \bullet\ar@{->}[ddrr]  && && \bullet\ar@{->}[ddrr]  && && \\
& &&&&& &&&& &&&& &&\\
2 & && &&\bullet\ar@{->}[ddrr] \ar@{->}[uurr]  && && \bullet \ar@{->}[ddrr]\ar@{->}[uurr] && && \bullet &&  \\
&&&&&& &&&&&&&&\\
3  & && \bullet \ar@{->}[dr] \ar@{->}[uurr] && && \bullet\ar@{->}[dr] \ar@{->}[uurr]  && && \bullet\ar@{->}[dr] \ar@{->}[uurr]  && && \\
4 &&\bigstar\ar@{->}[ur]&&\bigstar\ar@{->}[dr]& &\bigstar\ar@{->}[ur]&&\bigstar\ar@{->}[dr]& &\bigstar\ar@{->}[ur]&&\bigstar\ar@{->}[dr]& &  \bigstar&\\
5 & \bullet\ar@{->}[ddrr] \ar@{->}[ur]  && && \bullet\ar@{->}[ddrr] \ar@{->}[ur]  && && \bullet \ar@{->}[ddrr] \ar@{->}[ur] && && \bullet\ar@{->}[ddrr] \ar@{->}[ur]  &&\\
&&&&&&&&&&&&&&\\
6 & && \bullet\ar@{->}[ddrr] \ar@{->}[uurr]  && && \bullet\ar@{->}[ddrr] \ar@{->}[uurr]  && && \bullet\ar@{->}[ddrr] \ar@{->}[uurr]  && && \bullet  \\
 &&&&&& &&&&&&&&&&\\
7 & \bullet \ar@{->}[uurr]  && && \bullet \ar@{->}[uurr]  && && \bullet  \ar@{->}[uurr] && && \bullet  \ar@{->}[uurr] &&
}}}
 \]
\end{example}

Note that the coordinate system is unique up to constant. Furthermore, if we choose $\vee$ as an identity, the
coordinate system on $\Upsilon_{[Q]}$ is exactly the same as the original one of $\Gamma_Q$.

For $\vee$ in Section \ref{Sec:twisted Coxeter}, lengths of arrows in a twisted AR-quiver $\Upsilon_{[\ii_0]}$ and $\mathsf{d}$ in \eqref{def: d} are given as follows:
\begin{align} \label{eq: length of arrow}
\ell(\mathsf{a})= \begin{cases}
1 \text{ or } 1/2 & \text{ if $\vee$ is \eqref{eq: B_n} or \eqref{eq: F_4}},\\
1/2 & \text{ if $\vee$ is \eqref{eq: C_n}},\\
1/3 & \text{ if $\vee$ is \eqref{eq: G_2}},
\end{cases}
\ \text{ and } \mathsf{d} = \begin{cases}  1 & \text{ if $\vee$ is {\rm id}}, \\
2 & \text{ if $\vee$ is \eqref{eq: B_n}, \eqref{eq: C_n} or \eqref{eq: F_4}},\\  3 & \text{ if $\vee$ is \eqref{eq: G_2}}. \end{cases}
\end{align}

\subsection{Folded AR-quivers}
Now the following lemma tells that a (triply) twisted AR-quiver $\Upsilon_{[\ii_0]}$ for $[\ii_0]\in \lf \Qd\rf$ or $\lf \mathfrak{Q} \rf$ is \defn{foldable} in the following sense:
\begin{eqnarray}&&
\parbox{85ex}{
For distinct vertices $v,w$ in $\Upsilon_{[\ii_0]}$ whose coordinates are $(i,p)$ and $(j,q)$,
$(\widehat{i},p)$ and $(\widehat{j},q) \in \widehat{I} \times \frac{1}{\mathsf{d}}\Z$ are also distinct.
 }\label{eq: foldable}
\end{eqnarray}

\begin{lemma}\label{lem: foldable}
A $($triply$)$ twisted AR-quiver  $\Upsilon_{[\ii_0]}$ is foldable.
\end{lemma}

\begin{proof} Let $(i,p)$ and $(j,q)$ be  coordinates of $v$ and $w$ of $\Upsilon_{[\ii_0]}$, respectively.

(1) Let $\ii_0$ be of type $A_{2n+1}$ and $\vee$ be the one in \eqref{eq: B_n}. By the surgery in Algorithm \ref{Rem:surgery A}, if $j=2n+2-i$, the parity of $p$ and $q$ are different and hence
our assertion follows.

(2) For $\ii_0$ of type $D_{n+1}$ and  $\vee$ in \eqref{eq: C_n}, our assertion is obvious from the surgery in Algorithm \ref{Alg surgery D}.

The remained exceptional cases can be checked directly.
\end{proof}

We call the $(\widehat{i},p)\in \widehat{I} \times \frac{1}{\mathsf{d}}\Z$ in~\eqref{eq: foldable} the \defn{folded coordinate} of $v$.
Now we denote by $\widehat{\Upsilon}_{[\ii_0]}$ when we assign the folded coordinates system to the twisted AR-quiver $\Upsilon_{[\ii_0]}$
and call it the \defn{folded AR-quiver}.



\begin{example} \label{eq: folded 4} (1) The folded AR-quiver $\widehat{\Upsilon}_{[Q^<]}$ of $\Upsilon_{[Q^<]}$ in Example \ref{Ex: twisted Coxeter} can be drawn as follows:
$$ \ \ \ \scalebox{0.73}
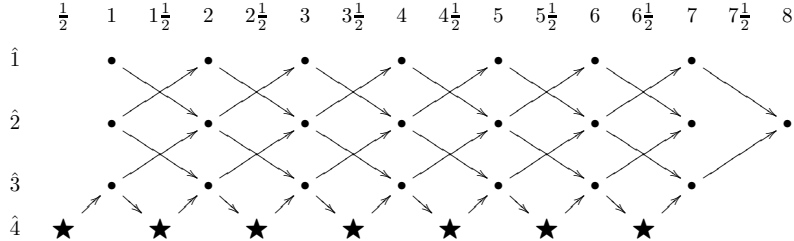
{\xymatrix@C=2ex@R=1ex{
& \frac{1}{2} & 1 & 1\frac{1}{2} &2 & 2\frac{1}{2} &3 & 3\frac{1}{2} &4 & 4\frac{1}{2} &5 & 5\frac{1}{2} &
6 & 6\frac{1}{2} &7 & 7\frac{1}{2} &8 \\
\hat{1} && \bullet\ar@{->}[ddrr] &&  \bullet\ar@{->}[ddrr] && \bullet\ar@{->}[ddrr] &&  \bullet\ar@{->}[ddrr]  && \bullet\ar@{->}[ddrr] && \bullet\ar@{->}[ddrr]  && \bullet\ar@{->}[ddrr] && \\
& &&&&& &&&& &&&& &&\\
\hat{2} && \bullet\ar@{->}[uurr] \ar@{->}[ddrr] && \bullet\ar@{->}[uurr] \ar@{->}[ddrr]  &&\bullet\ar@{->}[ddrr] \ar@{->}[uurr]  && \bullet\ar@{->}[uurr] \ar@{->}[ddrr] && \bullet \ar@{->}[ddrr]\ar@{->}[uurr] &&
\bullet\ar@{->}[uurr] \ar@{->}[ddrr] && \bullet && \bullet  \\
&&&&&& &&&&&&&&\\
\hat{3} && \bullet\ar@{->}[uurr] \ar@{->}[dr] && \bullet \ar@{->}[dr] \ar@{->}[uurr] &&\bullet\ar@{->}[uurr] \ar@{->}[dr]&& \bullet\ar@{->}[dr] \ar@{->}[uurr]  &&\bullet \ar@{->}[uurr] \ar@{->}[dr] && \bullet\ar@{->}[dr] \ar@{->}[uurr]
 && \bullet\ar@{->}[uurr] && \\
\hat{4}& \bigstar\ar@{->}[ur]&&\bigstar\ar@{->}[ur]&&\bigstar\ar@{->}[ur]& &\bigstar\ar@{->}[ur]&&\bigstar\ar@{->}[ur]& &\bigstar\ar@{->}[ur]&&\bigstar\ar@{->}[ur]& &&}}
$$

(2) The folded AR-quiver $\widehat{\Upsilon}_{[\ii_0]}$ of $\Upsilon_{[\ii_0]}$ in Example \ref{ex: Twisted D first} is given as follows:
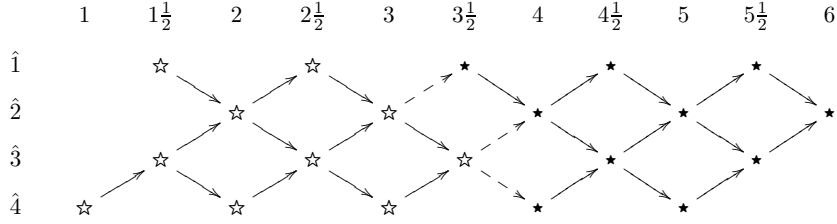
\begin{equation*}
 \scalebox{0.8}{\xymatrix@C=4ex@R=1ex{
 & 1 & 1\frac{1}{2} &2 & 2\frac{1}{2} &3 & 3\frac{1}{2} &4 & 4\frac{1}{2} &5 & 5\frac{1}{2} &
6  \\
\hat{1}&&\medstar\ar@{->}[dr]  && \medstar\ar@{->}[dr]  &&  \filledstar \ar@{->}[dr]  && \filledstar\ar@{->}[dr] && \filledstar \ar@{->}[dr]\\
\hat{2}&&& \medstar\ar@{->}[dr]\ar@{->}[ur]  &&\medstar\ar@{->}[dr]\ar@{-->}[ur]  && \filledstar\ar@{->}[dr]\ar@{->}[ur]  && \filledstar\ar@{->}[dr]\ar@{->}[ur] && \filledstar \\
\hat{3}& & \medstar\ar@{->}[dr]\ar@{->}[ur]  &&\medstar\ar@{->}[dr]\ar@{->}[ur]  &&\medstar\ar@{-->}[dr]\ar@{-->}[ur]  && \filledstar\ar@{->}[dr]\ar@{->}[ur] && \filledstar\ar@{->}[ur] \\
\hat{4}& \medstar\ar@{->}[ur]&& \medstar\ar@{->}[ur]  &&\medstar\ar@{->}[ur] && \filledstar \ar@{->}[ur] && \filledstar \ar@{->}[ur]
}}
\end{equation*}

(3)  The folded AR-quiver $\widehat{\Upsilon}_{[\ii^\natural_0]}$ of $\Upsilon_{[\ii^\natural_0]}$ in Example \ref{ex:E6 twist} is given as follows:
 $$\scalebox{0.58}
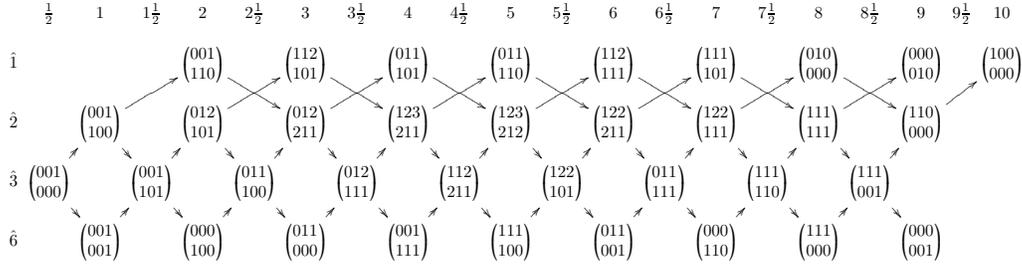
{\xymatrix@C=0.1ex@R=1.3ex{
 & \frac{1}{2} & 1 & 1\frac{1}{2} &2 & 2\frac{1}{2} &3 & 3\frac{1}{2} &4 & 4\frac{1}{2} &5 & 5\frac{1}{2} &
6 & 6\frac{1}{2} &7 & 7\frac{1}{2} &8 & 8\frac{1}{2} & 9 & 9\frac{1}{2}  & 10 \\
\hat{1} &&&& {\scriptstyle\prt{001}{110}} \ar@{->}[drr] &&{\scriptstyle\prt{112}{101}}\ar@{->}[drr]&& {\scriptstyle\prt{011}{101}}\ar@{->}[drr] && {\scriptstyle\prt{011}{110}} \ar@{->}[drr]
&& {\scriptstyle\prt{112}{111}}\ar@{->}[drr]
&& {\scriptstyle\prt{111}{101}} \ar@{->}[drr]
&& {\scriptstyle\prt{010}{000}} \ar@{->}[drr] &&
{\scriptstyle\prt{000}{010}} && {\scriptstyle\prt{100}{000}} \\
\hat{2} && {\scriptstyle\prt{001}{100}} \ar@{->}[urr]\ar@{->}[dr] &&  {\scriptstyle\prt{012}{101}}\ar@{->}[urr]\ar@{->}[dr]   && {\scriptstyle\prt{012}{211}}\ar@{->}[urr]\ar@{->}[dr] && {\scriptstyle\prt{123}{211}}\ar@{->}[urr]\ar@{->}[dr]
&&{\scriptstyle\prt{123}{212}} \ar@{->}[urr]\ar@{->}[dr]
&& {\scriptstyle\prt{122}{211}}\ar@{->}[urr]\ar@{->}[dr]
&&{\scriptstyle\prt{122}{111}} \ar@{->}[urr]\ar@{->}[dr] &&
{\scriptstyle\prt{111}{111}} \ar@{->}[urr]\ar@{->}[dr] && {\scriptstyle\prt{110}{000}}\ar@{->}[urr]\\
\hat{3} & {\scriptstyle\prt{001}{000}}\ar@{->}[ur]\ar@{->}[dr] && {\scriptstyle\prt{001}{101}} \ar@{->}[dr]\ar@{->}[ur] && {\scriptstyle\prt{011}{100}} \ar@{->}[ur]\ar@{->}[dr]
&& {\scriptstyle\prt{012}{111}} \ar@{->}[dr]\ar@{->}[ur] && {\scriptstyle\prt{112}{211}}\ar@{->}[ur]\ar@{->}[dr] &&{\scriptstyle\prt{122}{101}}\ar@{->}[ur]\ar@{->}[dr]
&&{\scriptstyle\prt{011}{111}}\ar@{->}[ur]\ar@{->}[dr] &&
{\scriptstyle\prt{111}{110}}\ar@{->}[dr]\ar@{->}[ur] && {\scriptstyle\prt{111}{001}}\ar@{->}[ur]\ar@{->}[dr]\\
\hat{6} && {\scriptstyle\prt{001}{001}}\ar@{->}[ur] && {\scriptstyle\prt{000}{100}} \ar@{->}[ur] && {\scriptstyle\prt{011}{000}} \ar@{->}[ur] &&
{\scriptstyle\prt{001}{111}}\ar@{->}[ur] &&{\scriptstyle\prt{111}{100}}\ar@{->}[ur] && {\scriptstyle\prt{011}{001}} \ar@{->}[ur]&& {\scriptstyle\prt{000}{110}}
\ar@{->}[ur] &&{\scriptstyle\prt{111}{000}}\ar@{->}[ur]  &&{\scriptstyle\prt{000}{001}}
}}
$$

(4) The folded AR-quiver $\widehat{\Upsilon}_{[\ii^\dagger_0]}$ of $\Upsilon_{[\ii^\dagger_0]}$ in Example \ref{ex:D4 triply twist} is given as follows:
$$\scalebox{0.67}
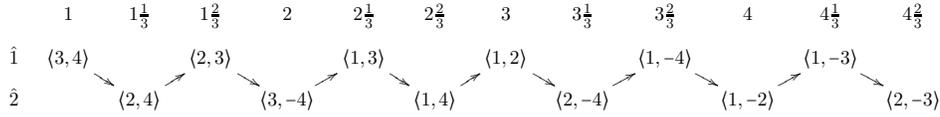
{\xymatrix@C=2ex@R=1ex{
 & 1 & 1\frac{1}{3} & 1\frac{2}{3} &2 & 2\frac{1}{3} &2\frac{2}{3} &3 & 3\frac{1}{3} & 3\frac{2}{3} &4 & 4\frac{1}{3} &4\frac{2}{3}  \\
\hat{1} &\lan 3, 4 \ran \ar@{->}[dr]  && \lan 2, 3 \ran \ar@{->}[dr]  && \lan 1, 3 \ran \ar@{->}[dr] &&\lan 1, 2 \ran \ar@{->}[dr]  &&\lan 1, -4 \ran \ar@{->}[dr]  && \lan 1, -3 \ran \ar@{->}[dr]\\
\hat{2} && \lan 2, 4 \ran \ar@{->}[ur] && \lan 3, -4 \ran\ar@{->}[ur] && \lan 1, 4 \ran \ar@{->}[ur]&& \lan 2, -4 \ran \ar@{->}[ur]&& \lan 1, -2 \ran \ar@{->}[ur] && \lan 2, -3 \ran
}}
$$
\end{example}

Now we can describe the algorithm which shows a way of obtaining $\widehat{\Upsilon}_{[\ii_0]r_i}$ from
$\widehat{\Upsilon}_{[\ii_0]}$ by using the notations on $\widehat{\Delta}$ which is almost  same as Algorithm \ref{alg: Ref Q}.

\begin{algorithm} \label{alg: fRef Q}
Let $\mathsf{h}^\vee$ be a dual Coxeter number of type $\widehat{X}$ and $\al_i$ $(i \in I)$ be a sink of $\widehat{\Upsilon}_{[\ii_0]}$ for $\lf \Qd \rf$.
\begin{enumerate}
\item[{\rm (A1)}] Remove the vertex $(\hat{i},p)$ such that $\widehat{\Omega}_{[\ii_0]}(\al_i)=(\hat{i},p)$ and the arrows adjacent to $(\hat{i},p)$.
\item[{\rm (A2)}] Add the vertex $(\hat{i},p- \mathsf{h}^\vee)$ and
the arrows to all $(\hat{j}, p-\mathsf{h}^\vee+\min(|\al_{\hat{i}}|^2,|\al_{\hat{j}}|^2)) \in \widehat{\Upsilon}_{[\ii_0]}$,
for $\hat{j}$ adjacent to $\hat{i}$ in $\widehat{\Delta}$.
\item[{\rm (A3)}] Label the vertex $(\hat{i},p- \mathsf{h}^\vee)$ with $\al_i$ and change the labels $\be$ to $s_i(\be)$ for all $\be \in \widehat{\Upsilon}_{[\ii_0]}
\setminus \{\al_i\}$.
\end{enumerate}
\end{algorithm}

\section{Labeling of a twisted AR-quiver} \label{Sec:label_twistedAR}

Basically, labels of combinatorial AR-quivers can be obtained by iterative computations, using~\eqref{compatible reading}.
In this section,  when $[\ii_0]$ is a twisted adapted class of type $A_{2n+1}$ or $D_{n+1}$,  we shall show that the shape of $\Upsilon_{[\ii_0]}$ completely determines the labels, without computations.

\subsection{Type $A_{2n+1}$} Recall that
a twisted AR-quiver $\Upsilon_{[\ii_0]}$ of type $A_{2n+1}$ can be constructed from some AR-quiver $\Gamma_Q$ with the surgery in Algorithm \ref{Rem:surgery A}.
Thus the full subquiver of $\Upsilon_{[\ii_0]}$ consisting of all vertices whose residues are $\{ n, n+1,n+2 \}$ can be classified as follows:

\vskip -1.5em

\begin{align}\label{eq: four situ}
\begin{cases}
\raisebox{1.2em}{\scalebox{0.5}{\xymatrix@C=2ex@R=1ex{
n &  &  &&\bullet \ar@{->}[dr]&&&& \cdots\cdots &&&&\bullet\ar@{->}[dr] \\
n+1 & \bigstar\ar@{->}[dr] && \bigstar \ar@{->}[ur] && \bigstar \ar@{->}[dr] &&&&&& \bigstar\ar@{->}[ur]&&\bigstar\ar@{->}[dr]\\
n+2 & & \bullet \ar@{->}[ur] &&&& \bullet && \cdots\cdots  && \bullet \ar@{->}[ur]&&&&\bullet \\
}}} & \text{(1) } \ \Upsilon_{[Q^<]} \ \text{ for } \ Q \ \text{ with the arrow }  \ {\xymatrix@R=3ex{ \bullet \ar@{-}[r]_<{ n \ } &  \bullet \ar@{<-}[l]^<{ \ n+1} }},  \\
\raisebox{1.2em}{\scalebox{0.5}{\xymatrix@C=2ex@R=1ex{
n   &&&\bullet \ar@{->}[dr]&&&& \cdots\cdots &&&&\bullet\ar@{->}[dr] \\
n+1 && \bigstar \ar@{->}[ur] && \bigstar \ar@{->}[dr] &&&&&& \bigstar\ar@{->}[ur]&&\bigstar\ar@{->}[dr] && \bigstar \\
n+2 & \bullet \ar@{->}[ur] &&&& \bullet && \cdots\cdots  && \bullet \ar@{->}[ur]&&&&\bullet \ar@{->}[ur] \\
}}} & \text{(2) } \ \Upsilon_{[Q^>]} \ \text{ for } \ Q \ \text{  with the arrow }  \ {\xymatrix@R=3ex{ \bullet \ar@{-}[r]_<{ n \ } &  \bullet \ar@{<-}[l]^<{ \ n+1} }},\\
\raisebox{1.2em}{\scalebox{0.5}{\xymatrix@C=2ex@R=1ex{
n   &&\bullet \ar@{->}[dr]&&&& \cdots\cdots &&&&\bullet\ar@{->}[dr] &&&&\bullet \\
n+1 & \bigstar \ar@{->}[ur] && \bigstar \ar@{->}[dr] &&&&&& \bigstar\ar@{->}[ur]&&\bigstar\ar@{->}[dr] && \bigstar\ar@{->}[ur] \\
n+2 &&&& \bullet && \cdots\cdots  && \bullet \ar@{->}[ur]&&&&\bullet \ar@{->}[ur] \\
}}} & \text{(3) } \ \Upsilon_{[Q^<]} \ \text{ for } \ Q \ \text{ with the arrow }  \ {\xymatrix@R=3ex{ \bullet \ar@{<-}[r]_<{ n \ } &  \bullet \ar@{-}[l]^<{ \ n+1} }},\\
\raisebox{1.2em}{\scalebox{0.5}{\xymatrix@C=2ex@R=1ex{
n   &\bullet \ar@{->}[dr]&&&& \cdots\cdots &&&&\bullet\ar@{->}[dr] &&&&\bullet \ar@{->}[dr]\\
n+1 && \bigstar \ar@{->}[dr] &&&&&& \bigstar\ar@{->}[ur]&&\bigstar\ar@{->}[dr] && \bigstar\ar@{->}[ur] && \bigstar \\
n+2 &&& \bullet && \cdots\cdots  && \bullet \ar@{->}[ur]&&&&\bullet \ar@{->}[ur] \\
}}} & \text{(4) } \ \Upsilon_{[Q^>]} \ \text{ for } \ Q \ \text{ with the arrow }  \ {\xymatrix@R=3ex{ \bullet \ar@{<-}[r]_<{ n \ } &  \bullet \ar@{-}[l]^<{ \ n+1} }}.
\end{cases}
\end{align}

\begin{remark} By the surgery, we know that
(1) for each $k \ge 1$, the $N$-path(resp. $S$-path) with $k$-arrows in $\Upsilon_{[\ii_0']}$ is unique, if it exists, (2)  an $N$-path (resp. $S$-path) consisting of $k$-arrows exists only if $k$ is one of the followings:
\begin{eqnarray} &&
  \parbox{95ex}{
\begin{itemize}
\item[{\rm (a)}] $1,\ldots,n-1,n+1,n+2,\ldots,2n$
(resp. $1,\ldots,n-2,n,n+1\ldots,2n$) in (1) or (4) of \eqref{eq: four situ},
\item[{\rm (b)}] $1,\ldots,n-2,n,n+1,\ldots,2n$
(resp. $1,\ldots,n-1,n+1,n+2,\ldots,2n$) in (2) or (3) of \eqref{eq: four situ}.
\end{itemize}
}\label{eq: total non-total}
\end{eqnarray}
Recall the notions induced and non-induced vertices in Remark \ref{Rem:4.16_0822}.
Note that a sectional path $\rho$ with $k(\ge n)$-arrows contains a non-induced vertex; that is, ${\scriptstyle\bigstar} \in \rho$.
Also, a sectional path $\rho'$ with $k(< n)$-arrows do not contains a non-induced vertex; that is, ${\scriptstyle\bigstar} \not\in \rho'$.
\end{remark}

\begin{proposition}  \label{prop: comp for length k ge n}
Let $k\geq n$.
\begin{enumerate}
\item[{\rm (1)}] Every vertex in an $N$-path with $k$-arrows has $2n+1-k$ as the first component.
\item[{\rm (2)}] Every vertex in an $S$-path with $k$-arrows has $k+1$ as the second component.
\end{enumerate}
\end{proposition}

\begin{proof}
Note that we have a maximal $N$-path with $2n$-arrows. By Proposition \ref {pro: section shares}, its first component should be $1$.
Since we exhaust all positive roots of the form $[1,*]$, we can apply the same argument for $[2,*]$, $[3,*]$,..., sequentially. The second assertion follows in the same way.
\end{proof}

Recall the notions in Definition \ref{Def:4.15_0822} and Remark \ref{Rem:4.16_0822} to classify vertices in $\Upsilon_{[\ii_0]}$ as follows.

\begin{definition} \label{def: central}
Fix any class $[\ii_0]$ in $\lf \Qd \rf$ of type $A_{2n+1}$ such that $\PPi([\ii_0])=[Q]$.
\begin{enumerate}
\item[{\rm (a)}] A vertex $v$ is a \defn{central vertex} of $\Upsilon_{[\ii_0]}$ (i) if it is not induced, that is, $v={\scriptstyle\bigstar} \in \Upsilon_{[\ii_0]} \setminus \Gamma_Q$
or (ii) if it is induced and it is the intersection of {\it two} sectional paths with ${\scriptstyle\bigstar}$'s.
\item[{\rm (b)}] The full subquiver $\Upsilon^{{\rm C}}_{[\ii_0]}$ of $\Upsilon_{[\ii_0]}$ consisting of all central vertices is called the
\defn{center} of $\Upsilon_{[\ii_0]}$.
\item[{\rm (c)}] The full subquiver $\uUp^{{\rm NE}}_{[\ii_0]}$ (resp. $\uUp^{{\rm SE}}_{[\ii_0]}$,$\uUp^{{\rm NW}}_{[\ii_0]}$ and $\uUp^{{\rm SW}}_{[\ii_0]}$)
of $\Upsilon_{[\ii_0]}$ consists of all vertices which are not contained in
$\Upsilon^{{\rm C}}_{[\ii_0]}$ and located in the North-East (resp. South-East, North-West and South-West) part of $\Upsilon_{[\ii_0]}$.
\end{enumerate}
\end{definition}

\begin{example} For $[\ii_0]=[Q^<]$ in Example (\ref{ex: Q<Q>}), we can decompose $\Upsilon_{[\ii_0]}$ into
$\uUp^{{\rm NE}}_{[\ii_0]}(\heartsuit)$, $\uUp^{{\rm SE}}_{[\ii_0]}(\square)$,$\uUp^{{\rm NW}}_{[\ii_0]}(\diamondsuit)$, $\uUp^{{\rm SW}}_{[\ii_0]}(\triangle)$ and
$\Upsilon^{{\rm C}}_{[\ii_0]}(\bullet,{\scriptstyle\bigstar})$ as follows:
$$  \scalebox{0.6}
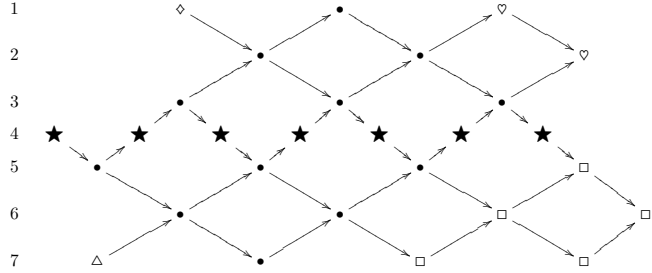
{\xymatrix@C=2ex@R=1ex{
1 &&&&  \diamondsuit\ar@{->}[ddrr] && &&  \bullet\ar@{->}[ddrr]  && && \heartsuit\ar@{->}[ddrr]  && && \\
& &&&&& &&&& &&&& &&\\
2 && && &&\bullet\ar@{->}[ddrr] \ar@{->}[uurr]  && && \bullet \ar@{->}[ddrr]\ar@{->}[uurr] && && \heartsuit &&  \\
&&&&&& &&&&&&&&\\
3 & & && \bullet \ar@{->}[dr] \ar@{->}[uurr] && && \bullet\ar@{->}[dr] \ar@{->}[uurr]  && && \bullet\ar@{->}[dr] \ar@{->}[uurr]  && && \\
4& \bigstar\ar@{->}[dr]&&\bigstar\ar@{->}[ur]&&\bigstar\ar@{->}[dr]& &\bigstar\ar@{->}[ur]&&\bigstar\ar@{->}[dr]& &\bigstar\ar@{->}[ur]&&\bigstar\ar@{->}[dr]& &&\\
5 && \bullet\ar@{->}[ddrr] \ar@{->}[ur]  && && \bullet\ar@{->}[ddrr] \ar@{->}[ur]  && && \bullet \ar@{->}[ddrr] \ar@{->}[ur] && && \square\ar@{->}[ddrr]   &&\\
&&&&&&&&&&&&&&\\
6 && && \bullet\ar@{->}[ddrr] \ar@{->}[uurr]  && && \bullet\ar@{->}[ddrr] \ar@{->}[uurr]  && && \square\ar@{->}[ddrr] \ar@{->}[uurr]  && && \square  \\
 &&&&&& &&&&&&&&&&\\
7 && \triangle \ar@{->}[uurr]  && && \bullet \ar@{->}[uurr]  && && \square  \ar@{->}[uurr] && && \square  \ar@{->}[uurr] &&
}}
$$
\end{example}

By Theorem \ref{thm: OS14}, we can get a reduced word $\ii_0' \in [\ii_0]$ by reading residues of vertices in $\Upsilon_{[\ii_0]}$
by the following order
\begin{align} \label{eq: order}
\{ \uUp^{{\rm NE}}_{[\ii_0]}, \ \uUp^{{\rm SE}}_{[\ii_0]} \}, \ \{ \Upsilon^{{\rm C}}_{[\ii_0]} \} \text{ and } \{ \uUp^{{\rm NW}}_{[\ii_0]}, \uUp^{{\rm SW}}_{[\ii_0]} \}.
\end{align}

Note that
\begin{eqnarray} &&
  \parbox{95ex}{
\begin{enumerate}
\item[{\rm (a)}] all vertices  in $\uUp^{{\rm NE}}_{[\ii_0]}$ and $\uUp^{{\rm NW}}_{[\ii_0]}$ have residues less than or equal to $n$,
\item[{\rm (b)}] all vertices  in $\uUp^{{\rm SE}}_{[\ii_0]}$ and $\uUp^{{\rm SW}}_{[\ii_0]}$ have residues larger than or equal to $n+2$,
\item[{\rm (c)}] $\uUp^{{\rm NE}}_{[\ii_0]}$,$\uUp^{{\rm NW}}_{[\ii_0]}$, $\uUp^{{\rm SE}}_{[\ii_0]}$, $\uUp^{{\rm SW}}_{[\ii_0]} \subset \Gamma_Q \cap \Upsilon_{[\ii_0]}$
where $\PPi([\ii_0])=[Q]$.
\end{enumerate}
}\label{eq: from Q labelling}
\end{eqnarray}

By Algorithm \ref{Rem:surgery A}, Theorem \ref{thm: labeling GammaQ} and \eqref{eq: from Q labelling}, we have the following lemma:

\begin{lemma} \label{lem: comp for length k le n} For $v\in \uUp^{{\rm NE}}_{[\ii_0]}(\heartsuit) \sqcup \uUp^{{\rm SW}}_{[\ii_0]}(\triangle)$ and
$v' \in \uUp^{{\rm SE}}_{[\ii_0]}(\square) \sqcup \uUp^{{\rm NW}}_{[\ii_0]}(\diamondsuit)$, we have
\begin{enumerate}
\item[{\rm (1)}] $v$ is labeled by $[a,b]$ $(b \le n)$ which is the same as the labeling $[a,b]$ of $v$ in $\Gamma_Q$,
\item[{\rm (2)}] $v'$ is labeled by $[a+1,b+1]$ $(a \ge n+1)$ where the labeling of $v$ in $\Gamma_Q$ is $[a,b]$.
\end{enumerate}
\end{lemma}

\begin{proof}
(1) By reading $\uUp^{{\rm NE}}_{[\ii_0]}$ first in \eqref{eq: order}, the labeling for $v$ in $\Upsilon_{[\ii_0]}$ should be the same as that in $\Gamma_Q$.

(2) By reading $\uUp^{{\rm SE}}_{[\ii_0]}$ first in \eqref{eq: order}, the labeling for $v'$ in $\Upsilon_{[\ii_0]}$ should be shifted by one.

The remained assertions follow by considering $\ii_0^{\rm rev}$ where $\ii_0^{\rm rev}=i_{l}i_{l-1} \cdots i_1$ for $\ii_0=i_1 i_2 \cdots i_{l}$.
\end{proof}

By Lemma \ref{lem: comp for length k le n} and Proposition \ref{prop: comp for length k ge n}, we have the following theorem:
\begin{theorem} \label{them: comp for length k ge 0} For every vertex in $\Upsilon_{[\ii_0]}$, we can label it as $[a,b] \in \Phi^+$ for some $1 \le a \le b \le 2n+1$  without computing like~\eqref{compatible reading}.
As consequences, we have
\begin{enumerate}
\item[{\rm (1)}] every induced $N$-path with $k$-arrows shares $2n+1-k$ as the first component,
\item[{\rm (2)}] every induced $S$-path with $k$-arrows shares $k+1$ as the second component.
\end{enumerate}
Hence, for every vertex in $\Upsilon_{[\ii_0]}$, we can label it as $[a,b] \in \Phi^+$ for some $1 \le a \le b \le 2n+1$.
\end{theorem}

\begin{proof}
Every induced central vertex $\bullet$ in $\Upsilon^{{\rm C}}_{[\ii_0]}$ is located at the intersection
of two maximal induced (but not totally induced) sectional paths with more that $n$-arrows and hence we can label them as $[a,b]$ for some $1 \le a \le b \le 2n+1$
by Proposition \ref{prop: comp for length k ge n}. The vertices in $\Upsilon_{[\ii_0]} \setminus \Upsilon^{{\rm C}}_{[\ii_0]}$
can be labeled by Lemma \ref{lem: comp for length k le n} and Theorem \ref{thm: labeling GammaQ}. Then only vertices ${\scriptstyle \bigstar}$
are not labeled completely; that is $[a,*]$ or $[*,b]$. Due to the system $\Phi^+$, we can label them ${\scriptstyle \bigstar}$ completely.
\end{proof}

\begin{example} \label{ex: label}
For a Dynkin quiver $Q = {\xymatrix@R=3ex{ \circ
\ar@{->}[r]_<{ \ 1} &  \circ
\ar@{<-}[r]_<{ \ 2}  &  \circ
\ar@{->}[r]_<{ \ 3} &\circ \ar@{->}[r]_<{ \ 4}
&\circ \ar@{<-}[r]_<{ \ 5}
& \circ \ar@{-}[l]^<{\ \ \ \ \ \ 6} }},$
let us consider the combinatorial AR-quiver for $\Upsilon_{[Q^<]}$
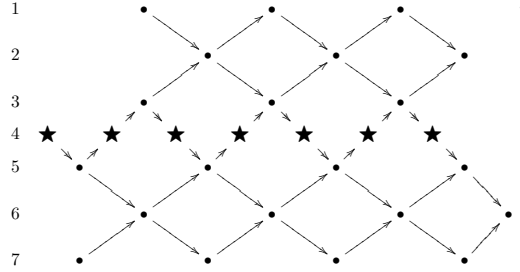
$$ \scalebox{0.6}{\xymatrix@C=1ex@R=1ex{
1 && &&  \bullet\ar@{->}[ddrr] && &&  \bullet\ar@{->}[ddrr]  && && \bullet\ar@{->}[ddrr]  && && \\
& &&&&& &&&& &&&& &&\\
2 && && &&\bullet\ar@{->}[ddrr] \ar@{->}[uurr]  && && \bullet \ar@{->}[ddrr]\ar@{->}[uurr] && && \bullet &&  \\
&&&&&& &&&&&&&&\\
3 & & && \bullet \ar@{->}[dr] \ar@{->}[uurr] && && \bullet\ar@{->}[dr] \ar@{->}[uurr]  && && \bullet\ar@{->}[dr] \ar@{->}[uurr]  && && \\
4& \bigstar\ar@{->}[dr]&&\bigstar\ar@{->}[ur]&&\bigstar\ar@{->}[dr]& &\bigstar\ar@{->}[ur]&&\bigstar\ar@{->}[dr]& &\bigstar\ar@{->}[ur]&&\bigstar\ar@{->}[dr]& &&\\
5 && \bullet\ar@{->}[ddrr] \ar@{->}[ur]  && && \bullet\ar@{->}[ddrr] \ar@{->}[ur]  && &&\bullet \ar@{->}[ddrr] \ar@{->}[ur] && && \bullet\ar@{->}[ddrr]   &&\\
&&&&&&&&&&&&&&\\
6 && && \bullet\ar@{->}[ddrr] \ar@{->}[uurr]  && && \bullet\ar@{->}[ddrr] \ar@{->}[uurr]  && && \bullet\ar@{->}[ddrr] \ar@{->}[uurr]  && &&\bullet  \\
 &&&&&& &&&&&&&&&&\\
7 && \bullet \ar@{->}[uurr]  && && \bullet \ar@{->}[uurr]  && && \bullet \ar@{->}[uurr] && &&\bullet \ar@{->}[uurr]
}}.$$
which is the case {\rm (1)} in \eqref{eq: four situ}. By  Theorem \ref{them: comp for length k ge 0},
we can complete finding labels for $\Upsilon_{[Q^<]}$ in three steps as follows:
\begin{align*}
& \scalebox{0.48}{\xymatrix@C=1ex@R=1ex{
1 && &&  [7]\ar@{->}[ddrr] && &&  [*,6]\ar@{->}[ddrr]  && && [*,2]\ar@{->}[ddrr]  && && \\
& &&&&& &&&& &&&& &&\\
2 && && &&[*,7]\ar@{->}[ddrr] \ar@{->}[uurr]  && && [*,6] \ar@{->}[ddrr]\ar@{->}[uurr] && && [2] &&  \\
&&&&&& &&&&&&&&\\
3 & & && [*,5] \ar@{->}[dr] \ar@{->}[uurr] && && [*,7]\ar@{->}[dr] \ar@{->}[uurr]  && && [*,6]\ar@{->}[dr] \ar@{->}[uurr]  && && \\
4& [*,4] \ar@{->}[dr]&& \bigstar \ar@{->}[ur]&&[*,5]\ar@{->}[dr]& & \bigstar \ar@{->}[ur]&&[*,7]\ar@{->}[dr]& &\bigstar \ar@{->}[ur]&&[*,6]\ar@{->}[dr]& &&\\
5 && [*,4]\ar@{->}[ddrr] \ar@{->}[ur]  && && [*,5]\ar@{->}[ddrr] \ar@{->}[ur]  && &&[*,7] \ar@{->}[ddrr] \ar@{->}[ur] && && [*,6]\ar@{->}[ddrr]   &&\\
&&&&&&&&&&&&&&\\
6 && && [*,4]\ar@{->}[ddrr] \ar@{->}[uurr]  && && [*,5]\ar@{->}[ddrr] \ar@{->}[uurr]  && && [*,7]\ar@{->}[ddrr] \ar@{->}[uurr]  && &&[*,6]  \\
 &&&&&& &&&&&&&&&&\\
7 && [*,1] \ar@{->}[uurr]  && && [*,4] \ar@{->}[uurr]  && && [*,5] \ar@{->}[uurr] && &&[*,7] \ar@{->}[uurr] &&
}}
\hspace{-5ex}
 \scalebox{0.48}{\xymatrix@C=1ex@R=1ex{
& &&  [7]\ar@{->}[ddrr] && &&  [3,6]\ar@{->}[ddrr]  && && [1,2]\ar@{->}[ddrr]  && && \\
& &&&&& &&&& &&&& &&\\
& && &&[3,7]\ar@{->}[ddrr] \ar@{->}[uurr]  && && [1,6] \ar@{->}[ddrr]\ar@{->}[uurr] && && [2] &&  \\
&&&&&& &&&&&&&&\\
 & && [3,5] \ar@{->}[dr] \ar@{->}[uurr] && && [1,7]\ar@{->}[dr] \ar@{->}[uurr]  && && [2,6]\ar@{->}[dr] \ar@{->}[uurr]  && && \\
 [*,4] \ar@{->}[dr]&&[3,*]\ar@{->}[ur]&&[*,5]\ar@{->}[dr]& &[1,*]\ar@{->}[ur]&&[*,7]\ar@{->}[dr]& &[2,*]\ar@{->}[ur]&&[*,6]\ar@{->}[dr]& &&\\
& [3,4]\ar@{->}[ddrr] \ar@{->}[ur]  && && [1,5]\ar@{->}[ddrr] \ar@{->}[ur]  && &&[2,7] \ar@{->}[ddrr] \ar@{->}[ur] && && [5,6]\ar@{->}[ddrr]   &&\\
&&&&&&&&&&&&&&\\
& && [1,4]\ar@{->}[ddrr] \ar@{->}[uurr]  && && [2,5]\ar@{->}[ddrr] \ar@{->}[uurr]  && && [5,7]\ar@{->}[ddrr] \ar@{->}[uurr]  && &&[6]  \\
 &&&&&& &&&&&&&&&&\\
& [1] \ar@{->}[uurr]  && && [2,4] \ar@{->}[uurr]  && && [5] \ar@{->}[uurr] && &&[6,7] \ar@{->}[uurr] &&
}} \\
& \qquad\qquad\qquad
\scalebox{0.6}{\xymatrix@C=1ex@R=1ex{
& \frac{1}{2} &\frac{2}{2} &\frac{3}{2} &\frac{4}{2} &\frac{5}{2} &\frac{6}{2} &\frac{7}{2} &\frac{8}{2} &\frac{9}{2}
&\frac{10}{2} &\frac{11}{2} &\frac{12}{2} &\frac{13}{2} & \frac{14}{2} &\frac{15}{2} &\frac{16}{2} \\
1&& &&  [7]\ar@{->}[ddrr] && &&  [3,6]\ar@{->}[ddrr]  && && [1,2]\ar@{->}[ddrr]  && && \\
& &&&&& &&&& &&&& &&\\
2&& && &&[3,7]\ar@{->}[ddrr] \ar@{->}[uurr]  && && [1,6] \ar@{->}[ddrr]\ar@{->}[uurr] && && [2] &&  \\
&&&&&& &&&&&&&&\\
3& & && [3,5] \ar@{->}[dr] \ar@{->}[uurr] && && [1,7]\ar@{->}[dr] \ar@{->}[uurr]  && && [2,6]\ar@{->}[dr] \ar@{->}[uurr]  && && \\
4& [4]\ar@{->}[dr]&&[3]\ar@{->}[ur]&&[4,5]\ar@{->}[dr]& &[1,3]\ar@{->}[ur]&&[4,7]\ar@{->}[dr]& &[2,3]\ar@{->}[ur]&&[4,6]\ar@{->}[dr]& &&\\
5&& [3,4]\ar@{->}[ddrr] \ar@{->}[ur]  && && [1,5]\ar@{->}[ddrr] \ar@{->}[ur]  && &&[2,7] \ar@{->}[ddrr] \ar@{->}[ur] && && [5,6]\ar@{->}[ddrr]   &&\\
&&&&&&&&&&&&&&\\
6&& && [1,4]\ar@{->}[ddrr] \ar@{->}[uurr]  && && [2,5]\ar@{->}[ddrr] \ar@{->}[uurr]  && && [5,7]\ar@{->}[ddrr] \ar@{->}[uurr]  && &&[6]  \\
 &&&&&& &&&&&&&&&&\\
7&& [1] \ar@{->}[uurr]  && && [2,4] \ar@{->}[uurr]  && && [5] \ar@{->}[uurr] && &&[6,7] \ar@{->}[uurr] &&
}}
\end{align*}
\end{example}

\begin{corollary} \label{cor: label for non-induced} \hfill
\begin{enumerate}
\item[{\rm (a)}] For (1) or (4) in \eqref{eq: four situ}, ${\scriptstyle \bigstar}$ in $S$-path (resp. $N$-path) is labeled by $[n+1,*]$
(resp. $[*,n]$).
\item[{\rm (b)}] For (2) or (3) in \eqref{eq: four situ}, ${\scriptstyle \bigstar}$ in $S$-path (resp. $N$-path) is labeled by $[n+2,*]$
(resp. $[*,n+1]$).
\end{enumerate}
\end{corollary}

\begin{proof}
(a) Note that, by \eqref{eq: total non-total}, $\Upsilon_{\ii_0}$ does not contain $N$-path (resp. $S$-path) with $n$-arrows (resp. $(n-1)$-arrows). Then our first assertion follows from
Theorem \ref{them: comp for length k ge 0}.

(b) The second assertion follows from the same argument.
\end{proof}

\begin{definition} \cite[\S 5.3]{Kac}
For $\alpha=\sum_{i\in I} m_i \alpha_i \in \PR$, the \defn{support of $\alpha$} is defined by
\[ \text{supp}(\alpha):=\{  \, \alpha_k \, |\,  m_k \neq 0,\, k\in I\,  \}.\]
Also, if $\alpha_k\in \text{supp}(\alpha)$ then we say $\alpha_k$ is a support of $\alpha.$
\end{definition}

Since every induced central vertex in $\Upsilon_{[\ii_0]}$ is located in a sectional path with more than or equal to $n$-arrows, \eqref{eq: total non-total}, Proposition \ref{prop: comp for length k ge n} and
Corollary \ref{cor: label for non-induced} tells the following corollary:

\begin{corollary} \label{cor: supports for induced central} For an induced central vertex in $ \Upsilon_{[\ii_0]}$ with the label $\beta\in \Phi^+$ ,
\begin{enumerate}
\item[{\rm (a)}] if $\ii_0$ is in the case of (1) or (4) in \eqref{eq: four situ},  ${\rm supp}(\be) \supseteq \{\al_n,\al_{n+1} \}$,
\item[{\rm (b)}] if $\ii_0$ is in the case of (2) or (3) in \eqref{eq: four situ}, ${\rm supp}(\be) \supseteq \{\al_{n+1},\al_{n+2} \}$.
\end{enumerate}
\end{corollary}

For an induced vertex in $\Upsilon_{[\ii_0]}$, we can summarize as follows:
\begin{corollary} \label{cor:label1}
Consider the map $\iota^+: I_{2n}\to I_{2n+1}$ such that $\iota^+(i)=i$ for $i=1, \cdots, n$ and $\iota^+(i)= i+1$ for $i=n+1, \cdots, 2n.$
Then the labeling for the induced vertex $v$ in $\Upsilon_{[\ii_0]}$ corresponding to $[a,b]$ in $\Gamma_Q$ is determined as follows:
\begin{align}\label{eq: induced label}
\left\{ \begin{array}{ll}  \sum_{i=a}^{b} \alpha_{\iota^+(i)} +\alpha_{n+1} & \text{ if }  v \in  \Upsilon^{{\rm C}}_{[\ii_0]}, \\
 \sum_{i=a}^{b} \alpha_{\iota^+(i)} & \text{ otherwise.}    \end{array}  \right.
\end{align}
\end{corollary}

\subsection{Type $D_{n+1}$}\label{subsec:label_D}
In this subsection, we let $[\ii_0]$ be the twisted adapted class of type $D_{n+1}$  and $Q$ be the Dynkin quiver of type $A_n$ such that $\mathfrak{p}^{D_{n+1}}_{A_n}([\ii_0])=[Q]$.
For a root $\alpha=\lan a,b\ran$ in $\Phi^+_{D_{n+1}}$ (see \ref{Eqn:root_D}), we say $a$ and $b$ are \defn{components} of $\alpha$.

As in the previous subsection, $N$-paths and $S$-paths in $\Upsilon_{[\ii_0]}$ do the crucial role. Especially, there are two sectional paths denoted by $\mathcal{N}$ and $\mathcal{S}$ where
\begin{itemize}
\item $\mathcal{N}$ is the leftmost $N$-path with $(n-1)$-arrows,
\item $\mathcal{S}$ is the rightmost $S$-path with $(n-1)$-arrows.
\end{itemize}
Note that
\[\text{ $(n-1)$ is the largest number of arrows in a sectional path in $\Upsilon_{[\ii_0]}$.}\]
Also, the followings are useful facts in this section:
\begin{eqnarray}&&
\parbox{85ex}{
\begin{enumerate}[(i)]
\item  $\mathcal{N}$ and $\mathcal{S}$ do not have intersections.
\item  Two vertices with residue $1$ on  $\mathcal{N}$ and $\mathcal{S}$ are adjacent to each other.
\end{enumerate}
}\label{eq: N S property2}
\end{eqnarray}

\begin{definition}
Fix any class $[\ii_0]$ in $\lf \Qd \rf$ of type $D_{n+1}$ such that $\mathfrak{p}^{D_{n+1}}_{A_n}([\ii_0])=[Q]$ of type $A_n$.
\begin{enumerate}[(a)]
\item The full subquiver $\Gamma^W_Q$ (resp. $\Gamma^E_Q$) of $\Gamma_Q$ is the West part (resp. East part) of $\Gamma_Q$ whose boundary consists of the $S$-path with $(n-1)$-arrows in $\Gamma_Q$, which is unique.
\item The full subquiver $\Upsilon^W_{[\ii_0]}$ (resp. $\Upsilon^E_{[\ii_0]}$) of $\Upsilon_{[\ii_0]}$ is the West part (resp. East part) of $\Upsilon_{[\ii_0]}$ whose boundary consists of $\mathcal{N}$ (resp. $\mathcal{S}$).
\item The full subquiver $\Upsilon^C_{[\ii_0]}$ called the center of $\Upsilon_{[\ii_0]}$ is defined by
\[ \Upsilon^C_{[\ii_0]}= \Upsilon_{[\ii_0]} \setminus  ( \Upsilon^W_{[\ii_0]} \cup \Upsilon^E_{[\ii_0]}).\]
\end{enumerate}
\end{definition}

Note that, we have quiver isomorphisms
\begin{align} \label{Eqn:EW}
 \iota_E: \Gamma_Q^E \overset{\simeq}{\longrightarrow} \Upsilon_{[\ii_0]}^E, \quad \iota_W: \Gamma_Q^W \overset{\simeq}{\longrightarrow} \Upsilon_{[\ii_0]}^W.
\end{align}

\begin{example} \label{Ex:D_type_WEC}
Recall the quiver $\Upsilon_{[\ii_0]}$ in Example \ref{ex: Twisted D first}.
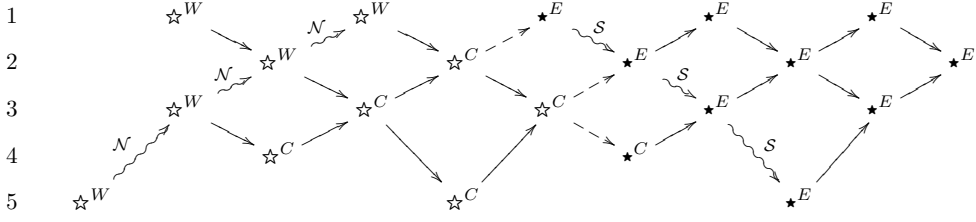
\begin{equation*}
 \scalebox{0.8}{\xymatrix@C=4ex@R=1ex{
1&&\medstar^W\ar@{->}[dr]  && \medstar^W\ar@{->}[dr]  &&  \filledstar^E \ar@{~>}[dr]^{\mathcal{S}}  && \filledstar^E\ar@{->}[dr] && \filledstar^E \ar@{->}[dr]\\
2&&& \medstar^W\ar@{->}[dr]\ar@{~>}[ur]^{\mathcal{N}}  &&\medstar^C\ar@{->}[dr]\ar@{-->}[ur]  && \filledstar^E\ar@{~>}[dr]^{\mathcal{S}}\ar@{->}[ur]  && \filledstar^E\ar@{->}[dr]\ar@{->}[ur] && \filledstar^E \\
3& & \medstar^W\ar@{->}[dr]\ar@{~>}[ur]^{\mathcal{N}}  &&\medstar^C\ar@{->}[ddr]\ar@{->}[ur]  &&\medstar^C\ar@{-->}[dr]\ar@{-->}[ur]  && \filledstar^E\ar@{~>}[ddr]^{\mathcal{S}}\ar@{->}[ur] && \filledstar^E\ar@{->}[ur] \\
4&&& \medstar^C\ar@{->}[ur]  &&&& \filledstar^C \ar@{->}[ur] \\
5&\medstar^W\ar@{~>}[uur]^{\mathcal{N}}  &&&& \medstar^C\ar@{->}[uur]  &&&& \filledstar^E \ar@{->}[uur]
}}
\end{equation*}
Then
\begin{itemize}
\item  $\Upsilon^W_{[\ii_0]}$ (resp. $\Upsilon^E_{[\ii_0]}$) is the set of $\medstar^W$'s  (resp. $\filledstar^E$'s),
\item  $\Upsilon^C_{[\ii_0]}$ is the set of $\medstar^C$'s and $\filledstar^C$'s,
\item  the sectional paths $\mathcal{N}$  and $\mathcal{S}$  consist of the arrows
\[\scalebox{0.8}{\xymatrix@C=8ex@R=1ex{ &\medstar^W \ar@{~>}[r]^{\mathcal{N}} &\medstar^W}} \quad \text{and}  \scalebox{0.8}{\xymatrix@C=8ex@R=1ex{ &\filledstar^E \ar@{~>}[r]^{\mathcal{S}} &\filledstar^E}} \quad  \text{ respectively.}  \]
\end{itemize}

On the other hand, consider the quiver $\Gamma_Q$ of type $A_n$:
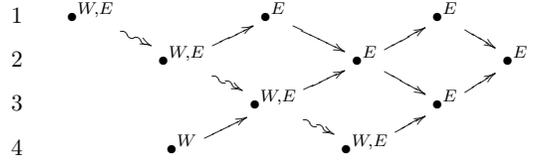
$$ \scalebox{0.8}{\xymatrix@C=3ex@R=1ex{
1 &   \bullet^{W,E} \ar@{~>}[dr]  &&   \bullet^E\ar@{->}[dr] &&   \bullet^E \ar@{->}[dr]\\
2& &   \bullet^{W,E}\ar@{~>}[dr]\ar@{->}[ur]  &&   \bullet^E\ar@{->}[dr]\ar@{->}[ur] &&  \bullet^E \\
 3&  &&   \bullet^{W,E}\ar@{~>}[dr]\ar@{->}[ur] &&   \bullet^E\ar@{->}[ur] \\
4& &  \bullet^W \ar@{->}[ur] && \bullet^{W,E}\ar@{->}[ur]
}}.
$$
We can see that
\begin{itemize}
\item  $\Gamma^W_Q$ (resp. $\Gamma^E_Q$) is the set of $\bullet^W$'s  (resp. $\bullet^E$'s) and $\bullet^{W,E}$'s,
\item  $\Gamma^E_Q\simeq \Upsilon^E_{[\ii_0]}$ by the canonical map,
\item   $\Gamma^W_Q\simeq \Upsilon^W_{[\ii_0]}$ by putting $\Gamma^W_Q$ upside down.
\end{itemize}
\end{example}

\begin{proposition} \label{Prop:label_E}
The labeling of $\Gamma^E_Q$ naturally induces the labeling of $\Upsilon^E_{[\ii_0]}.$ More precisely,
\begin{enumerate}
\item[{\rm (i)}] if the twisted Coxeter element $i_1 i_2 \cdots i_n \vee$ has the index $n+1$ then
\[ \iota_E([a,b])= \left\{ \begin{array}{ll}
 \lan a, -b-1\ran & \text{ if } b\neq n, \\
 \lan a, b+1\ran & \text{ if } b=n.
 \end{array}\right. \]
\item[{\rm (ii)}]  if the twisted Coxeter element has the index $n$ then
  \[ \iota_E([a,b])=  \lan a, -b-1\ran.\]
  \end{enumerate}
\end{proposition}

\begin{proof}
Let us denote the simple root $\al_i$ of type $A_n$ by $\alpha^A_i$ and the simple root $\al_i$  of type $D_{n+1}$ by $\alpha^D_i$.    For $1\leq a \leq b \leq n$, recall that
\[ \, [a,b]= \sum_{i=a}^b \alpha^A_i, \quad  \lan a, -b-1\ran= \sum_{i=a}^b \alpha^D_i, \quad  \lan a, n+1\ran= \left( \sum_{i=a}^{n-1} \alpha^D_i \right)+ \alpha^D_{n+1}.   \]
For the case (i), if $i_1 i_2 \cdots i_k$ is a compatible reading of $\Upsilon_{[\ii_0]}^E$ then $(i_1 i_2 \cdots i_k)^\vee$ is  a compatible reading of $\Gamma_Q^E.$ Hence $\iota_E(\sum_{i=a}^b \alpha^A_i)=\sum_{i=a}^b \alpha^D_{i^\vee}.$
For the case (ii), if $i_1 i_2 \cdots i_k$ is a compatible reading of $\Upsilon_{[\ii_0]}^E$ then $i_1 i_2 \cdots i_k$ is  a compatible reading of $\Gamma_Q^E.$ Hence $\iota_E(\sum_{i=a}^b \alpha^A_i)=\sum_{i=a}^b \alpha^D_{i}.$
\end{proof}

 \begin{proposition} \label{Prop:label_W}
 The labeling of $\Gamma^W_Q$  induces the labeling of $\Upsilon^W_{[\ii_0]}.$ More precisely,
 \begin{enumerate}
\item[{\rm (i)}] if the twisted Coxeter element $i_1 i_2 \cdots i_n \vee$ has the index $n+1$ then
   \[ \iota_W([a,b])=  \lan a, -b-1\ran.\]
\item[{\rm (ii)}] if the twisted Coxeter element has the index $n$ then
 \[ \iota_W([a,b])= \left\{ \begin{array}{ll}
 \lan a, -b-1\ran & \text{ if } b\neq n, \\
 \lan a, b+1\ran & \text{ if } b=n.
 \end{array}\right. \]
 \end{enumerate}
\end{proposition}

\begin{proof}
Here, we only show the proof when $n+1$ is even and $\Upsilon^W_{[\ii_0]}$ does not have a vertex with residue $n+1$.  Other cases can be proved similarly. In this case, the twisted Coxeter element contains the index $n+1$.

Let $i_1 i_2 \cdots i_k j_l j_{l-1} \cdots j_2 j_1$ be an element in $[\ii_0]$ such that $j_l j_{l-1} \cdots j_2 j_1$ is a compatible reading of $\Upsilon^W_{[\i_0]}$.
Note that a label $\beta^D$ in $\Upsilon^W_{[\ii_0]}$ is
\begin{equation}\label{Eqn:7.7_0830}
\beta^D= s_{i_1}\cdots   s_{i_{k-1}}s_{i_k} s_{j_l} s_{j_{l-1}}\cdots  s_{j_{m+1}}(\alpha^D_{j_m})= s_{j^{D*}_1}\cdots   s_{j^{D*}_{m-1}}(\alpha^D_{j^{D*}_m})= s_{j_1}\cdots   s_{j_{m-1}}(\alpha^D_{j_m}).
\end{equation}
Here $D*$ denotes the involution in Definition \ref{Def:inv}, which is the identity since  $n+1$ is even.

On the other hand,  $(n+1-j_l)(n+1- j_{l-1} )\cdots (n+1-j_2)(n+1- j_1)$ is a compatible reading of $\Gamma^W_Q$ and, for the type $A_n$ involution $A*$, we have $(n+1-j)^{A*}=j$. Hence the label $\beta^A:= \iota^{-1}_W(\beta^D)$ is
\begin{equation}\label{Eqn:7.8_0830}
 \beta^A=s_{j_1}\cdots   s_{j_{m-1}}(\alpha^A_{j_m}).
 \end{equation}
 By (\ref{Eqn:7.7_0830}) and (\ref{Eqn:7.8_0830}), we proved the proposition.
\end{proof}

 \begin{corollary} \label{Cor:label_D}\
\begin{enumerate}
\item Every vertex in $\mathcal{S}$ shares the second component $\pm (n+1)$ and every vertex in $\mathcal{N}$ shares the second component $\mp (n+1)$.
\item Let $\lan a,\pm (n+1) \ran$ be the label of a vertex in $\mathcal{S}$ with folded residue $\hat{i}$. Then $\lan a, \mp (n+1)\ran$ is the label of a vertex in $\mathcal{N}$ with folded residue $\widehat{n+1-i}$.
 \end{enumerate}
 \end{corollary}

\begin{proof}
 From the above two propositions, our assertions follow by comparing  {\rm (i)} (resp. {\rm (ii)}'s) of Proposition \ref{Prop:label_E} and {\rm (i)} (resp. {\rm (ii)}'s) of Proposition \ref{Prop:label_W}
\end{proof}

 \begin{lemma}  \label{Lem:add_D}
 Let $\alpha$ and $\beta$ be distinct roots in $\Phi^+$.
\begin{enumerate}
\item[{\rm (1)}] If there are two intersection points  $\gamma$ and $\delta$ of sectional paths through $\alpha$ and sectional paths through $\beta$ then $($see $(II-1)$ in \eqref{eq: complacted socle of D} below$)$
\[ \alpha+\beta= \gamma+\delta.\]
\item[{\rm (2)}]  Suppose there is only one intersection point $\gamma$  of  a sectional path through $\alpha$ and a sectional path through $\beta$. If the folded residues of $\alpha$ and $\beta$ are $\hat{i}$ and $\hat{j}$ and that of $\gamma$ is $\widehat{i+j}$ then $($see $(II-3)$ in \eqref{eq: complacted socle of D} below$)$
    \[ \alpha+\beta= \gamma.\]
\end{enumerate}
 \end{lemma}

 \begin{proof}(1) Suppose that the  folded  coordinates of positive roots  $\alpha, \beta, \gamma,$ and $\delta$ are $(\hat{i}, p+1)$, $ (\hat{i},p)$, $(\widehat{i-1},p+\frac{1}{2})$ and $ (\widehat{i+1}, p+\frac{1}{2}),$  respectively, for $i\leq n-1$. In other words, let $\alpha,\beta,\gamma,\delta$ consist of size $1\times 1$ rectangular.
Then there is an element $w\in W$  such that
\[ \alpha= w(\alpha_i), \, \beta= ws_i s_{\widetilde{i+1}} s_{i-1} (\alpha_{i}), \,  \gamma= ws_i s_{\widetilde{i+1}} (\alpha_{i-1}), \,  \delta= ws_i (\alpha_{\widetilde{i+1}}),\]
where $\widetilde{i+1}=n\text{ or }n+1$ if $i=n-1$ and  $\widetilde{i+1}=i+1$, otherwise. Hence $\alpha+\beta= \gamma+\delta.$

Now, if  $\alpha,\beta,\gamma,\delta$ consists of size $m\times n$ rectangular, then by applying the same argument $m\cdot n$ times, we get  $\alpha+\beta= \gamma+\delta.$

(2) Suppose that $\alpha$, $\beta$ and $\gamma$ have the folded coordinates $(\hat{1}, p)$, $(\hat{1}, p+1)$ and $(\hat{2}, p+\frac{1}{2}).$ Then there is an element $w\in W$ such that
\[ \alpha= w(\alpha_1), \quad \beta= w s_1 s_2 (\alpha_1), \quad \gamma=  w s_1(\alpha_2).\]
Hence $\alpha+\beta= \gamma$. Now, by using (1), we can deduce the lemma
\end{proof}

\begin{proposition} \label{Prop:Center_D_label}
Labeling for vertices in $\Upsilon^C_{[\ii_0]}$ is completely determined by Lemma \ref{Lem:add_D}. More precisely, if $\gamma\in \Upsilon^C_{[\ii_0]}$ is the intersection of
$$ \text{ an $N$-path crossing $\alpha=\lan a, \pm (n+1) \ran$ and an $S$-path crossing  $\beta=\lan b, \mp (n+1) \ran$},$$
then $\gamma=\lan a, b\ran$.
\end{proposition}

\begin{proof}
Without loss of generality, we can assume that $\alpha$ and $\beta$ are in $\mathcal{S}$ and $\mathcal{N}$.
Since $\alpha,\beta,\gamma$ satisfy assumptions in Lemma \ref{Lem:add_D} (2), we have $\gamma=\alpha+\beta$.
\end{proof}

\begin{theorem} \label{thm: Label Upsil d}
We can label  $\Upsilon_{[\ii_0]}$ by only observing its shape.
\end{theorem}

\begin{proof}
(1) By Proposition \ref{Prop:label_E}, Proposition \ref{Prop:label_W}, we can label the vertices lying in (i) all sectional paths with less than $(n-1)$-arrows, (ii) $\mathcal{N}$ and (iii) $\mathcal{S}$ by using Theorem \ref{thm: labeling GammaQ}. Since every vertices in $\Upsilon^C_{[\ii_0]}$ can be labeled by $\mathcal{N}$ and $\mathcal{S}$ by
Proposition \ref{Prop:Center_D_label} and, in addition, Theorem \ref{thm: labeling GammaQ} and  Proposition \ref{Prop:Center_D_label} depend only on the shape of $\Upsilon_{[\ii_0]}$,
our assertion follows.
\end{proof}

For the rest of this subsection, we shall list up the combinatorial properties of the labeling of $\widehat{\Upsilon}_{[\ii_0]}$ followed by Theorem
\ref{thm: Label Upsil d}:

\begin{proposition} \label{Prop:share comp}
A folded AR-quiver  $\widehat{\Upsilon}_{[\ii_0]}$ satisfies the following properties:
\begin{enumerate}
\item[{\rm (1)}] Every vertex in a sectional path shares a component.
\item[{\rm (2)}] Consider the $N$-path and the $S$-path which have the vertices with folded coordinates $(\hat{1}, p)$ and $(\hat{1}, p+1)$, respectively. If every vertex in the $N$-path shares the component $i$ then every vertex in the $S$-path shares the component $-i$.
\item[{\rm (3)}] Consider the $N$-path and the $S$-path which have the vertices with folded coordinates $(\hat{n}, q)$ and $(\hat{n}, q-1)$, respectively. If every vertex in the $N$-path shares the component $i$ then every vertex in the $S$-path also shares the component $i$.
\end{enumerate}
\end{proposition}

\begin{remark}
Inspired from Proposition \ref{Prop:share comp} (3), we will define a \defn{swing} in Definition \ref{Def:swing} below, which plays an important role in later sections.
\end{remark}

By Proposition \ref{Prop:share comp}, we get the following corollary.

\begin{corollary}
If there are two vertices $\alpha, \beta$ in $\Upsilon_{[\ii_0]}$ with  folded coordinates $(\hat{i}, p)$ and $(\widehat{n+1-i}, p+ \frac{n+1}{2})$, then there are $1\leq a<b\leq n+1$ such that $\{\alpha, \beta\}=\{ \lan a,b\ran, \lan a, -b\ran\}$.
\end{corollary}

\begin{proof}
When two vertices $\alpha$ and $\beta$ have folded coordinates  $(\hat{i}, p)$ and $(\widehat{n+1-i}, p+ \frac{n+1}{2})$, we have

{\rm (i)} the $N$-path passing $\alpha$ and the $S$-path passing $\beta$ which satisfy the assumptions in Proposition \ref{Prop:share comp} (2),

{\rm (ii)} the $S$-path passing $\alpha$ and the $N$-path passing $\beta$ which satisfy the assumptions in Proposition \ref{Prop:share comp} (3).

\noindent
Hence the corollary follows.
\end{proof}

Recall we can identify $\Upsilon_{[\ii_0]}$ with  $\Gamma_{Q^*} \overset{+}{\sqcup} \Gamma_{Q}$ (see Algorithm \ref{Alg surgery D}). Since $\Gamma_{Q^*}\simeq \Gamma_Q$ as in  Remark \ref{Rem:typeDtwisted_typeA_two_copy}, we can consider $\Upsilon_{[\ii_0]}$ as a union of two copies of $\Gamma_Q.$ Using this observation,  we can find the labeling of $\Upsilon_{[\ii_0]}$ in an efficient way:

\begin{proposition} \label{Prop:label_D_alter}
There exists an efficient algorithm for labeling of $\Upsilon_{[\ii_0]}$ which is canonically induced from the labeling of $\Gamma_Q$ :  Let $\Gamma_{Q_1}$ be the
quiver which is obtained by upside down $\Gamma_Q$ and has the same labeling of $\Gamma_Q$ and let $\Gamma_{Q_2}=\Gamma_Q$. Then we define the labeling map
$$ \widetilde{\iota} : \Gamma_{Q_1} \overset{+}{\sqcup} \Gamma_{Q_2} \longrightarrow  \Upsilon_{[\ii_0]}
\quad \text{ given by }\quad
\widetilde{\iota}([a,b]) = \begin{cases}
\iota_E([a,b]) & \text{ for } [a,b] \in \Gamma_{Q_2}^E, \\
\iota_W([a,b])  & \text{ for } [a,b] \in  \Gamma_{Q_1}^W, \\
\lan a,b+1\ran & \text{ otherwise }.
\end{cases}
$$
Here, {\rm (i)} $\Gamma_{Q_1} \overset{+}{\sqcup} \Gamma_{Q_2}$ denotes the quiver obtained by gluing $\Gamma_{Q_1}$ and $\Gamma_{Q_2}$, {\rm (ii)} $\Gamma_{Q_1}^W$ is obtained by upside down the quiver $\Gamma_Q^W$, and
{\rm (iii)} $\Gamma_{Q_2}^E$ is the same as $\Gamma_Q^E.$
\end{proposition}

\begin{example}
Let us consider the class $[\ii_0] \in \lf \Qd \rf$ of type  $D_{n+1}$  and $[Q] \in \lf \Delta \rf$ of type  in Example \ref{Ex:D_type_WEC}.
Now we can label $\Upsilon_{[\ii_0]}$ by only observing its shape and using the results in this subsection:

1. Draw two copies $\Gamma_{Q_1}$ and $\Gamma_{Q_2}$ of $\Gamma_{Q}$'s.
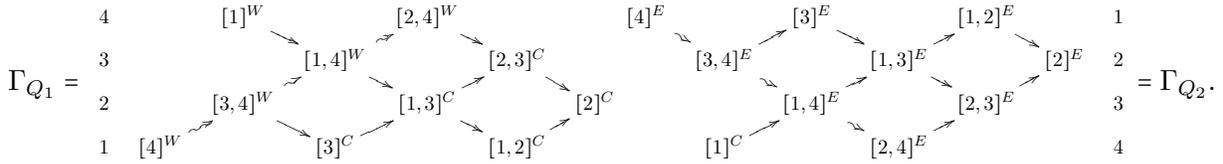
$$
\Gamma_{Q_1} =   \raisebox{2.5em}{\scalebox{0.7}{ \xymatrix@C=2ex@R=1ex{
4 && [1]^W \ar@{->}[dr]  && [2,4]^W \ar@{->}[dr]  \\
3 &&& [1,4]^W\ar@{->}[dr]\ar@{~>}[ur]  &&[2,3]^C\ar@{->}[dr] \\
2 & & [3,4]^W\ar@{->}[dr]\ar@{~>}[ur]  &&[1,3]^C\ar@{->}[dr]\ar@{->}[ur]&&[2]^C\\
1 &[4]^W\ar@{~>}[ur]&& [3]^C\ar@{->}[ur] && [1,2]^C\ar@{->}[ur]
}}}
\raisebox{2.5em}{\scalebox{0.7}{\xymatrix@C=2ex@R=1ex{
[4]^E\ar@{~>}[dr]  && [3]^E\ar@{->}[dr] && [1,2]^E \ar@{->}[dr]&& 1\\
& [3,4]^E\ar@{~>}[dr]\ar@{->}[ur]  && [1,3]^E\ar@{->}[dr]\ar@{->}[ur] &&[2]^E & 2\\
&& [1,4]^E\ar@{~>}[dr]\ar@{->}[ur] && [2,3]^E\ar@{->}[ur] &&3 \\
& [1]^C \ar@{->}[ur] && [2,4]^E\ar@{->}[ur]&&&4
}}} = \Gamma_{Q_2}.
$$
Recall that we can label $\Gamma_{Q}$ by observing its shape only.

2. Glue two quivers. Substitute $[a_1, a_2] \in \Upsilon^E_{[\ii_0]} \cup \Upsilon^W_{[\ii_0]}$ to $\lan a_1, -a_2-1\ran$ and substitute
$[a_1, a_2] \in \Upsilon^C_{[\ii_0]}$ to $\lan a_1, a_2+1\ran.$
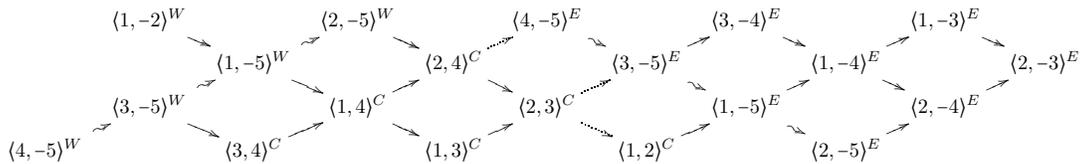
$$
   \raisebox{2.5em}{\scalebox{0.7}{ \xymatrix@C=2ex@R=1ex{
 && \lan1, -2\ran^W \ar@{->}[dr]  && \lan 2,-5\ran^W  \ar@{->}[dr] &&  \lan 4,-5 \ran^E \ar@{~>}[dr]  && \lan 3, -4 \ran^E \ar@{->}[dr] && \lan1,-3 \ran^E \ar@{->}[dr] \\
  &&& \lan1,-5\ran^W \ar@{->}[dr]\ar@{~>}[ur]  &&\lan 2,4\ran^C \ar@{->}[dr]\ar@{.>}[ur] &&\lan 3,-5\ran^E \ar@{~>}[dr]\ar@{->}[ur]  && \lan 1,-4\ran^E \ar@{->}[dr]\ar@{->}[ur] &&\lan 2,-3\ran^E   \\
 & & \lan 3,-5\ran^W  \ar@{->}[dr]\ar@{~>}[ur]  &&\lan 1,4\ran^C \ar@{->}[dr]\ar@{->}[ur]&&\lan 2,3\ran^C \ar@{.>}[ur]\ar@{.>}[dr] && \lan1,-5\ran^E \ar@{~>}[dr]\ar@{->}[ur] && \lan 2,-4\ran^E \ar@{->}[ur]\\
 &\lan 4, -5\ran^W  \ar@{~>}[ur]&& \lan 3, 4\ran^C \ar@{->}[ur] && \lan 1, 3\ran^C \ar@{->}[ur]& & \lan 1,2 \ran^C  \ar@{->}[ur] && \lan 2,-5\ran^E \ar@{->}[ur]
}}}
$$

3. Finally, by considering that the twisted Coxeter element has $5=n+1$, we substitute $\lan i, -5 \ran$ in $\mathcal{S}$ to $\lan i,5 \ran$ and get $\Upsilon_{[\ii_0]}$.
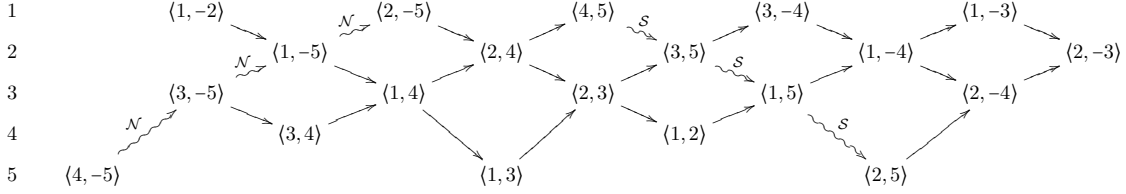
\begin{equation*}
 \scalebox{0.70}{\xymatrix@C=4ex@R=1ex{
1&&\lan 1, -2\ran\ar@{->}[dr]  && \lan 2, -5\ran \ar@{->}[dr]  &&   \lan 4, 5\ran \ar@{~>}[dr]^{\mathcal{S}}  &&  \lan 3, -4\ran\ar@{->}[dr] &&  \lan 1, -3\ran \ar@{->}[dr]\\
2&&&\lan 1, -5\ran\ar@{->}[dr]\ar@{~>}[ur]^{\mathcal{N}}  && \lan 2,4\ran\ar@{->}[dr]\ar@{->}[ur]  &&  \lan 3, 5\ran\ar@{~>}[dr]^{\mathcal{S}}\ar@{->}[ur]  &&  \lan 1, -4\ran\ar@{->}[dr]\ar@{->}[ur] &&  \lan 2, -3\ran \\
3& & \lan 3, -5\ran\ar@{->}[dr]\ar@{~>}[ur]^{\mathcal{N}}  && \lan 1,4\ran\ar@{->}[ddr]\ar@{->}[ur]  && \lan 2, 3\ran\ar@{->}[dr]\ar@{->}[ur]   &&  \lan 1, 5\ran\ar@{~>}[ddr]^{\mathcal{S}}\ar@{->}[ur] &&  \lan 2, -4\ran\ar@{->}[ur] \\
4&&&  \lan 3,4\ran\ar@{->}[ur]  &&&&  \lan 1,2\ran \ar@{->}[ur]  \\
5&\lan 4, -5\ran\ar@{~>}[uur]^{\mathcal{N}}  &&&&\lan 1,3\ran \ar@{->}[uur]  &&&&  \lan 2, 5\ran \ar@{->}[uur]
}}
\end{equation*}
\end{example}

\section{Representations of quantum affine algebras}\label{sec: Quantum affine and application}
In \cite{Oh14A, Oh14D, Oh15E}, the first named author interpreted denominator formulas and Dorey's rule for $U_q'(X^{(1)})$ $(X=A_n,D_n,E_n)$, in terms of AR-quivers of type $X$. Similarly, in Section \ref{Sec:app_AR}, there are analogous results for $U_q'(\widehat{X}^{(1)})$ $(\widehat{X}=B_{n+1},C_n,F_4,G_2)$ using twisted and folded AR-quivers. In this section, we briefly introduce some notions and theorems in the theory of quantum affine algebras including $R$-matrices, denominator formulas and Dorey's rule.

\subsection{Quantum affine algebras and their representations} Let $I_{\rm aff} = I \sqcup \{ 0 \}$ be the set of indices. An \defn{affine Cartan datum} is
a quadruple $(\cmA,P,\Pi,\Pi^\vee)$ consisting of
\begin{enumerate}
\item[({\rm a})] a matrix $\cmA= (a_{ij})_{i,j\in I_{\rm aff}}$ of corank $1$, called the \defn{affine Cartan matrix} satisfying
$$ ({\rm i}) \ a_{ii}=2 \ (i \in I_{\rm aff}), \quad ({\rm ii}) \ a_{ij} \in \Z_{\le 0}, \quad  ({\rm iii}) \ a_{ij}=0 \text{ if } a_{ji}=0$$
with $\mathsf{D}= {\rm diag}(\mathsf{d}_i \in \Z_{>0}  \mid i \in I_{\rm aff})$ making $\mathsf{D}\cmA$ symmetric,
\item[({\rm b})] a free abelian group $P$ of rank $n+2$, called the \defn{weight lattice},
\item[({\rm c})] $\Pi = \{ \alpha_i \mid i\in I_{\rm aff} \} \subset P$, called the set of \defn{simple roots},
\item[({\rm d})] $\Pi^{\vee} = \{ h_i \mid i\in I_{\rm aff}\} \subset P^{\vee} := {\rm Hom}( P, \Z )$, called the set of \defn{simple coroots},
\end{enumerate}
which satisfy
\begin{itemize}
\item[(1)] $\langle h_i, \alpha_j \rangle  = a_{ij}$ for all $i,j\in I_{\rm aff}$,
\item[(2)] $\Pi$ and $\Pi^{\vee}$ are linearly independent sets,
\item[(3)] for each $i \in I_{\rm aff}$, there exists $\Lambda_i \in P$ such that $\langle h_i, \Lambda_j \rangle =\delta_{ij}$ for all $j \in I_{\rm aff}$.
\end{itemize}
We set $\rl = \bigoplus_{i \in I_{\rm aff}} \Z \alpha_i$, $\rl_+ = \bigoplus_{i \in I_{\rm aff}} \Z_{\ge 0} \alpha_i$, $\rl^\vee = \bigoplus_{i \in I_{\rm aff}} \Z h_i$ and
$\rl^\vee_+ = \bigoplus_{i \in I_{\rm aff}} \Z_{\ge 0} h_i$.
We choose the \defn{imaginary root} $\delta=\sum_{i \in I_{\rm aff}}\mathsf{a}_i \alpha_i \in \rl_+$ and the \defn{center} $c=\sum_{i \in I_{\rm aff}} \mathsf{c}_ih_i \in \rl^\vee_+$
such that (\cite[Chapter 4]{Kac})
$$ \{ \lambda \in \rl \ | \ \langle h_i,\lambda \rangle =0 \text{ for every } i \in I_{\rm aff} \} =\Z \delta \quad \text{ and } \quad
\{ h \in \rl^\vee \ | \ \langle h,\alpha_i \rangle =0 \text{ for every } i \in I_{\rm aff} \} =\Z c.$$

Set $\h = \Q \otimes_\Z P^\vee$.
Then there exists a symmetric bilinear form $( \ , \ )$ on $\h^*$ satisfying
$$ \langle h_i,\lambda \rangle = \dfrac{  2(\alpha_i,\lambda)}{(\alpha_i,\alpha_i)} \quad \text{ for any } i \in I_{\rm aff} \text{ and } \lambda \in \h^*.$$
We normalize the bilinear form by
$$ \langle c,\lambda \rangle = (\delta,\lambda) \quad \text{ for any } \lambda \in \h^*.$$

Let us denote by $\g$ the affine Kac-Moody Lie algebra associated with $(\cmA,P,\Pi,\Pi^\vee)$ and by $\mathsf{W}_{{\rm aff}}$ the Weyl group of $\g$, generated by
$(\mathsf{s}_i)_{i \in I_{{\rm aff}}}$. We define $\g_0$ the subalgebra of $\g$ generated by the Chevalley generators $e_i,f_i,$ and
$h_i$ for $i \in I= I_{{\rm aff}} \setminus \{ 0 \}$. Then $\g_0$ becomes a finite dimensional simple Lie algebra.

Let $\mathsf{d}$ be the smallest positive integer such that
$$ \mathsf{d}(\alpha_i,\alpha_i)/2 \in \Z \quad \text{ for any } i \in I_{\rm aff}.$$
Note that $\mathsf{d}$ coincides with $\mathsf{d}$ in \eqref{eq: length of arrow}.

Let $q$ be an indeterminate. For $m,n \in \Z_{\ge 0}$ and $i\in I_{\rm aff}$, we define
$q_i = q^{(\alpha_i,\alpha_i)/2}$ and
\begin{equation*}
 \begin{aligned}
 \ &[n]_i =\frac{ q^n_{i} - q^{-n}_{i} }{ q_{i} - q^{-1}_{i} },
 \ &[n]_i! = \prod^{n}_{k=1} [k]_i ,
 \ &\left[\begin{matrix}m \\ n\\ \end{matrix} \right]_i=  \frac{ [m]_i! }{[m-n]_i! [n]_i! }.
 \end{aligned}
\end{equation*}

\begin{definition} \label{Def: GKM}
The {\em quantum affine algebra} $U_q(\g)$ associated
with $(\cmA,P,\Pi,\Pi^{\vee})$ is the associative
algebra over $\Q(q^{1/\mathsf{d}})$ with $1$ generated by $e_i,f_i$ $(i \in I_{{\rm aff}})$ and
$q^{h}$ $(h \in \mathsf{d}^{-1}P^{\vee})$ satisfying following relations:
\begin{enumerate}
  \item  $q^0=1, q^{h} q^{h'}=q^{h+h'} $ for $ h,h' \in \mathsf{d}^{-1}P^{\vee},$
  \item  $q^{h}e_i q^{-h}= q^{\langle h, \alpha_i \rangle} e_i,
          \ q^{h}f_i q^{-h} = q^{-\langle h, \alpha_i \rangle }f_i$ for $h \in \mathsf{d}^{-1}P^{\vee}$,
  \item  $e_if_j - f_je_i =  \delta_{ij} \dfrac{K_i -K^{-1}_i}{q_i- q^{-1}_i }, \ \ \text{ where } K_i=q_i^{ h_i},$
  \item  $\displaystyle \sum^{1-a_{ij}}_{k=0}
  (-1)^ke^{(1-a_{ij}-k)}_i e_j e^{(k)}_i =  \sum^{1-a_{ij}}_{k=0} (-1)^k
  f^{(1-a_{ij}-k)}_i f_jf^{(k)}_i=0 \quad \text{ for }  i \ne j, $
\end{enumerate}
where $e_i^{(k)}=e_i^k/[k]_i!$ and $f_i^{(k)}=f_i^k/[k]_i!$.
\end{definition}
We denote by $U_q^+(\g)$ (resp. $U_q^-(\g)$) the subalgebra of $U_q(\g)$ generated by $e_i$ (resp. $f_i$) $(i \in I_{{\rm aff}})$.
Let $U_q'(\g)$ be the subalgebra of $U_q(\g)$ generated by
$e_i,f_i,K^{\pm 1}_i$ $(i \in I_{{\rm aff}})$ and call it {\it also} the quantum
affine algebra. We mainly deal with $U_q'(\g)$-modules.

For $U_q'(\g)$-modules $M$ and $N$, $M \otimes N$ becomes a $U_q'(\g)$-module by the coproduct $\Delta$ of $U_q'(\g)$:
$$ \Delta(q^h)=q^h \otimes q^h, \ \ \Delta(e_i)=e_i \otimes K_i^{-1}+1 \otimes e_i, \ \Delta(f_i)=f_i \otimes 1 +K_i \otimes f_i.$$

A $U_q'(\g)$-module $M$ is called \defn{integrable} provided that
\begin{enumerate}
\item[({\rm a})] $ M = \bigoplus_{ \mu \in P_{\rm cl}} M_\mu$, where $P_{\rm cl} := P / \Z \delta$ and $M_\mu := \{ v \in M \ | \ K_i v = q_i^{\langle h_i, \mu \rangle } v \}$,
\item[({\rm b})] $e_i$ and $f_i$ $(i \in I)$ act on $M$ nilpotently.
\end{enumerate}
In this paper, we mainly consider
\begin{equation} \label{tensor category}
\text{ $\Ca_\g\,= $ the abelian tensor category consisting of finite dimensional integrable $U_q'(\g)$-modules.}
\end{equation}

We are interested in another family of $U_q'(\g)$-modules called \defn{good}.  Since the whole definition of the good module is not needed,  we just refer to
\cite{Kas02} for the precise definition.  However, the following is one of conditions of a good module, which we want to emphasize:
A good module $M$ contains the unique (up to constant) weight vector $v_M$ of weight $\lambda$, such that
$$   {\rm wt}(M) \subset \lambda + \sum_{i \in I} \Z_{\ge 0} {\rm cl}(\alpha_i).$$
We call $v_M$ the \defn{dominant extremal weight vector} and $\lambda$ \defn{dominant extremal weight}.

Let us consider \defn{the level $0$ fundamental weight} $\varpi_i$, for $i \in I$,  defined by
$$ \varpi_i := {\rm gcd}(\mathsf{c}_0,\mathsf{c}_i)^{-1}(\mathsf{c}_0\Lambda_i-\mathsf{c}_i\Lambda_0) \in P.$$
Then $\{ {\rm cl}(\varpi_i) \ | \ i \in I \}$, where ${\rm cl} \colon P \to P_{\rm cl}$ as the canonical projection, forms a basis for the space of \defn{classical integral weight level $0$},
denoted by $P^0_{\rm cl}$, which is defined as follows:
$$ P^0_{\rm cl} = \{ \lambda \in P_{\rm cl} \ | \ \langle c,\lambda \rangle =0  \}.$$
The Weyl group $W$ of $\g_0$, generated by $(\mathsf{s}_i)_{i \in I}$, acts on $P^0_{\rm cl}$ (see \cite[\S 1.2]{AK}).
We denote by $w_0$ the longest element of $W$.

\begin{definition} \cite[\S 1.3]{AK} \label{def: fundamental module}
For $i \in I$, the \defn{$i$-th fundamental module} is a unique finite dimensional integrable $U_q'(\g)$-module $V(\varpi_i)$ satisfying the following properties:
\begin{enumerate}
\item[(1)] The weights of $V(\varpi_i)$ are contained in the convex hull of $W{\rm cl}(\varpi_i)$.
\item[(2)] $V(\varpi_i)_{{\rm cl}(\varpi_i)}=\C(q) \mathsf{v}_{\varpi_i}$. (We call the vector $\mathsf{v}_{\varpi_i}$ a \defn{dominant integral weight vector}.)
\item[(3)] For any $\mu \in W{\rm cl}(\varpi_i)$, we can associate a non-zero vector $u_\mu$, called an \defn{extremal vector of weight $\mu$}, such that
\begin{equation} \label{eq: extremal}
\mathcal{S}_i\cdot u_\mu :=  u_{s_i\mu}= \begin{cases} f_i^{(\langle h_i, \mu \rangle)}u_\mu & \text{ if } \langle h_i, \mu \rangle \ge 0, \\
e_i^{(-\langle h_i, \mu \rangle)}u_\mu & \text{ if } \langle h_i, \mu \rangle \le 0, \end{cases} \quad \text{ for any $i \in I$.}
\end{equation}
\item[(4)] $\mathsf{v}_{\varpi_i}$ generates $V(\varpi_i)$ as a $U_q'(\g)$-module.
\end{enumerate}
\end{definition}
For instance, the $i$-th fundamental representation
is a good and integrable module.

\smallskip

Now, we fix the base field of  $U_q'(\g)$-modules
$\ko$ as the algebraic closure of $\C(q)$ in $\cup_{m >0}\C((q^{1/m}))$. When we deal with $U_q'(\g)$-modules, we regard the base field as $\ko$.

For an indeterminate $z$
and a $U_q'(\g)$-module $M$, let us denote by $M_z=\ko[z^{\pm 1}]\otimes M$
the $U_q'(\g)$-module  with the action of $U_q'(\g)$ given by
$$e_i(u_z)=z^{\delta_{i,0}}(e_iu)_z, \quad
f_i(u_z)=z^{-\delta_{i,0}}(f_iu)_z, \quad
K_i(u_z)=(K_iu)_z.
$$

\begin{definition} \label{def: C_Q module category} (\cite{HL11})
Let $Q$ be a Dynkin quiver of type $X=A_n$, $D_n$ or $E_n$.
For any positive root $\beta$ contained in $\Phi^+_X$, we set the $U_q'(X^{(1)})$-module $V_{Q}(\beta)$ as follows:
\begin{align} \label{eq: VQ(beta)}
V_{Q}(\beta) \seteq V(\varpi_{i})_{(-q)^{p}} \qquad  \text{ where } \quad \Omega_Q(\beta)=({i},p).
\end{align}
Denote by $\mathscr{C}_{Q}$ the smallest abelian full subcategory  of  the category $\Ca_{X^{(1)}}$ defined in (\ref{tensor category}) such that
\begin{itemize}
\item[{\rm (a)}] it is stable under subquotient, tensor product and extension,
\item[{\rm (b)}] it contains $V_{Q}(\beta)$ for all $\beta \in \Phi^+_X$, and the trivial module $\mathbf{1}$.
\end{itemize}
\end{definition}

\subsection{$R$-matrices, denominators and Dorey's rule}
In this subsection, we recall the notion of $R$-matrices, denominators and Dorey's rule for quantum affine algebras. We follow \cite[\S 8]{Kas02}.
Let us take a basis $\{ P_\nu \}$ of $U_q^+(\g)$ and a basis $\{ Q_{\nu} \}$ of $U_q^-(\g)$ which are dual to each other with respect to a suitable
coupling on $U_q^+(\g) \times U_q^-(\g)$. 
Then, for $U_q'(\g)$-modules $M$ and $N$, then there exists the \defn{universal R-matrix}  (\cite{Drin})
\begin{align}
\Runiv{M,N}(u \otimes v) = q^{(\wt(u),\wt(v))}\sum_\nu P_\nu v \otimes Q_\nu u,
\end{align} 
so that $\Runiv{M,N}$ gives a $U_q'(\g)$-homomorphism from $M \otimes N$ to $N \otimes M$ provided that an infinite sum has a meaning. For
$M, N \in \Ca_\g$, $\Runiv{M_{z_M},N_{z_N}}$ converges in $(z_N/z_M)$-adic topology. Thus we have a morphism of
$\ko[[z_N/z_M]] \otimes_{\ko[z_N/z_M]}  \otimes U_q'(\g)[z_M^{\pm 1},z_N^{\pm 1}]$-modules
\begin{align*}
\Runiv{M_{z_M},N_{z_N}}: \ko[[z_N/z_M]] \otimes_{\ko[z_N/z_M]} (M_{z_M} \otimes N_{z_N}) \to \ko[[z_N/z_M]] \otimes_{\ko[z_N/z_M]} (N_{z_N}\otimes M_{z_M}).
\end{align*}

We say that $\Runiv{M_{z_M},N_{z_N}}$ is \defn{rationally renormalizable} if there exist  $a \in \ko(z_N/z_M)$  and a $ U_q'(\g)[z_M^{\pm 1},z_N^{\pm 1}]$-module homomorphism
$$ \Rren{M_{z_M},N_{z_N}}: M_{z_M} \otimes N_{z_N} \to N_{z_N} \otimes M_{z_M}$$
such that $\Rren{M_{z_M},N_{z_N}} = a\Runiv{M_{z_M},N_{z_N}}$.
Then we can choose $\Rren{M_{z_M},N_{z_N}}$ so that, for any $c_1,c_2 \in \ko^\times$, the specialization of $\Rren{M_{z_M},N_{z_N}}$ at $z_M=c_1$, $z_N=c_2$,
$$\Rren{M_{z_M},N_{z_N}}|_{z_M=c_1,z_N=c_2}: M_{c_1} \otimes N_{c_2} \to N_{c_2} \otimes M_{c_1}$$
does not vanish under the assumption that $M$ and $N$ are non-zero modules in $\Ca_\g$.
Such an $\Rren{}$ is unique up to $\ko[(z_M/z_N)^{\pm 1}]^\times =
\sqcup_{n \in \Z}\ko^\times z_M^{n}z_N^{-n}$ and it is called a \defn{renormalized $R$-matrix}.

 We denote by
$$ \rmat{M,N} \seteq \Rren{M_{z_M},N_{z_N}}|_{z_M=1,z_N=1}: M \otimes N \to N \otimes M$$
and call it the \defn{$R$-matrix}. The $R$-matrix $ \rmat{M,N}$ is well-defined up to a constant multiple when
$\Runiv{M_{z_M},N_{z_N}}$ is rationally renormalizable. By definition, $ \rmat{M,N}$ never vanishes.

For simple $U_q'(\g)$-modules $M$ and $N$ in $\Ca_\g$, the universal $R$-matrix $\Runiv{M_{z_M},N_{z_N}}$ is rationally renormalizable. Then, for
dominant extremal weight vectors $u_M$ and $u_N$ of $M$ and $N$, there exists $a_{M,N}(z_N/z_M) \in \ko[[z_N/z_M]]^\times$ such that
\begin{align}\label{eq: aMN}
\Runiv{M_{z_M},N_{z_N}} \big( (u_M)_{z_M} \otimes (u_N)_{z_N} \big) =  a_{M,N}(z_N/z_M) \big( (u_N)_{z_N} \otimes (u_M)_{z_M} \big).
\end{align}
Then $\Rnorm{M_{z_M},N_{z_N}} \seteq a_{M,N}(z_N/z_M)^{-1}\Runiv{M_{z_M},N_{z_N}}$ is a unique $\ko(z_M,z_N) \otimes_{\ko[z_M^{\pm1},z_N^{\pm1}]} U_q'(\g)$-module homomorphism sending
$\big( (u_M)_{z_M} \otimes (u_N)_{z_N} \big)$ to $\big( (u_N)_{z_N} \otimes (u_M)_{z_M} \big)$:
$$
\Rnorm{M_{z_M},N_{z_N}}: \ko(z_M,z_N) \otimes_{\ko[z_M^{\pm1},z_N^{\pm1}]} ( M_{z_M} \otimes N_{z_N}) \to
\ko(z_M,z_N) \otimes_{\ko[z_M^{\pm1},z_N^{\pm1}]} ( N_{z_N} \otimes M_{z_M}).
$$

It is known that $\ko(z_M,z_N) \otimes_{\ko[z_M^{\pm1},z_N^{\pm1}]} ( M_{z_M} \otimes N_{z_N})$ is simple
$\ko(z_M,z_N) \otimes_{\ko[z_M^{\pm1},z_N^{\pm1}]} U_q'(\g)$-module (\cite[Proposition 9.5]{Kas02}). We call
$\Rnorm{M_{z_M},N_{z_N}}$ the \defn{normalized $R$-matrix}.

Let us denote by $d_{M,N}(u) \in \ko[u]$ a monic polynomial of the smallest degree such that the image
$d_{M,N}(z_N/z_M)$ $\Rnorm{M_{z_M},N_{z_N}}$ is contained in $N_{z_N} \otimes M_{z_M}$. We call $d_{M,N}$
the \defn{denominator} of $\Rnorm{M_{z_M},N_{z_N}}$. Then,
$$  \Rren{M_{z_M},N_{z_N}}= d_{M,N}(z_N/z_M)\Rnorm{M_{z_M},N_{z_N}} $$
and
$$d_{M,N}(z_N/z_M)\Rnorm{M_{z_M},N_{z_N}}|_{z_M=1,z_N=1}= c_{M,N} \cdot  \rmat{M,N}$$
for a constant $c_{M,N}$.

From the following theorem, one can notice that the denominator formulas between fundamental representations
provides crucial information of the representation theory on $\mathcal{C}_\g$.

\begin{theorem} \cite{AK,Chari,Kas02} $($see also \cite[Theorem 2.2.1]{KKK13A}$)$  \label{Thm: basic properties}
\begin{enumerate}
\item[{\rm (1)}] For good modules $M,N$, the zeroes of $d_{M,N}(z)$ belong to
$\cqm$ for some $m\in\Z_{>0}$.
\item[{\rm (2)}] Let $M_k$ be a good module
with a dominant extremal vector $u_k$ of weight $\lambda_k$, and
$a_k\in\mathbf{k}^\times$ for $k=1,\ldots, t$.
Assume that $a_j/a_i$ is not a zero of $d_{M_i, M_j}(z) $ for any
$1\le i<j\le t$. Then the following statements hold.
\begin{enumerate}
\item[{\rm (i)}]
 $(M_1)_{a_1}\otimes\cdots\otimes (M_t)_{a_t}$ is generated by
$u_1\otimes\cdots \otimes u_t$.
\item[{\rm (ii)}] The head of
$(M_1)_{a_1}\otimes\cdots\otimes (M_t)_{a_t}$ is simple.
\item[{\rm (iii)}] Any non-zero
submodule of $(M_t)_{a_t}\otimes\cdots\otimes (M_1)_{a_1}$ contains the vector $u_t\otimes\cdots\otimes u_1$.
\item[{\rm (iv)}] The socle of $(M_t)_{a_t}\otimes\cdots\otimes (M_1)_{a_1}$
is simple.
\item[{\rm (v)}]
 Let
$\rmat{}: (M_1)_{a_1}\otimes\cdots\otimes (M_t)_{a_t}
\to (M_t)_{a_t}\otimes\cdots\otimes (M_1)_{a_1}$
 be  the specialization of $R^{{\rm norm}}_{M_1,\ldots, M_t}$
at $z_k=a_k$.  Then the image of $\rmat{}$ is simple and
it coincides with the head of
$(M_1)_{a_1}\otimes\cdots\otimes (M_t)_{a_t}$
and also with the socle of $(M_t)_{a_t}\otimes\cdots\otimes (M_1)_{a_1}$.
\end{enumerate}
\item[{\rm (3)}]
For a simple integrable $U_q'(\g)$-module M, there exists
a finite sequence $\big((i_1,a_1),\ldots, (i_t,a_t)\big)$
in $I\times \mathbf{k}^\times$
such that
$d_{V(\varpi_{i_k}),V(\varpi_{ i_{k'} })}(a_{k'}/a_k) \seteq d_{i_k,i_{k'}}(a_{k'}/a_k)\not=0$ for $1\le k<k'\le t$ and
$M$ is isomorphic to the head of
$V(\varpi_{i_1})_{a_1}\otimes\cdots\otimes V(\varpi_{i_t})_{a_t}$.
Moreover, such a sequence $\big((i_1,a_1),\ldots, (i_t,a_t)\big)$
is unique up to permutation.
\item[{\rm (4)}] $d_{k,l}(z)=d_{l,k}(z)=d_{k^*,l^*}(z)=d_{l^*,k^*}(z)$ for $k,l \in I$.
\end{enumerate}
\end{theorem}

 The denominator formulas between fundamental representations over classical quantum affine algebras were calculated in \cite{AK,DO94,KKK13b,Oh14}:

\begin{theorem} \cite{AK,DO94,KKK13b,Oh14}  \label{thm: denom 1A2n}
\begin{enumerate}
\item[{\rm (a)}] $d^{A^{(1)}_{n}}_{k,l}(z) = \displaystyle\prod_{s=1}^{\min(k,l,n+1-k,n+1-l)} (z-(-q)^{|k-l|+2s}).$
\item[{\rm (b)}] $ d^{B^{(1)}_{n+1}}_{k,l}(z) =
\begin{cases}
  \displaystyle \prod_{s=1}^{\min (k,l)} \big(z-(-1)^{k+l}q^{|k-l|+2s}\big)  \prod_{s=1}^{\min (k,l)} \big(z-(-1)^{k+l}q^{2n+1-k-l+2s}\big) & \hspace{-2ex} \text{ if } 1 \le k,l \le n, \\
 \displaystyle  \prod_{s=1}^{k}\big(z-(-1)^{n+1+k}(q^{1/2})^{2n-2k+1+4s}\big) &  \hspace{-13ex}  \text{ if } 1 \le k \le n \text{ and } l=n+1, \\
  \displaystyle \prod_{s=1}^{n+1} \big(z-(q^{1/2})^{4s-2}\big) \times (z-q^{\mathsf{h}^\vee}) &  \hspace{-13ex}  \text{ if } k=l=n+1.
\end{cases}
$
\item[{\rm (c)}] $d^{C^{(1)}_{n}}_{k,l}(z)= \displaystyle {\prod_{s=1}^{ \min(k,l,n-k,n-l)}
\big(z-(-q^{1/2})^{|k-l|+2s}\big)}{\prod_{s=1}^{ \min(k,l)} \big(z-(-q^{1/2})^{2n+2-k-l+2s}\big)}$.
\item[{\rm (d)}] $ d^{D^{(1)}_{n+1}}_{k,l}(z)  =
\begin{cases}
\ \displaystyle {\prod_{s=1}^{\min(k,l)}  (z - (-q)^{|k-l|+2s})} {\prod_{s=1}^{\min(k,l)} (z -(-q)^{2n-k-l+2s})}  & \text{if } 1 \le k,l \le n-1, \\
\ \displaystyle \prod_{s=1}^{k}(z-(-q)^{n-k+2s}) & \hspace{-10ex}  \text{if } 1 \le k \le n-1,  l \in \{ n, n+1\}, \\
\ \displaystyle \prod_{s=1}^{\lfloor \frac{n}{2} \rfloor} (z-(-q)^{4s}) &  \hspace{-10ex} \text{if } \{k,l\}=\{n,n+1\},   \\
\ \displaystyle \prod_{s=1}^{\lfloor \frac{n+1}{2} \rfloor} (z-(-q)^{4s-2}) &   \hspace{-10ex} \text{if }  k=l \in \{ n, n+1\}.
 \end{cases}$
\end{enumerate}
\end{theorem}

A non-zero $U_q'(\g)$-module homomorphism $\psi$ is called a \defn{Dorey's type homomorphism} if
$$ \psi \in {\rm Hom}_{U'_q(\g)}( V(\varpi_k)_z,V(\varpi_i)_x \otimes V(\varpi_j)_y)$$
for some $i,j,k \in I$ and $x,y,z \in \mathbf{k}^\times$. By \cite[Theorem 3.2]{KKKO14S}, such $\psi$ is unique up to non-zero constant multiple.

\smallskip

The following theorems are referred as \defn{Dorey's rule} (see \cite{CP96}):

\begin{theorem}\label{thm: Dorey ADE} \cite{CP96,Oh14A,Oh14D,Oh15E}
Let $(i,x)$, $(j,y)$, $(k,z) \in I \times \mathbf{k}^\times$. Then
$$ {\rm Hom}_{U_q'(X^{(1)})}\big( V(\varpi_{k})_z, V(\varpi_{i})_x \otimes V(\varpi_{j})_y \big) \ne 0 \qquad (X=A_{n}, D_{n} \text{ or } E_{n})$$
if and only if there exists an adapted class $[Q]$ and $\al,\be,\ga \in \Phi_{X}^+$ such that
\begin{itemize}
\item[{\rm (i)}] $(\al,\be)$ is a pair of positive roots such that $\al+\be=\ga$,
\item[{\rm (ii)}] $V(\varpi_{j})_y  = V_{Q}(\be)_t, \ V(\varpi_{i})_x  = V_{Q}(\al)_t, \ V(\varpi_{k})_z  = V_{Q}(\ga)_t$
for some $t \in \mathbf{k}^\times$.
\end{itemize}
\end{theorem}

Now we present Dorey's rule for $U_q'(B^{(1)}_{n+1})$ and $U_q'(C^{(1)}_{n})$ which are interested in this paper:

\begin{theorem} \cite[Theorem 8.1]{CP96} $($see also \cite{Oh14}$)$
\begin{enumerate}
\item[{\rm (a)}] For $(i,x)$, $(j,y)$, $(k,z) \in I=\{1,2,\ldots, n,n+1 \} \times \mathbf{k}^\times$,
$$ {\rm Hom}_{U_q'(B^{(1)}_{n+1})}\big( V(\varpi_{k})_z  , V(\varpi_{i})_x \otimes V(\varpi_{j})_y   \big) \ne 0 $$
if and only if one of the following conditions holds$:$
\begin{eqnarray}&&
\left\{\hspace{1ex}\parbox{75ex}{
\begin{enumerate}
\item[{\rm (i)}] $\ell \seteq \max(i,j,k) \le n$, $i+j+k=2\ell$
and
$$ \hspace{-20ex} \left( y/z,x/z \right) =
\begin{cases}
\big( (-1)^{j+k}q^{-i},(-1)^{i+k}q^{j} \big), & \text{ if } \ell = k,\\
\big( (-1)^{j+k}q^{i-(2n+1)},(-1)^{i+k}q^{j} \big), & \text{ if } \ell = i,\\
\big( (-1)^{j+k}q^{-i},(-1)^{i+k}q^{2n+1-j}  \big), & \text{ if } \ell = j.
\end{cases}
$$
\item[{\rm (ii)}] $s \seteq \min(i,j,k) \le n$, the others are the same as $n+1$ and
$$  \hspace{-20ex}  (y/z,x/z) =
\begin{cases}
\big( (-1)^{n+1+k}(q^{1/2})^{-2(n-k)+1},(-1)^{n+1+k}(q^{1/2})^{2(n-k)-1} ), & \text{ if } s = k,\\
\big(  (q^{1/2})^{-4i-4},(-1)^{i+n+1} (q^{1/2})^{2(n-i)-1}  ), & \text{ if } s = i,\\
\big( (-1)^{j+n+1}(q^{1/2})^{-2(n-j)+1}, (q^{1/2})^{4j+4}), & \text{ if } s = j.
\end{cases}
$$
\end{enumerate}
}\right. \label{eq: Dorey B}
\end{eqnarray}
\item[{\rm (b)}] For $(i,x)$, $(j,y)$, $(k,z) \in I=\{1,2,\ldots, n \} \times \mathbf{k}^\times$,
$$ {\rm Hom}_{U_q'(C^{(1)}_{n})}\big( V(\varpi_{k})_z  , V(\varpi_{i})_x \otimes V(\varpi_{j})_y  \big) \ne 0 $$
if and only if one of the following conditions holds$:$
\begin{eqnarray}&&
\left\{\hspace{1ex}\parbox{75ex}{
$\ell \seteq \max(i,j,k) \le n$, $i+j+k=2\ell$
and
$$  \hspace{-20ex}  \left( y/z,x/z \right) =
\begin{cases}
\big( (-q^{1/2})^{-i},(-q^{1/2})^{j} \big), & \text{ if } \ell = k,\\
\big( (-q^{1/2})^{i-(2n+2)},(-q^{1/2})^{j} \big), & \text{ if } \ell = i,\\
\big( (-q^{1/2})^{-i},(-q^{1/2})^{2n+2-j}  \big), & \text{ if } \ell = j.
\end{cases}
$$
}\right. \label{eq: Dorey C}
\end{eqnarray}
\end{enumerate}
 \end{theorem}

\section{Distance and folded distance polynomials}  \label{Sec:app_AR}

In \cite{Oh15E}, the first named author described denominator formulas for untwisted affine type ADE using so-called distance polynomials, which are obtained by observing AR-quivers. In this section, we first review the distance polynomials defined on the adapted $r$-cluster point $\lf \Delta \rf$ of finite type ADE and relations between distance polynomials and denominator formulas. Then we introduce how to generalized the results to the cases of untwisted affine type BCFG by inventing folded distance polynomials on $\lf \Qd \rf$.
Also, as a generalization of Theorem \ref{thm: Dorey ADE},  we record the positions of minimal pairs for every $\ga \in \PR \setminus \Pi$ in $\Upsilon_{[\Qd]}$ to describe Dorey's rule for $U_q'(B^{(1)}_{n+1})$
and $U_q'(C^{(1)}_{n})$, in terms of twisted and folded AR-quivers.

\subsection{Notions on sequences of positive roots} In this section, we briefly review the notions on sequences of positive roots
which were mainly introduced in \cite{Mc12,Oh15E}.

\smallskip

Following~\eqref{compatible reading}, for a reduced expression $\jj_0=i_1i_2 \cdots i_{\N}$ of $w_0 \in W$, we set
$$\beta^{\jj_0}_k \seteq s_{i_1}\cdots s_{i_{k-1}}\alpha_{i_k} \in \Phi^+.$$
Now, we identify a sequence $\um_{\jj_0}=(m_1,m_2,\ldots,m_{\N}) \in \Z_{\ge 0}^{\N}$ with
\[
(m_1\beta^{\jj_0}_1,\ldots,m_\N\beta^{\jj_0}_\N) \in  (\Z_{\ge 0}\be^{\jj_0}_k)_{1 \le k \le \N} \simeq \Z_{\ge 0}^\N.
\]
If there is no danger of confusion, we omit the subscript $\jj_0$.

The \defn{weight} $\wt(\um)$ of a sequence $\um$ is defined by
\[
\sum_{i=1}^{\N} m_i\beta_i.
\]

\begin{definition}[\cite{Mc12,Oh15E}]
We define the partial orders $<^\tb_{\jj_0}$ and $\prec^\tb_{[\jj_0]}$ on $\Z_{\ge 0}^{\N}$ as follows:
\begin{enumerate}
\item[{\rm (i)}] $<^\tb_{\jj_0}$ is the bi-lexicographical partial order induced by $<_{\jj_0}$. Namely, $\um<^\tb_{\jj_0}\um'$ if there exist
$j$ and $k$ ($1\le j\le k\le \N$) such that
\begin{itemize}
\item $\um_s=\um'_s$ for $1\le s<j$ and $\um_j<\um'_j$,
\item $\um_{s}=\um'_{s}$ for $k<s\le \N$ and $\um_k<\um'_k$.
\end{itemize}
\item[{\rm (ii)}] For sequences $\um$ and $\um'$, we have $\um \prec^\tb_{[\jj_0]} \um'$ if and only if $\wt_{\jj_0}(\um)=\wt_{\jj_0}(\um')$ and
$\un<^\tb_{\jj_0'} \un'$ for all $\jj_0' \in [\jj_0]$,
where $\un$ and $\un'$ are sequences such that
$\un_{\jj_0'}=\um_{\jj_0}$ and $\un'_{\jj_0'}=\um_{\jj_0}$.
\end{enumerate}
\end{definition}

We give the following definitions from~\cite{Mc12,Oh15E}.
We call a sequence $\um$ a \defn{pair} if $|\um|\seteq \sum_{i=1}^\N m_i=2$ and $m_i \le 1$ for $1\le i\le \N$. We mainly use the notation $\up$ for a pair.
Frequently, we write a pair $\up$ as $(\alpha,\beta) \in (\PR)^2$.

We say a sequence $\um=(\um_1,\um_2,\ldots,\um_{\N}) \in \Z^{\N}_{\ge 0}$ is \defn{$[\jj_0]$-simple} if it is minimal with respect to the partial order $\prec^\tb_{[\jj_0]}$. For a given $[\jj_0]$-simple sequence $\us=(s_1,\ldots,s_{\N}) \in \Z^{\N}_{\ge 0}$, we call a cover\footnote{Recall that a cover of $x$ in a poset $P$ with partial order $\prec$ is an element $y \in P$ such that $x \prec y$ and there does not exists $y' \in P$ such that $x \prec y' \prec y$.} of $\us$ under $\prec^{\tb}_{[\jj_0]}$ a \defn{$[\jj_0]$-minimal sequence of $\us$}. The \defn{generalized $[\jj_0]$-distance  $\gdist_{[\jj_0]}(\um)$} of a sequence $\um$ is the largest integer $k \geq 0$ such that
\[
\um^{(0)} \prec^\tb_{[\jj_0]} \cdots \prec^\tb_{[\jj_0]} \um^{(k)} = \um
\]
and $\um^{(0)}$ is $[\jj_0]$-simple.

Consider a pair $\up$ such that there exists a unique $[\jj_0]$-simple sequence $\us$ satisfying
$\us \preceq^\tb_{[\jj_0]} \up$, we call $\us$ the \defn{$[\jj_0]$-socle} of $\up$ and denoted it by $\soc_{[\jj_0]}(\up)$.



For a non-simple positive root $\gamma \in \PR \setminus \Pi$, the \defn{$[\jj_0]$-radius} of $\gamma$, denoted by $\rds_{[\jj_0]}(\gamma)$,
is the integer defined as follows:
$$\rds_{[\jj_0]}(\gamma)=\max({\rm gdist}_{[\jj_0]}(\up) \ | \   \up: \text{ a pair}, \   \gamma  \prec_{[\jj_0]}^\tb \up ).$$


For $\eta =\sum_{i \in I}m_i\al_i \in \PR$, define the \defn{multiplicity} of $\eta$ as the integer defined by
\[
\mathsf{m}(\eta) = \max \{\, m_i \ | \ i \in I \}.
\]

\begin{theorem}\cite[Theorem 4.15, Theorem 4.20]{Oh15E}\cite[Theorem 3.4]{Oh14A} \label{thm: known for Q} Let $Q$ be any Dynkin quiver of type $A_n$, $D_n$ or $E_n$
\begin{enumerate}
\item[{\rm (1)}] For any pair $\up=(\al,\be) \in (\PR)^2$,  we have $0 \le \gdist_{[Q]}(\al,\be) \le \max\{ \mathsf{m}(\al), \mathsf{m}(\be) \}$.
\item[{\rm (2)}] For any $\gamma \in \PR \setminus \Pi$, we have $\rds_{[Q]}(\gamma) \le \mathsf{m}(\ga)$. Equality holds when $Q$ is of type $A_n$ or $D_n$.
\item[{\rm (3)}] For any pair $\up=(\al,\be) \in (\PR)^2$, $\soc_{[Q]}(\up)$ is well-defined.
\end{enumerate}
\end{theorem}

\subsection{Distance polynomials and Dorey's rule on $\lf \Delta \rf$} Let $Q$ be a Dynkin quiver of type $ADE$.
Following~\cite{Oh15E}, for an AR quiver $\Gamma_Q$, indices $k,l \in I$ and an integer $t \in \Z_{\ge 1}$,
we define the subset $\Phi_{Q}(k,l)[t] \subset (\PR)^2$ as the pairs $(\alpha,\beta)$ such that $\alpha$ and $\beta$ are comparable under $\prec_{Q}$ and
\[
\{ \Omega_Q(\al),\Omega_Q(\beta) \} = \{ (k,a), (l,b)\} \quad \text{ such that } \quad |a-b| = t.
\]

\begin{lemma} \cite[Lemma 6.12]{Oh15E} \label{lem: o well}
For any  $(\alpha^{(1)},\beta^{(1)})$ and $(\alpha^{(2)},\beta^{(2)})$ in $\Phi_{Q}(k,l)[t]$, we have
$$ o^{\overline{Q}}_t(k,l) := \gdist_Q(\alpha^{(1)},\beta^{(1)})=\gdist_{Q}(\alpha^{(2)},\beta^{(2)}). $$
\end{lemma}

We denote by $Q^{{\rm rev}}$ the quiver obtained
by reversing all arrows of $Q$, and by $Q^*$ the quiver obtained from $Q$ by replacing vertices of $Q$ from $i$ to $i^*$.

\begin{proposition} \cite[Proposition 6.16]{Oh15E} \label{prop: DQ DQ' 1A2n}
The integer, defined by
$$\mathtt{o}_t(k,l) \seteq  \max( o^{Q}_t(k,l),o^{Q^\rev}_t(k,l) )$$
does not depend on the choice of $Q$; that is,
$$\mathtt{o}^{\overline{Q}}_t(k,l) = \mathtt{o}^{\overline{Q'}}_t(k,l)$$
for any distinct Dynkin quivers $Q$, $Q'$ of the same type.
\end{proposition}

Since $\mathtt{o}_t(k,l)$ does not depend on the choice of $Q$, we can the define
distance polynomials $D_{k,l}(z) \in \ko[z]$ on $\lf \Delta \rf$.

\begin{definition} \cite{Oh15E} \label{def distance polynomial}
For $k,l \in I$, we define the \defn{distance polynomial} $D_{k,l}(z) \in \ko[z]$ on $\lf \Delta \rf$
\begin{align} \label{eq: distance polynomial}
D_{k,l}^X(z) \seteq  \prod_{ t \in \Z_{\ge 0} } (z-(-1)^t q^{t} )^{\mathtt{o}_t(k,l)}.
\end{align}
Here $X$ denotes the type of $\lf \Delta \rf$.
\end{definition}

\begin{theorem} \cite[Theorem 6.18]{Oh15E}\label{thm: dist denom}
For any Dynkin quiver $Q$ of type $X$, the denominator formulas for the quantum affine algebra $U'_q(X^{(1)})$ can be read as follows $(X=A_{n}$ or $D_{n})$:
\begin{align*}
d^{X^{(1)}}_{k,l}(z) & =D^{X}_{k,l}(z) \times
(z-(-q)^{\mathsf{h}^\vee})^{\delta_{l,k^*}}
\end{align*}
where $\mathsf{h}^\vee$ is the dual Coxeter number of type $X$.
\end{theorem}

\subsection{Generalized distance and radius on $\lf \Qd \rf$}
For this subsection, we will prove the following theorem:

\begin{theorem} \label{thm: dist upper bound Qd}  For a non-trivial automorphism $\vee$, recall $\mathsf{d}$ is defined by using $\widehat{\Phi}^+$ in \eqref{eq: length of arrow}. Take any $[\ii_0] \in \lf\Qd\rf$ or $\lf \mathfrak{Q} \rf$.
\begin{enumerate}
\item[{\rm (1)}] For any pair $\up=(\al,\be) \in (\PR)^2$ $($not $(\widehat{\Phi}^+)^2)$, we have $ 0 \le \gdist_{[\ii_0]}(\al,\be) \le \mathsf{d}$.
\item[{\rm (2)}] For any $\gamma \in \PR \setminus \Pi$, we have  $1 \le \rds_{[\ii_0]}(\gamma)\le \mathsf{d}$.
\item[{\rm (3)}] For any $\up=(\al,\be) \in (\PR)^2$ and $\vee$ in \eqref{eq: B_n} or \eqref{eq: C_n}, $\soc_{[\ii_0]}(\up)$ is well-defined.
\end{enumerate}
\end{theorem}

\subsubsection{Proof of Theorem \ref{thm: dist upper bound Qd} for type $A_{2n+1}$} \label{subsec: dis radi A}

\begin{lemma}
For a non-simple root $\gamma$ corresponding to non-central vertex in $\Upsilon_{[\ii_0]}$, we have
$$\rds_{[\ii_0]}(\gamma)=1.$$
\end{lemma}

\begin{proof}
Let us assume that $\ga \in \uUp^{{\rm NE}}_{[\ii_0]} \sqcup \uUp^{{\rm SW}}_{[\ii_0]}$. Then, by Lemma \ref{lem: comp for length k le n} and Theorem \ref{them: comp for length k ge 0},
every nonzero component of a  sequence $\um$ with $\wt(\um)=\gamma$ should appear in $\uUp^{{\rm NE}}_{[\ii_0]} \sqcup \uUp^{{\rm SW}}_{[\ii_0]}$. Hence our assertion immediately follows from
Algorithm \ref{Rem:surgery A}, Lemma \ref{cor:label1} and Theorem \ref{thm: known for Q}. We can prove for
$\ga \in \uUp^{{\rm SE}}_{[\ii_0]} \sqcup \uUp^{{\rm NW}}_{[\ii_0]}$ in the similar way.
\end{proof}

\begin{lemma} \label{lem: less than eq 2}
For any $\gamma \in \PR \setminus \Pi$ corresponding to an induced central vertex, $$\rds_{[\ii_0]}(\gamma)\le \mathsf{d}=2.$$
\end{lemma}

\begin{proof}
By Corollary \ref{cor: label for non-induced} and Corollary \ref{cor: supports for induced central},
there exists a unique pair $(\al^\star,\be^\star)$ lying in the $n+1$-th layer ${\scriptstyle \bigstar}$ and $\al^\star+\be^\star=\ga$.
Also, by Algorithm \ref{Rem:surgery A}, Lemma \ref{cor:label1} and \cite[Proposition 4.24]{Oh15E}, other pairs $(\al,\be)$ such that $\al+\be=\ga$
correspond to induced vertices. Moreover, the pairs of induced vertices are not $\prec^\tb_{[\ii_0]}$-comparable to each other by Theorem \ref{thm: known for Q}.
Thus our assertion follows from the fact that sometimes the exceptional pair $(\al^\star,\be^\star)$ is comparable to a pair which consisting of
induced vertices. In Example \ref{ex: label}, we can see
$$  [3,5] \prec^{\tb}_{[\ii_0]} (\al^\star,\be^\star)=([4,5],[3]) \prec^{\tb}_{[\ii_0]} ([5],[3,4]).$$
By \cite[Theorem 3.2]{Oh14A}, the  non-induced vertices pairs of weight $\gamma$ are less than other induced vertices pairs of weight
$\gamma$ with respect to $\prec^{\tb}_{[\ii_0]}$ whenever they are comparable.
\end{proof}

\begin{lemma} \label{lem: non-induced}
For any $\gamma \in \PR \setminus \Pi$ corresponding to a non-induced central vertex, $$\rds_{[\ii_0]}(\gamma)=1.$$
\end{lemma}

\begin{proof}
Let us assume that $[\ii_0]$ satisfies the case $(1)$ in \eqref{eq: four situ}. Then $\gamma$ is $[a,n]$ or $[n+1,b]$ by Corollary \ref{cor: label for non-induced}. We assume further that
$\gamma=[a,n]$. Then every pair for $\ga$ is of the form $\{ [a,b-1],[b,n] \}$. Without loss of generality,  we assume that there are pairs
$$ \{ [a,c-1],[c,n] \} \prec^\tb_{[\ii_0]} \{ [a,d-1],[d,n] \}. $$
Then there is a path between $[c,n]$ and $[d,n]$ by Lemma \ref{lem: comp for length k le n} and Corollary \ref{cor: label for non-induced}.

Suppose that $[d,n] \prec_{[\ii_0]}[c,n]$. Then there is a path from $[c,n]$ to $[d,n]$.
We can take a path going through two vertices :
(i)  $V_d=[d,k]$ right before $[d,n]$ lying in the $(n+2)$-th layer,
(ii) $V_c=[c,k]$ for some $k > n$; that is,
$$\xymatrix@C=10ex{ [c,n] \ar@{->}[r]^{N\text{-sectional}}  & [c,k] \ar@{->}[r]^{S\text{-sectional}}_{\text{share } [ \ ,k]} & [d,k] \ar@{->}[r]^{N\text{-sectional}}_{\text{one arrow}} & [d,n]}
$$
Here $k > n$ by Corollary \ref{cor: supports for induced central} and $[d,k]$ is an induced central vertex.
Now we know
\begin{itemize}
\item $[d,n] \prec_{[\ii_0]}  [a,c-1]$ so that $[d,k] \prec_{[\ii_0]} [a,c-1], [c,k]$,
\item $[a, c-1], [c,k] \prec_{[\ii_0]} [a,d-1]$ by the fact that $[c,k] \prec_{[\ii_0]} [c,n]$.
\end{itemize}
Hence
$$\{[a,c-1], [c,k]\} \prec^\tb_{[\ii_0]} \{[a,d-1],[d,k]\},$$
where they are induced. However, it contradicts to Theorem \ref{thm: known for Q} (2).

Also, when there is a path from $[d,n]$ to $[c,n]$, we can prove by similar arguments.
 \end{proof}

\begin{proof} [{\bf The first step for  Theorem \ref{thm: dist upper bound Qd}}]
From the above three lemmas, the second assertion of Theorem \ref{thm: dist upper bound Qd} follows. Furthermore, the first and the third assertions
for $(\al,\be)$ with $\al+\be \in \PR$ also hold.
\end{proof}

\begin{proposition} \cite[Proposition 4.5]{Oh15E}
For a Dynkin quiver $Q$ of type $A_m$ and $(\al,\be)$ with $\al+\be \not\in \PR$ and $\gdist_{[Q]}(\al,\be)=1$, there exists a unique rectangle in $\Gamma_Q$
given as follows:
\begin{align}\label{eq:rectangle2}
\raisebox{3.5em}{\scalebox{0.8}{{\xy (3,-3)*{}="T"; (5,-25)*{}="B";(15,-15)*{}="C";
"T"; "C" **\dir{-}; "C"; "B" **\dir{-}; "T"; "T"+(-5,-5) **\dir{-};
"T"+(-5,-5); "T"+(-10,-10) **\dir{-}; "B"; "B"+(-5,5) **\dir{-};
"B"+(-5,5); "B"+(-12,12) **\dir{-}; "C"*{\bullet}; "B"*{\bullet};
"T"*{\bullet}; "B"+(-12,12)*{\bullet}; "B"+(-17,12)*{\scriptstyle
\be}; "B"+(0,-3)*{\scriptstyle \ga}; "C"+(5,0)*{\scriptstyle
\al}; "T"+(0,3)*{\scriptstyle \eta};
\endxy}}}
\end{align}
where $(\ga,\eta) \prec^\tb_{Q} (\al,\be)$. Furthermore,
 there is no pair $(\al',\be') \ne (\ga,\eta)$ such that
$(\al',\be') \prec^\tb_{Q} (\al,\be)$.
\end{proposition}

\begin{proposition} \label{prop: less than eq to 2}
For any pair  $(\al,\be)$ of type $A_{2n+1}$ such that $\al+\be \not\in \PR$,
$$\gdist_{[\ii_0]}(\al,\be) \le 2.$$
\end{proposition}

\begin{proof}
For $(\al,\be)$ which satisfies one of the following properties:
\begin{itemize}
\item ${\rm supp}(\al) \cap {\rm supp}(\be) = \emptyset$,
\item their first (resp. second) components are the same,
\item they are incomparable with respect to $\prec_{[\ii_0]}$,
\end{itemize}
one can prove easily that $\gdist_{[\ii_0]}(\al,\be)=0$ by using the convexity of $\prec_{[\ii_0]}$ and Theorem \ref{them: comp for length k ge 0}.
Thus $\gdist_{[\ii_0]}(\al,\be)>0$  implies that there exists a rectangle like \eqref{eq:rectangle2} or one of $\alpha$ and $\beta$ is a non-induced vertex.

(1) Now we assume that $\al,\be$ are all induced. By Algorithm \ref{Rem:surgery A} if there is a  rectangle like \eqref{eq:rectangle2} then  the rectangle can be classified with the followings:
\begin{itemize}
\item[{\rm (i)}] the rectangle without non-induced vertices on it,
\item[{\rm (ii)}] the rectangle with two non-induced vertices whose first or second component are the same,
\item[{\rm (iii)}] the rectangle with two non-induced vertices whose sum is contained in $\PR$ by Lemma \ref{cor: label for non-induced}.
\end{itemize}
For {\rm (i)}, the proofs are the same as in \cite[Proposition 4.5]{Oh15E}.
The cases {\rm (ii)} and {\rm (iii)} can be depicted as follows.

\vskip -1.5em

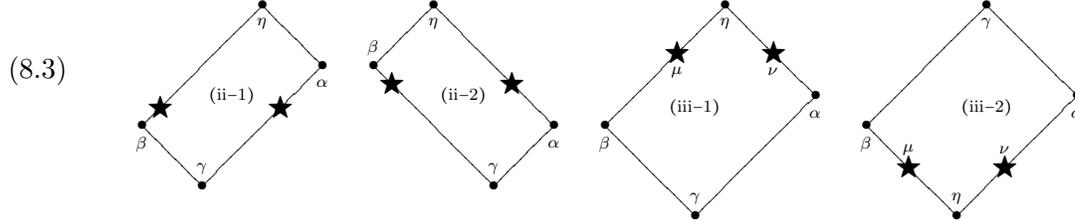
\begin{align} \label{eq: case A 1}
\raisebox{2.5em}{\scalebox{0.8}{{\xy (0,0)*{}="T1"; (10,-10)*{}="R1";(-20,-20)*{}="L1"; (-10,-30)*{}="B1";
"T1"; "R1" **\dir{-}; "T1"; "L1" **\dir{-}; "L1"; "B1" **\dir{-};,"R1"; "B1" **\dir{-};
"T1"*{\bullet}; "B1"*{\bullet};,"L1"*{\bullet};"R1"*{\bullet};
"L1"+(3,3)*{\bigstar};"B1"+(13,13)*{\bigstar};
"T1"+(-5,-15)*{_{{\rm (ii-1)}}};
"L1"+(0,-3)*{_{\be}};"R1"+(0,-3)*{_{\al}};
"T1"+(0,-3)*{_{\eta}};"B1"+(0,3)*{_{\ga}};
\endxy}} \quad
\scalebox{0.8}{{\xy (0,0)*{}="T2"; (-10,-10)*{}="L2";(20,-20)*{}="R2"; (10,-30)*{}="B2";
"T2"; "R2" **\dir{-}; "T2"; "L2" **\dir{-}; "L2"; "B2" **\dir{-};,"R2"; "B2" **\dir{-};
"T2"*{\bullet}; "B2"*{\bullet};,"L2"*{\bullet};"R2"*{\bullet};
"L2"+(3,-3)*{\bigstar};"T2"+(13,-13)*{\bigstar};
"T2"+(5,-15)*{_{{\rm (ii-2)}}};
"L2"+(0,+3)*{_{\be}};"R2"+(0,-3)*{_{\al}};
"T2"+(0,-3)*{_{\eta}};"B2"+(0,3)*{_{\ga}};
\endxy}}
\quad
\scalebox{0.8}{{\xy (0,0)*{}="T3"; (-20,-20)*{}="L3";(15,-15)*{}="R3"; (-5,-35)*{}="B3";
"T3"; "R3" **\dir{-}; "T3"; "L3" **\dir{-}; "L3"; "B3" **\dir{-};,"R3"; "B3" **\dir{-};
"T3"*{\bullet}; "B3"*{\bullet};,"L3"*{\bullet};"R3"*{\bullet};
"T3"+(-8,-8)*{\bigstar};"T3"+(8,-8)*{\bigstar};
"T3"+(-8,-11)*{_{\mu}};"T3"+(8,-11)*{_{\nu}};
"T3"+(-5,-17)*{_{{\rm (iii-1)}}};
"L3"+(0,-3)*{_{\be}};"R3"+(0,-3)*{_{\al}};
"T3"+(0,-3)*{_{\eta}};"B3"+(0,3)*{_{\ga}};
\endxy}}
\quad
\scalebox{0.8}{{\xy (0,0)*{}="T4"; (-20,-20)*{}="L4";(15,-15)*{}="R4"; (-5,-35)*{}="B4";
"T4"; "R4" **\dir{-}; "T4"; "L4" **\dir{-}; "L4"; "B4" **\dir{-};,"R4"; "B4" **\dir{-};
"T4"*{\bullet}; "B4"*{\bullet};,"L4"*{\bullet};"R4"*{\bullet};
"B4"+(-8,8)*{\bigstar};"B4"+(8,8)*{\bigstar};
"T4"+(0,-17)*{_{{\rm (iii-2)}}};
"L4"+(0,-3)*{_{\be}};"R4"+(0,-3)*{_{\al}};
"T4"+(0,-3)*{_{\ga}};"B4"+(0,3)*{_{\eta}};
"B4"+(-8,11)*{_{\mu}};"B4"+(8,11)*{_{\nu}};
\endxy}}}
\end{align}
where ${\scriptstyle\bigstar}$'s denote non-induced vertices.
\begin{eqnarray} &&
\parbox{85ex}{
 Note that if $\um \prec_{[\ii_0]}^{\tb} (\al,\be)$, then positive roots  occurring in $\um$ should be contained in or on the rectangle. Also, Theorem \ref{them: comp for length k ge 0} and
 Theorem \ref{thm: known for Q} tell that  $\um$ cannot consist of induced vertices except the pair $(\eta,\gamma)$
}\label{eq: obe}
\end{eqnarray}

\noindent {\rm (ii)} By \eqref{eq: obe}, $\um \ne (\eta,\gamma)$ must contain a vertex ${\scriptstyle\bigstar}$ if it exists, where ${\scriptstyle\bigstar}$'s in \eqref{eq: case A 1} share second component.
However, the convexity $\prec_{[\ii_0]}$, the system $\Phi^+$ and
Corollary \ref{cor: label for non-induced} tell that such an $\um$ cannot exist. Thus $\gdist_{[\ii_0]}(\al,\be)=1$.

\vskip 2mm
\noindent {\rm (iii)} By Corollary \ref{cor: label for non-induced}, $\mu+\nu=\eta$ and hence we have
$$(\eta,\gamma) \prec_{[\ii_0]}^{\tb} (\nu,\mu,\gamma) \prec_{[\ii_0]}^{\tb} (\al,\be).$$
As in the case {\rm (ii)}, there is no $\um$  in or on the rectangle with $\wt(\um)=\al+\be$ and different from $(\nu,\mu,\gamma)$ and $ (\eta,\gamma)$.
Thus $\gdist_{[\ii_0]}(\al,\be)=2$.

\vskip 2mm

(2) Now let $\alpha$ be an induced vertex and $\beta$ be a non-induced vertex. As in (1), we can classify as follows:

\vskip -2em

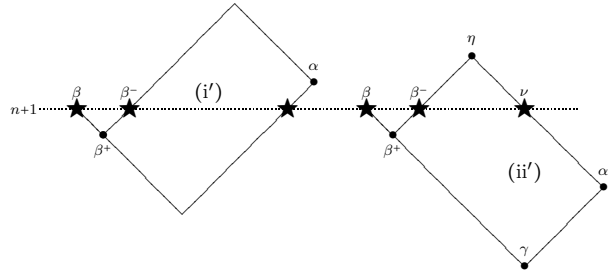
$$
\scalebox{0.7}{{\xy (0,0)*{}="T1"; (15,-15)*{}="R1";(-25,-25)*{}="L1"; (-10,-40)*{}="B1";
"L1"+(55,0)*{}="L2"; "L2"+(15,15)="T2";
"T2"+(25,-25)="R2";"R2"+(-15,-15)="B2";
"T2"; "R2" **\dir{-}; "T2"; "L2" **\dir{-}; "L2"; "B2" **\dir{-};,"R2"; "B2" **\dir{-};
"L2"+(-5,5); "L2" **\dir{-}; "L2"+(-5,8)*{_{\be}};"L2"+(5,8)*{_{\be^-}};
"L2"+(-5,5)*{\bigstar};"L2"+(5,5)*{\bigstar};
"T1"; "R1" **\dir{-}; "T1"; "L1" **\dir{-}; "L1"; "B1" **\dir{-};,"R1"; "B1" **\dir{-};
"L1"+(-5,5); "L1" **\dir{-}; "L1"+(-5,8)*{_{\be}};"L1"+(5,8)*{_{\be^-}};
"L1"+(20,8)*{ {\rm (i')} };
"L1"+(-5,5)*{\bigstar};"L1"+(5,5)*{\bigstar};
"L1"+(0,-3)*{_{\be^+}};"L1"*{\bullet};
"T2"*{\bullet};"T2"+(0,3)*{_{\eta}};
"R1"*{\bullet};"R1"+(0,3)*{_{\al}};
"R1"+(-5,-5)*{\bigstar};
"T2"+(10,-10)*{\bigstar};
"T2"+(10,-7)*{_{\nu}};
"B2"*{\bullet};"B2"+(0,3)*{_{\ga}};
"R2"*{\bullet};"R2"+(0,3)*{_{\al}};
"R2"+(-15,3)*{ {\rm (ii')} };
"L2"+(0,-3)*{_{\be^+}};"L2"*{\bullet};
"L1"*{\bullet};"L1"+(-15,5)*{_{n+1}};
"L1"+(-12,5); "L1"+(90,5) **\dir{.};
\endxy}}
$$
where $\beta^-$ is the largest non-induced vertex such that $\beta \prec_{[\ii_0]} \beta^-$, and $\beta^+=\beta+\beta^-$.

 In order to see $\gdist_{[\ii_0]}(\alpha, \beta)$, we need to find a set of vertices such that (a) every element is in or on the rectangle determined by $\beta^+$
 and $\alpha$ (b) the sum of elements is $\alpha+\beta$. Hence depending on whether
  $(\alpha, \beta^+)$ is of the case {\rm(i)} or {\rm(ii)} in (1), we get $\gdist_{[\ii_0]}(\alpha, \beta)=0$ or $1.$
 For the latter case, we have
$(\gamma,\nu) \prec_{[\ii_0]} (\al,\beta)$ in {\rm(ii$'$)}
since
\[  \al+\be = \al+\be^+ - \be^- = \gamma+ \eta - \be^- = \gamma +\nu + \be^- - \be^-=\gamma +\nu. \qedhere \]
\end{proof}

\begin{proof} [{\bf The second step for  Theorem \ref{thm: dist upper bound Qd}}]
From the above propositions, the first and the third assertions are completed.
\end{proof}

\begin{remark} \label{rmK: radius 2}
Note that, for each pair $(\al,\be)$ with $\gdist_{[\ii_0]}(\al,\be)=2$, there exists a non-simple sequence $\um$
$$   \um \prec^\tb_{[\ii_0]} (\al,\be)$$
which tells $\gdist_{[\ii_0]}(\al,\be)=2$. Furthermore,
\begin{enumerate}
\item[{\rm (1)}] if $\al+\be \in \PR$, then $\um$ is a pair consisting of non-induced central vertices,
\item[{\rm (2)}] if $\al+\be \not \in \PR$, $\um$ is a triple $( \mu,\nu,\eta) \in \left(\PR\right)^3$ such that
\begin{itemize}
\item[{\rm (i)}] $\mu+\nu \in \PR$, $(\mu,\nu)$ is an $[\ii_0]$-minimal pair of $\mu+\nu$ and $\al-\mu,\be-\nu \in \PR$,
\item[{\rm (ii)}] $\eta$ is not comparable to $\mu$ and $\nu$ with respect to $\prec_{[\ii_0]}$,
\item[{\rm (iii)}] $\eta=(\al-\mu)+(\be-\nu)$, and $((\al-\mu),(\be-\nu))$ is an $[\ii_0]$-minimal pair for $\eta$,
\item[{\rm (iv)}] $(\al-\mu,\mu)$, $(\nu,\be-\nu)$ are $[\ii_0]$-minimal pairs for $\al$ and $\be$, respectively.
\end{itemize}
In Example \ref{ex: label}, we have
$$ \um=([4,7],[1,3],[2,5]) \prec^\tb_{[\ii_0]} ([2,7],[1,5]).$$
\end{enumerate}
\end{remark}

\subsubsection{Proof of Theorem \ref{thm: dist upper bound Qd} for type $D_{n+1}$} In this subsection, we assume that  $[\ii_0]$ is a twisted adapted class of type $D_{n+1}$ and $\mathfrak{p}^{D_{n+1}}_{A_n}([\ii_0])=[Q]$. The proof mainly uses the properties of $\Gamma_Q$, for example Theorem \ref{thm: known for Q}. We refer to the proof of the theorem (\cite{Oh14A}) for more details.\\

\noindent \textbf{Proof of Theorem \ref{thm: dist upper bound Qd} (2).} It follows by Lemma \ref{Lem:rad_WE} and Lemma \ref{Lem:rad_C} below.

\begin{lemma} \label{Lem:rad_WE}
For a non-simple root $\gamma$ in $\Upsilon^W_{[\ii_0]} \cup \Upsilon^E_{[\ii_0]}$, we have
$$\rds_{[\ii_0]}(\gamma)=1.$$
\end{lemma}

\begin{proof}
Suppose that $\gamma= \lan a, -b-1\ran\in \Upsilon^E_{[\ii_0]}$ and both $(\alpha_1, \beta_1)$ and $(\alpha_2, \beta_2)$ are pairs with weight $\gamma$. Then there are three cases:
\begin{enumerate}[(i)]
\item $\alpha_1$, $\alpha_2$, $\beta_1$, $\beta_2$ are all in  $\Upsilon^E_{[\ii_0]}$.
\item One of the roots, say $\beta_2$ is in  $\Upsilon^W_{[\ii_0]}$ and the others are in $\Upsilon^E_{[\ii_0]}$.
\item Two of the roots, say $\beta_1, \beta_2$ are in  $\Upsilon^W_{[\ii_0]}$ and the others are in $\Upsilon^E_{[\ii_0]}$.
\end{enumerate}

Consider the case (i). Since the labeling of $\Upsilon^E_{[\ii_0]}$ is naturally induced from the labeling of $\Gamma^E_Q$, by Theorem \ref{thm: known for Q}, we can see that $(\alpha_1, \beta_1)$ and $(\alpha_2, \beta_2)$ are incomparable.

In the case of (ii), without loss of generality, two roots, say $\alpha_1$ and $\alpha_2$ are in the same $S$-path, since two roots should share the component $-b-1$. Also, we know that $\alpha_1, \alpha_2, \beta_1 \prec_{[\ii_0]} \beta_2$, since $\alpha_1, \alpha_2, \beta_1 \in \Upsilon^E_{[\ii_0]}$ and $\beta_2 \in \Upsilon^W_{[\ii_0]}$. Hence $(\alpha_1, \beta_1)$ and $(\alpha_2, \beta_2)$ are comparable if and only if $\alpha_2 \prec_{[\ii_0]}\alpha_1, \beta_1$. On the other hand,  properties of $\Gamma_Q$ and the assumption (ii) implies that $\alpha_2 \prec_{[\ii_0]} \gamma  \prec_{[\ii_0]} \alpha_1$ and the two roots $\beta_1$ and $\alpha_2$ are not comparable. In conclusion, $(\alpha_1, \beta_1)$ and $(\alpha_2, \beta_2)$ are not comparable.



In the case of (iii), we have $\beta_1\prec_{[\ii_0]}\beta_2 $ if and only if $\alpha_1\prec_{[\ii_0]}\alpha_2 $, by the property of $\Gamma_Q$. Hence, again, $(\alpha_1, \beta_1)$ and $(\alpha_2, \beta_2)$ are not comparable.

As a conclusion, we have $\rds_{[\ii_0]}(\gamma)=1.$
\end{proof}

\begin{lemma} \label{Lem:rad_C}
For a non-simple root $\gamma$ in $\Upsilon^C_{[\ii_0]}$, we have
$$\rds_{[\ii_0]}(\gamma)=2.$$
\end{lemma}

\begin{proof}
Let us denote $\gamma= \lan a, b\ran$ for $1\leq a, b\leq n.$ In order to show that $\rds_{[\ii_0]}(\gamma)\geq 2$, consider the pairs
$(\lan a, c\ran, \lan b, -c \ran)$ and  $(\lan a, -c\ran, \lan b, c \ran)$ for $b< c \leq n+1.$ Then two pairs are comparable since we have
\[\lan a, c\ran \prec_{[\ii_0]}\lan b, c \ran \text{ if and only if }  \lan a, -c\ran \prec_{[\ii_0]}\lan b, -c \ran. \]

Now it is enough to show that
$(\lan a, c\ran, \lan b, -c \ran)$ and  $(\lan a, d\ran, \lan b, -d \ran)$ are not comparable when $d\neq \pm c.$ It can be proved using the properties of labeling of $\Gamma_Q$, as we did in Lemma \ref{Lem:rad_WE}. Since it is lengthy but straight forwards, we omit the detailed proof.
\end{proof}

\noindent \textbf{Proof of Theorem \ref{thm: dist upper bound Qd} (1).} We shall state the theorem more explicitly in Proposition \ref{Prop:gdist_D}.

\begin{definition} \label{Def:swing} Take $\alpha \in \Phi^+$.
\begin{enumerate}
\item
Suppose that the $N$-path passing $\alpha$ has the vertex with the folded coordinate  $(\hat{n}, p)$. Then the union of the $N$-path and the $S$-path with vertex $(\hat{n}, p-1)$ is called the \defn{$N$-swing associated to $\alpha$}.
 \item Suppose that the $S$-path passing $\alpha$ has the vertex with the folded coordinate  $(\hat{n}, p)$. Then the union of the $S$-path and the $N$-path with vertex $(\hat{n}, p+1)$ is called the \defn{$S$-swing associated to $\alpha$}.
\end{enumerate}
\end{definition}

Using the new notion, we can state the following lemma from Proposition \ref{Prop:share comp}.

\begin{lemma} \label{Lem:swing_property}\hfill
\begin{enumerate}
\item[{\rm (1)}] There are exactly $n$ swings in $\widehat{\Upsilon}_{[\redez]}$.
\item[{\rm (2)}] If there are two distinct swings $\mathsf{S}_\alpha$  and $\mathsf{S}_\beta$ share the components $r_1$  and $r_2$, respectively, then $r_1$ and $r_2$ are distinct elements in $\{1, 2, \cdots, n\}$.
\item[{\rm (3)}] Every vertex in a swing shares a component and all the vertices sharing the component consists of a swing.
\item[{\rm (4)}] If $\mathsf{S}$ is a swing with the shared component $r$ then the only one of the following is true$\colon$
\begin{itemize}
\item[{\rm (a)}] $\mathsf{S}$ has the $N$-path passing $(\hat{1}, p)$ and the $S$-path passing  $(\hat{1}, p+1)$ consists of all the roots with the component $-r$.
\item[{\rm (b)}] $\mathsf{S}$ has the $S$-path passing $(\hat{1}, p)$ and the $N$-path passing  $(\hat{1}, p-1)$ consists of all the roots with the component $-r$.
\end{itemize}
\end{enumerate}
\end{lemma}

\begin{example}
The following quiver is folded AR-quiver $\widehat{\Upsilon}_{[\ii_0]}$ corresponding to the twisted Coxeter element $2135 \vee.$
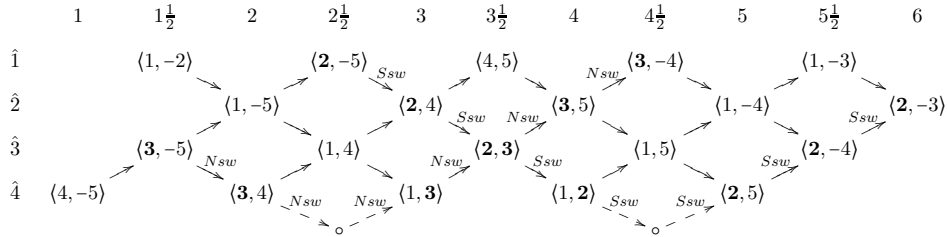
$$
   \raisebox{2.5em}{\scalebox{0.7}{ \xymatrix@C=2ex@R=1ex{
 &  1 & 1\frac{1}{2} & 2 & 2\frac{1}{2} &3 & 3\frac{1}{2} &4 & 4\frac{1}{2} & 5 & 5\frac{1}{2} &  6 &   \\
\hat{1} && \lan1, -2\ran \ar@{->}[dr]  && \lan {\bf 2},-5\ran \ar@{->}[dr]^{Ssw} &&  \lan 4,5 \ran \ar@{->}[dr]  && \lan {\bf 3}, -4 \ran \ar@{->}[dr] && \lan1,-3 \ran \ar@{->}[dr] \\
\hat{2}  &&& \lan1,-5\ran \ar@{->}[dr]\ar@{->}[ur]  &&\lan {\bf 2},4\ran \ar@{->}[dr]^{Ssw}\ar@{->}[ur] &&\lan {\bf 3},5\ran \ar@{->}[dr]\ar@{->}[ur]^{Nsw}  && \lan 1,-4\ran \ar@{->}[dr]\ar@{->}[ur] &&\lan {\bf 2},-3\ran   \\
\hat{3} & & \lan {\bf 3},-5\ran \ar@{->}[dr]^{Nsw}\ar@{->}[ur]  &&\lan 1,4\ran \ar@{->}[dr]\ar@{->}[ur]&&{\bf \lan 2,3\ran} \ar@{->}[ur]^{Nsw}\ar@{->}[dr]^{Ssw} && \lan1,5\ran \ar@{->}[dr]\ar@{->}[ur] && \lan {\bf 2},-4\ran \ar@{->}[ur]^{Ssw} \\
\hat{4} &\lan 4, -5\ran \ar@{->}[ur]&& \lan {\bf 3}, 4\ran \ar@{->}[ur]  \ar@{-->}[dr]^{Nsw}&& \lan 1, {\bf 3}\ran \ar@{->}[ur]^{Nsw}& & \lan 1,{\bf 2} \ran \ar@{-->}[dr]^{Ssw}  \ar@{->}[ur] && \lan {\bf 2},5\ran \ar@{->}[ur]^{Ssw}  \\
&& && \circ \ar@{-->}[ur]^{Nsw} &&&& \circ  \ar@{-->}[ur]^{Ssw}
}}}
$$
Here, the vertices $\circ$ do not exist in $\widehat{\Upsilon}_{[\ii_0]}$. However, in order to show shapes of swings, we put fake vertices $\circ$ in the quiver.
Note that
\begin{itemize}
\item The quiver consists of arrows $\xrightarrow{Nsw}$ is the $N$-swing associated to $\lan 2,3 \ran$.
\item The quiver consists of arrows $\xrightarrow{Ssw}$ is the $S$-swing associated to $\lan 2,3 \ran$.
\item The $N$-swing shares the component $3$ and the $S$-swing shares the component $2$.
\end{itemize}
\end{example}

\begin{lemma} \label{Lem:sigle_pair}
Let $\gamma$ be a root in $\Phi^+\setminus \Pi$.  In $\Upsilon_{[\ii_0]}$, a pair $(\alpha,\beta)$ with weight $\gamma$ is one of $(\alpha_0, \beta_0)$, $(\alpha_1, \beta_1)$,  $(\alpha_2, \beta_2)$, $(\alpha_3, \beta_3)$ and $(\alpha_4, \beta_4)$ in the following picture:

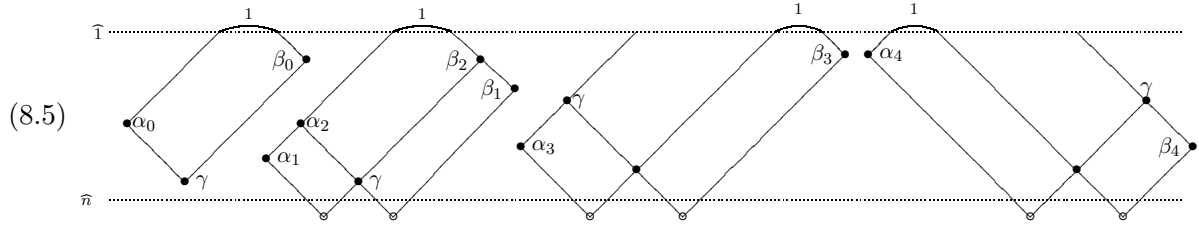
\begin{align}
\scalebox{0.77}
{{\xy
(-20,0)*{}="DL";(-50,-10)*{}="DD";(-40,20)*{}="DT";(10,10)*{}="DR";
"DT"+(-20,-4)-(15,0); "DT"+(150,-4)**\dir{.};
"DD"+(-10,-3)-(15,0); "DD"+(160,-3) **\dir{.};
"DT"+(-22,-4)-(15,0)*{\scriptstyle {\widehat{1}}};
"DT"+(-24,-33)-(15,0)*{\scriptstyle {\widehat{n}}};
%
"DD"+(-30,0)+(6,10)+(2,0); "DD"+(-30,0)+(22,26)+(2,0) **\dir{-};
"DD"+(-30,0)+(37,21)+(2,0); "DD"+(-30,0)+(32,26)+(2,0) **\dir{-};
"DD"+(-30,0)+(16,0)+(2,0); "DD"+(-30,0)+(37,21)+(2,0) **\dir{-};
"DD"+(-30,0)+(16,0)+(2,0); "DD"+(-30,0)+(6,10)+(2,0) **\dir{-};
"DD"+(-30,0)+(37,21)+(2,0)*{\bullet};"DD"+(-30,0)+(33,21)+(2,0)*{\beta_0};
"DD"+(-30,0)+(6,10)+(2,0)*{\bullet}; "DD"+(-30,0)+(9,10)+(2,0)*{\alpha_0};
"DD"+(-30,0)+(16,0)+(2,0)*{\bullet}; "DD"+(-30,0)+(19,0)+(2,0)*{\gamma};
"DD"+(-30,0)+(22,26)+(2,0); "DD"+(-30,0)+(32,26)+(2,0) **\crv{"DD"+(-30,0)+(27,28)+(2,0)};
"DD"+(-30,0)+(27,29)+(2,0)*{\scriptstyle 1};
%
"DD"+(0,0)+(0,4)+(2,0); "DD"+(0,0)+(10,-6)+(2,0) **\dir{-};
"DD"+(0,0)+(16,0)+(2,0); "DD"+(0,0)+(10,-6)+(2,0) **\dir{-};
"DD"+(0,0)+(16,0)+(2,0); "DD"+(0,0)+(22,-6)+(2,0) **\dir{-};
"DD"+(0,0)+(43,16)+(2,0); "DD"+(0,0)+(22,-6)+(2,0) **\dir{-};
"DD"+(0,0)+(0,4)+(2,0); "DD"+(0,0)+(22,26)+(2,0) **\dir{-};
"DD"+(0,0)+(43,16)+(2,0); "DD"+(0,0)+(32,26)+(2,0) **\dir{-};
"DD"+(0,0)+(16,0)+(2,0); "DD"+(0,0)+(37,21)+(2,0) **\dir{-};
"DD"+(0,0)+(16,0)+(2,0); "DD"+(0,0)+(6,10)+(2,0) **\dir{-};
"DD"+(0,0)+(37,21)+(2,0)*{\bullet};"DD"+(0,0)+(33,21)+(2,0)*{\beta_2};
"DD"+(0,0)+(6,10)+(2,0)*{\bullet}; "DD"+(0,0)+(9,10)+(2,0)*{\alpha_2};
"DD"+(0,0)+(16,0)+(2,0)*{\bullet}; "DD"+(0,0)+(19,0)+(2,0)*{\gamma};
"DD"+(0,0)+(10,-6)+(2,0)*{\circ};
"DD"+(0,0)+(22,-6)+(2,0)*{\circ};
"DD"+(0,0)+(22,26)+(2,0); "DD"+(0,0)+(32,26)+(2,0) **\crv{"DD"+(0,0)+(27,28)+(2,0)};
"DD"+(0,0)+(27,29)+(2,0)*{\scriptstyle 1};
"DD"+(0,0)+(0,4)+(2,0)*{\bullet};
"DD"+(0,0)+(4,4)+(2,0)*{ \al_1};
"DD"+(0,0)+(43,16)+(2,0)*{\bullet};
"DD"+(0,0)+(39,16)+(2,0)*{ \be_1};
(40,0)+(-44,-4)*{\bullet}; (40,0)+(-40,-4)*{{\alpha_3}}; (40,0)+(-32,-16)*{\circ}; (40,0)+(-16,-16)*{\circ};(40,0)+(12,12)*{\bullet};(40,0)+(-36,4)*{\bullet};(40,0)+(-34,4)*{\gamma};(40,0)+(8,12)*{{\beta_3}};
(40,0)+(-44,-4);(40,0)+(-36,4)**\dir{-};
(40,0)+(-44,-4);(40,0)+(-32, -16)**\dir{-};
(40,0)+(-36,4);(40,0)+(-16,-16)**\dir{-};
(40,0)+(-32,-16);(40,0)+(0,16)**\dir{-};
(40,0)+(8,16);(40,0)+(12,12)**\dir{-};
(40,0)+(-16,-16);(40,0)+(12,12)**\dir{-};
(40,0)+(0,16); (40,0)+(8,16)**\crv{(40,0)+(4,18)}; (40,0)+(4,20)*{_{1}};
(40,0)+(-24,16);(40,0)+(-36,4)**\dir{-};
(40,0)+(-24,-8)*{\bullet};
(100,0)+(0,-16)*{\circ}; (100,0)+(-16,-16)*{\circ};
(100,0)+(0,-16); (100,0)+(0,-16)+(-32,32)**\dir{-};
(100,0)+(-16,-16);(100,0)+(-16,-16)+(-28,28)**\dir{-};
(100,0)+(0,-16)+(-32,32); (100,0)+(-16,-16)+(-24,32)**\crv{(100,0)+(0,-16)+(-36,34)}; (100,0)+(0,-16)+(-36,36)*{_{1}};
(100,0)+(-16,-16)+(-28,28)*{\bullet};
(100,0)+(-16,-16)+(-28,28)+(4,0)*{\alpha_4};
(100,0)+(-16,-16)+(-28,28);(100,0)+(-16,-16)+(-24,32)**\dir{-};
(100,0)+(-16,-16)+(8,8)*{\bullet};
(100,0)+(-16,-16);(100,0)+(-16,-16)+(8,8)**\dir{-};
(100,0)+(-16,-16)+(8,8);(100,0)+(-16,-16)+(8,8)+(12,12)**\dir{-};
(100,0)+(0,-16);(100,0)+(0,-16)+(12,12)**\dir{-};
(100,0)+(0,-16)+(12,12)*{\bullet};
(100,0)+(0,-16)+(8,12)*{\beta_4};
(100,0)+(-16,-16)+(8,8)+(12,12)*{\bullet};
(100,0)+(-16,-16)+(8,8)+(12,14)*{\gamma};
(100,0)+(0,-16)+(12,12); (100,0)+(0,-16)+(12,12)+(-20, 20)**\dir{-};
\endxy}}
\end{align}
In the case of $(\alpha_0, \beta_0)$,  we assume only one of $\alpha_0$ and $\beta_0$ shares a swing with $\gamma$. Hence $(\alpha_0, \beta_0)$ and $(\alpha_2, \beta_2)$ indicate different cases.
\end{lemma}

\begin{proof}
Let us denote $\gamma=\lan a, b\ran$. Then a pair $(\alpha, \beta)$ has weight $\gamma$ if and only if (i) there is $c\neq \pm a,\pm b$ such that $\alpha$ has $c$ as a component and $\beta$ has $-c$ as a component (ii) $\alpha$ and $\gamma$ share a component (iii) $\beta$  and $\gamma$ share the other component of $\gamma$.

Consider the case when $1\leq a<b\leq n$ and $c>0$. By Lemma \ref{Lem:swing_property}, we know the following facts:
\begin{itemize}
\item $\alpha$ and $\gamma$ share a swing, namely $\mathsf{S}_\al$.
\item $\beta$ is in the other swing, namely $\mathsf{S}_\be$, which also passes $\gamma$.
\item Consider the other swing  $\mathsf{S}'_\al  \neq \mathsf{S}_\al$ crossing $\alpha$  and the sectional path $\mathsf{P}'_\be$ passing $\beta$ which is not in $\mathsf{S}_\be$. Then $\mathsf{S}'_\al$ and $\mathsf{P}'_\be$ have the property in Lemma \ref{Lem:swing_property} (4).
\end{itemize}
Similarly, we can deal with the case $1\leq a<b\leq n$ and $c<0$. Hence, when  $1\leq a<b\leq n$, we can show $(\alpha,\beta)$ should be one of $(\alpha_1, \beta_1)$,  $(\alpha_2, \beta_2)$, $(\alpha_3, \beta_3)$ and $(\alpha_4, \beta_4)$ in the picture.

\vskip 2mm

Now, let us consider the case when $b<0$ or $b=n+1$ and suppose $\alpha$ and $\gamma$ share the component  $b$.
Then, by Lemma \ref{Lem:swing_property}, we know the following facts:
\begin{itemize}
\item $\alpha$ and $\gamma$ share a sectional path, which is not contained in a swing.
\item $\beta$ and $\gamma$ share a swing, namely $\mathsf{S}_\be$.
\item Consider the swing associated to $\alpha$ and the sectional path passing $\beta$ which is not in the swing. Then they have the property in Lemma \ref{Lem:swing_property} (4).
\end{itemize}
We can do the similar thing when $\alpha$ and $\gamma$ share the component $a$. Now, we can show $(\alpha,\beta)$ should be one of $(\alpha_0, \beta_0)$,  $(\alpha_3, \beta_3)$ and $(\alpha_4, \beta_4)$ in the picture.
\end{proof}

\begin{proposition} \label{Prop:gdist_D}
Let $\alpha, \beta \in \Phi^+$ satisfy $\beta\prec_{[\ii_0]} \alpha$. The following pictures show the sectional paths and swings passing $\alpha$ or $\beta$. The value of $\gdist(\alpha,\beta)$ is determined as follows.

\begin{align} \label{Pic:gdist_D_1}
\scalebox{0.77}{{\xy
(-20,0)*{}="DL";(-10,-10)*{}="DD";(0,20)*{}="DT";(10,10)*{}="DR";
"DT"+(-30,-4)-(15,0); "DT"+(135,-4)**\dir{.};
"DD"+(-20,-3)-(15,0); "DD"+(145,-3) **\dir{.};
"DT"+(-32,-4)-(15,0)*{\scriptstyle {\widehat{1}}};
"DT"+(-34,-33)-(15,0)*{\scriptstyle {\widehat{n}}};
"DD"+(0,4)+(37,0); "DD"+(10,-6)+(37,0) **\dir{-};
"DD"+(16,0)+(37,0); "DD"+(10,-6)+(37,0) **\dir{-};
"DD"+(16,0)+(37,0); "DD"+(22,-6)+(37,0) **\dir{-};
"DD"+(43,16)+(37,0); "DD"+(22,-6)+(37,0) **\dir{-};
"DD"+(0,4)+(37,0); "DD"+(22,26)+(37,0) **\dir{-};
"DD"+(43,16)+(37,0); "DD"+(32,26)+(37,0) **\dir{-};
"DD"+(16,0)+(37,0); "DD"+(37,21)+(37,0) **\dir{-};
"DD"+(16,0)+(37,0); "DD"+(6,10)+(37,0) **\dir{-};
"DD"+(16,0)+(37,0)*{\bullet}; "DD"+(16,0)+(41,0)*{\gamma};
"DD"+(37,21)+(37,0)*{\bullet};
"DD"+(34,21)+(37,0)*{\xi};
"DD"+(6,10)+(37,0)*{\bullet};
"DD"+(9,10)+(37,0)*{\eta};
"DD"+(10,-6)+(37,0)*{\circ};
"DD"+(22,-6)+(37,0)*{\circ};
"DD"+(22,26)+(37,0); "DD"+(32,26)+(37,0) **\crv{"DD"+(27,28)+(37,0)};
"DD"+(27,29)+(37,0)*{\scriptstyle m>1};
"DD"+(0,4)+(37,0)*{\bullet};
"DL"+(27,0)+(37,0)*{{\rm(I-3)}};
"DD"+(4,4)+(37,0)*{\scriptstyle \al};
"DD"+(43,16)+(37,0)*{\bullet};
"DD"+(39,16)+(37,0)*{\scriptstyle \be};
"DL"+(27,0)+(37,-20)*{\gdist(\alpha,\beta)=1};
"DD"+(0,4)+(2,0); "DD"+(10,-6)+(2,0) **\dir{-};
"DD"+(16,0)+(2,0); "DD"+(10,-6)+(2,0) **\dir{-};
"DD"+(16,0)+(2,0); "DD"+(22,-6)+(2,0) **\dir{-};
"DD"+(43,16)+(2,0); "DD"+(22,-6)+(2,0) **\dir{-};
"DD"+(0,4)+(2,0); "DD"+(22,26)+(2,0) **\dir{-};
"DD"+(43,16)+(2,0); "DD"+(32,26)+(2,0) **\dir{-};
"DD"+(16,0)+(2,0); "DD"+(37,21)+(2,0) **\dir{-};
"DD"+(16,0)+(2,0); "DD"+(6,10)+(2,0) **\dir{-};
"DD"+(37,21)+(2,0)*{\bullet};"DD"+(33,21)+(2,0)*{\xi};
"DD"+(6,10)+(2,0)*{\bullet}; "DD"+(9,10)+(2,0)*{\eta};
"DD"+(16,0)+(2,0)*{\bullet}; "DD"+(19,0)+(2,0)*{\gamma};
"DD"+(10,-6)+(2,0)*{\circ};
"DD"+(22,-6)+(2,0)*{\circ};
"DD"+(22,26)+(2,0); "DD"+(32,26)+(2,0) **\crv{"DD"+(27,28)+(2,0)};
"DD"+(27,29)+(2,0)*{\scriptstyle 1};
"DD"+(0,4)+(2,0)*{\bullet};
"DL"+(27,0)+(2,0)*{{\rm(I-2)}};
"DD"+(4,4)+(2,0)*{\scriptstyle \al};
"DD"+(43,16)+(2,0)*{\bullet};
"DD"+(39,16)+(2,0)*{\scriptstyle \be};
"DL"+(27,0)+(2,-20)*{\gdist(\alpha,\beta)=2};
"DD"+(35,4)-(70,0); "DD"+(45,-6)-(70,0) **\dir{-};
"DD"+(51,0)-(70,0); "DD"+(45,-6)-(70,0) **\dir{-};
"DD"+(51,0)-(70,0); "DD"+(57,-6)-(70,0) **\dir{-};
"DD"+(70,7)-(70,0); "DD"+(57,-6) -(70,0)**\dir{-};
"DD"+(35,4)-(70,0); "DD"+(54,23)-(70,0) **\dir{-};
"DD"+(70,7)-(70,0); "DD"+(54,23)-(70,0) **\dir{-};
"DD"+(51,0)-(70,0); "DD"+(64,13) -(70,0)**\dir{-};
"DD"+(51,0)-(70,0); "DD"+(41,10)-(70,0) **\dir{-};
"DD"+(41,10)-(70,0)*{\bullet}; "DD"+(44,10)-(70,0)*{\eta};
"DD"+(64,13) -(70,0)*{\bullet}; "DD"+(61,13) -(70,0)*{\xi};
"DD"+(35,4)-(70,0)*{\bullet};
"DD"+(54,23)-(70,0)*{\bullet}; "DD"+(57,23)-(70,0)*{\delta};
"DD"+(45,-6)-(70,0)*{\circ};
"DD"+(57,-6)-(70,0)*{\circ};
"DD"+(51,0)-(70,0)*{\bullet}; "DD"+(54,0)-(70,0)*{\gamma};
"DL"+(62,0)-(70,0)*{{\rm(I-1)}};
"DD"+(38,4)-(70,0)*{\scriptstyle \al};
"DD"+(70,7)-(70,0)*{\bullet};
"DD"+(66,7)-(70,0)*{\scriptstyle \be};
"DL"+(62,0)-(70,20)*{\gdist(\alpha,\beta)=2};
"DD"+(74,4); "DD"+(84,-6) **\dir{-};
"DD"+(74,4); "DD"+(96,26) **\dir{-};
"DD"+(109,19); "DD"+(102,26) **\dir{-};
"DD"+(84,-6); "DD"+(109,19) **\dir{-};
"DD"+(74,4)*{\bullet};
"DL"+(101,0)*{{\rm(I-4)}};
"DD"+(78,4)*{\scriptstyle \al};
"DD"+(109,19)*{\bullet};
"DD"+(109,17)*{\scriptstyle \be};
"DD"+(84,-6)*{\circ};
"DD"+(99,4); "DD"+(109,-6) **\dir{-};
"DD"+(99,4); "DD"+(118,23) **\dir{-};
"DD"+(128,13); "DD"+(118,23) **\dir{-};
"DD"+(109,-6); "DD"+(128,13) **\dir{-};
"DD"+(99,4)*{\bullet};
"DD"+(118,23)*{\bullet}; "DD"+(121,23)*{\delta};
"DL"+(126,0)*{{\rm(I-5)}};
"DD"+(103,4)*{\scriptstyle \al};
"DD"+(128,13)*{\bullet};
"DD"+(128,11)*{\scriptstyle \be};
"DD"+(109,-6)*{\circ};
"DD"+(96,-10)*{\gdist(\alpha,\beta)=0};
\endxy}}
\end{align}

\begin{align} \label{Pic:gdist_D_1-2}
\scalebox{0.77}{{\xy
(-20,0)*{}="DL";(-10,-10)*{}="DD";(0,20)*{}="DT";(10,10)*{}="DR";
"DT"+(-30,-4)-(15,0); "DT"+(135,-4)**\dir{.};
"DD"+(-20,-3)-(15,0); "DD"+(145,-3) **\dir{.};
"DT"+(-32,-4)-(15,0)*{\scriptstyle {\widehat{1}}};
"DT"+(-34,-33)-(15,0)*{\scriptstyle {\widehat{n}}};
(-44,-4)*{\bullet}; (-42,-4)*{_{\alpha}}; (-32,-16)*{\circ}; (-16,-16)*{\circ};(12,12)*{\bullet};(-36,4)*{\bullet};(-34,4)*{\eta};(10,12)*{_{\beta}};
(-44,-4);(-36,4)**\dir{-};
(-44,-4);(-32, -16)**\dir{-};
(-36,4);(-16,-16)**\dir{-};
(-32,-16);(0,16)**\dir{-};
(8,16);(12,12)**\dir{-};
(-16,-16);(12,12)**\dir{-};
(0,16); (8,16)**\crv{(4,18)}; (4,20)*{_{1}};
(-24,16);(-36,4)**\dir{-};
(-24,-8)*{\bullet};
(-24,0)*{\rm{(I-6)}};
(-24,-20)*{\gdist(\alpha,\beta)=1};
(-2,-6)*{\bullet}; (0,-6)*{_{\alpha}}; (8,-16)*{\circ}; (28,-16)*{\circ};(56,12)*{\bullet};(8,4)*{\bullet};(54,12)*{_{\beta}};
(-2,-6);(8,4)**\dir{-};
(-2,-6);(8, -16)**\dir{-};
(8,4);(28,-16)**\dir{-};
(8,-16);(40,16)**\dir{-};
(52,16);(56,12)**\dir{-};
(28,-16);(56,12)**\dir{-};
(40,16); (52,16)**\crv{(46,18)}; (46,20)*{_{m>1}};
(20,16);(8,4)**\dir{-};
(18,-6)*{\bullet};
(18,3)*{\rm{(I-7)}};
(18,-20)*{\gdist(\alpha,\beta)=0};
(56,8)*{\bullet};(58,8)*{_{\alpha}}; (80, -16)*{\circ};(104,-16)*{\circ};(132,12)*{\bullet};(130,12)*{_{\beta}};
(56,8);(64,16)**\dir{-};
(56,8);(80,-16)**\dir{-};
(72,16);(104,-16)**\dir{-};
(80,-16);(112,16)**\dir{-};
(128,16);(132,12)**\dir{-};
(104,-16);(132,12)**\dir{-};
(92,-4)*{\bullet};
(92,4)*{\rm{(I-8)}};
(92,-20)*{\gdist(\alpha,\beta)=0};
\endxy}}
\end{align}

\begin{align} \label{eq: complacted socle of D}
\scalebox{0.77}{{\xy
(-20,0)*{}="DL";(-10,-10)*{}="DD";(0,20)*{}="DT";(10,10)*{}="DR";
"DT"+(-40,-4); "DT"+(95,-4)**\dir{.};
"DD"+(-30,-10); "DD"+(105,-10) **\dir{.};
"DT"+(-42,-4)*{\scriptstyle {\widehat{1}}};
"DT"+(-44,-40)*{\scriptstyle {\widehat{n}}};
"DL"+(-10,-3); "DD"+(-10,-3) **\dir{-};
"DR"+(-14,-7); "DD"+(-10,-3) **\dir{-};
"DT"+(-14,-7); "DR"+(-14,-7) **\dir{-};
"DT"+(-14,-7); "DL"+(-10,-3) **\dir{-};
"DL"+(-6,-3)*{\scriptstyle \al};
"DR"+(-18,-7)*{\scriptstyle \be};
"DL"+(2,0)*{{\rm (II-1)}};
"DL"+(-10,-3)*{\bullet};"DR"+(-14,-7)*{\bullet};
"DT"+(-14,-7)*{\bullet}; "DD"+(-10,-3)*{\bullet};
"DT"+(-14,-5)*{\scriptstyle \delta};
"DD"+(-10,-5)*{\scriptstyle \gamma};
"DD"+(-10,-13)*{\gdist(\alpha,\beta)=1};
"DD"+(69,4)-(50,0); "DD"+(79,-6)-(50,0) **\dir{-};
"DD"+(85,0)-(50,0); "DD"+(79,-6)-(50,0) **\dir{-};
"DD"+(69,4)-(50,0); "DD"+(91,26)-(50,0) **\dir{-};
"DD"+(104,19)-(50,0); "DD"+(97,26)-(50,0) **\dir{-};
"DD"+(85,0)-(50,0); "DD"+(104,19)-(50,0) **\dir{-};
"DD"+(69,4)-(50,0)*{\bullet};
"DL"+(96,0)-(52,0)*{{\rm (II-2)}};
"DD"+(73,4)-(51,0)*{\scriptstyle \al};
"DD"+(91,26)-(50,0); "DD"+(97,26)-(50,0) **\crv{"DD"+(94,28)-(50,0)};
"DD"+(104,19)-(50,0)*{\bullet};
"DD"+(104,17)-(50,0)*{\scriptstyle \be};
"DD"+(79,-6)-(50,0)*{\bullet};
"DD"+(79,-8)-(50,0)*{\scriptstyle \gamma};
"DD"+(94,29)-(50,0)*{\scriptstyle 1};
"DD"+(69,4); "DD"+(79,-6) **\dir{-};
"DD"+(85,0); "DD"+(79,-6) **\dir{-};
"DD"+(79,-8)-(50,5)*{\gdist(\alpha,\beta)=1};
"DD"+(65,8); "DD"+(79,-6) **\dir{-};
"DD"+(65,8); "DD"+(83,26) **\dir{-};
"DD"+(104,19); "DD"+(97,26) **\dir{-};
"DD"+(85,0); "DD"+(104,19) **\dir{-};
"DD"+(65,8)*{\bullet};
"DL"+(92,0)*{{\rm (II-3)}};
"DD"+(71,8)*{\scriptstyle \al};
"DD"+(83,26); "DD"+(97,26) **\crv{"DD"+(90,28)};
"DD"+(104,19)*{\bullet};
"DD"+(104,17)*{\scriptstyle \be};
"DD"+(79,-6)*{\bullet};
"DD"+(77,-8)*{\scriptstyle \gamma};
"DD"+(94,29)*{\scriptstyle m>1};
"DD"+(77,-13)*{\gdist(\alpha,\beta)=0};
\endxy}}
\end{align}

\end{proposition}

\begin{proof}
By Proposition \ref{Prop:share comp} and Lemma \ref{Lem:add_D}, we can check that
\begin{equation} \label{Eqn:D_gdist_2}
\alpha+\beta = \eta+\xi=  \gamma+\delta, \text{ and } \alpha+\beta = \eta+\xi=  \gamma
\end{equation}
in (I-1) and (I-2), respectively. By Lemma \ref{Lem:add_D}, we have
\begin{equation}\label{Eqn:D_gdist_1}
 \alpha+\beta= \eta+\xi,\ \eta, \, \gamma+\delta, \text{ and } \gamma
 \end{equation}
in (I-3), (I-6), (II-1), and (II-2), respectively.

 Now, it is enough to show that sequences in (\ref{Eqn:D_gdist_2}) and (\ref{Eqn:D_gdist_1}) are all we need to consider. We can check that if a sequence $\um$ consists only one or two roots and is smaller than $(\alpha,\beta)$ with respect to $ \prec_{[\ii_0]}^\tb $ then $\um$ is one of sequences we listed in (\ref{Eqn:D_gdist_2}) or (\ref{Eqn:D_gdist_1}). (Here we omit the detailed proof but the main idea is the same as the argument in Lemma \ref{Lem:sigle_pair}.)

The last thing we need to show is that there is no sequence $\um$ which consists of more than two roots and satisfies  $\um \prec_{[\ii_0]}^\tb (\alpha,\beta)$. Suppose it is not true in the case (I-1). This implies that there is a triple $ \left( \PR \right)^3 \ni (\gamma_1, \gamma_2, \gamma_3) \prec_{[\ii_0]}^\tb (\alpha,\beta).$ Since \[\gamma_1+\gamma_2+\gamma_3= \alpha+\beta= \gamma+\delta= \eta+\xi,\] the sum of two roots, say  $\gamma_1+\gamma_2$, in $(\gamma_1, \gamma_2, \gamma_3)$ should be equal to $\alpha,\beta, \gamma, \delta, \eta$ or $\xi$.
However, by Lemma \ref{Lem:sigle_pair}, any of $\alpha,\beta, \gamma, \delta, \eta$ and $\xi$ cannot be obtained as a sum of two roots which are both smaller than $\alpha$ and bigger than $\beta$ with respect to $\prec_{[\ii_0]}$. Hence it contracts to our assumption and there is no such triple. As a conclusion, for the case (I-1), we have $\gdist(\alpha,\beta)=2.$

Similarly for other cases, we can prove that there exists no sequence $\um$ which consists of more than two roots and satisfies  $\um \prec_{[\ii_0]}^\tb (\alpha,\beta)$ by Lemma \ref{Lem:sigle_pair}.  Hence we proved the proposition.
\end{proof}

\noindent \textbf{Proof of Theorem \ref{thm: dist upper bound Qd} (3).} In each case of Proposition \ref{Prop:gdist_D}, there is a unique socle $\soc(\alpha,\beta)$

\begin{remark} \label{rmK: radius 2 d}
As in  Remark \ref{rmK: radius 2} for type $A_{2n+1}$, for each pair $(\alpha,\beta)$ with $\gdist_{[\ii_0]}(\alpha, \beta)=2$, there exists unique chain of  non-simple sequences
\[ \um_1 \prec^\tb_{[\ii_0]} (\alpha,\beta)\]
which tells $\gdist_{[\ii_0]}(\alpha, \beta)=2$. (See the cases of (I-1) and (I-2) of (\ref{Pic:gdist_D_1}) in Proposition \ref{Prop:gdist_D}.)
\end{remark}

\subsubsection{Proof of Theorem \ref{thm: dist upper bound Qd} for exceptional cases}
For $[\ii_0] \in \lf \Qd \rf$ (resp.  $[\ii_0] \in \lf \mathfrak{Q} \rf$) of type $E_6$ with respect $\vee$ in \eqref{eq: F_4} (resp.
$D_4$ with respect $\vee$ in \eqref{eq: G_2}) , we can check that
\begin{itemize}
\item[{\rm (1)}] For any pair $(\al,\be) \in (\PR)^2$ $($not $(\widehat{\Phi}^+)^2)$, we have $ 0 \le \gdist_{[\ii_0]}(\al,\be) \le 2 \le \mathsf{d}$.
\item[{\rm (2)}] For any $\gamma \in \PR \setminus \Pi$, we have  $1 \le \rds_{[\ii_0]}(\gamma) \le 2 \le \mathsf{d}$,
\end{itemize}
by observing all twisted AR-quivers.

\subsection{Folded distance polynomials on $\lf \Qd \rf$ and $\lf \mathfrak{Q} \rf$}  In this subsection, we define folded distance polynomials by considering
folded AR-quiver $\widehat{\Upsilon}_{[\ii_0]}$ and prove that the
folded distance polynomials are well-defined on $\lf \Qd \rf$ (resp. $\lf \mathfrak{Q} \rf$).

\begin{definition}
For a folded AR-quiver $\wUp_{[\ii_0]}$, indices $\hat{k},\hat{l} \in \widehat{I}$ and $t \in \mathbb{N}/\mathsf{d}$,
we define the subset $\Phi_{[\ii_0]}(\hat{k},\hat{l})[t]$ of $(\PR)^2$ as follows:

A pair $(\alpha,\beta)$ is contained in $\Phi_{[\ii_0]}(\hat{k},\hat{l})[t]$ if $\alpha \prec_{[\ii_0]} \beta$ or $\be \prec_{[\ii_0]} \al$ and
$$\{ \widehat{\Omega}_{[\ii_0]}(\al),\widehat{\Omega}_{[\ii_0]}(\be) \}=\{ (\hat{k},a), (\hat{l},b)\} \quad \text{ such that } \quad |a-b|=t.$$
\end{definition}

\begin{lemma} \label{lem: folded exponent}
For any  $(\alpha^{(1)},\beta^{(1)})$ and $(\alpha^{(2)},\beta^{(2)})$ in $\Phi_{[\ii_0]}(\hat{k},\hat{l})[t]$, we have
$$ o^{[\ii_0]}_t(\hat{k},\hat{l}) := \gdist_{[\ii_0]}(\alpha^{(1)},\beta^{(1)})=\gdist_{[\ii_0]}(\alpha^{(2)},\beta^{(2)}). $$
\end{lemma}

\begin{proposition} \label{prop: DQ DQ' twisted}
The integer, defined by
$$\mathtt{o}^{[\ii_0]}_t(\hat{k},\hat{l}) \seteq \left\lceil \dfrac{o^{[\ii_0]}_t(\hat{k},\hat{l})}{ \mathsf{d} } \right\rceil,$$
does not depend on the choice of $[\ii_0] \in \lf \Qd \rf$; that is,
$$\mathtt{o}^{[\ii_0]}_t(\hat{k},\hat{l}) = \mathtt{o}^{[\ii'_0]}_t(\hat{k},\hat{l})$$
for any distinct Dynkin quivers $[\ii_0],[\ii'_0] \in \lf \Qd \rf$ of the same type.
\end{proposition}

From Proposition \ref{prop: DQ DQ' twisted}, we can define $\widehat{D}_{\hat{k},\hat{l}}(z)$ for the twisted adapted $r$-cluster point $\lf \Qd \rf$ as in the below, and call it the folded distance polynomial at $\hat{k}$ and $\hat{l}$.

\begin{definition} \label{def: Dist poly Q}
For any $\hat{k},\hat{l} \in \widehat{I}$ and folded AR-quiver, we define \defn{the folded distance
polynomial} $\widehat{D}^X_{\hat{k},\hat{l}}(z) \in \ko[z]$ on $\lf \Qd \rf$ as follows:
$$\widehat{D}^X_{\hat{k},\hat{l}}(z) \seteq
\begin{cases}
\displaystyle \prod_{ \frac{t}{\mathsf{d}} \in \frac{1}{\mathsf{d}}\mathbb{N}}
(z- (-1)^{\hat{k}+\hat{l}}(q^{1/\mathsf{d}})^{t})^{\mathtt{o}^{[\ii_0]}_t(\hat{k},\hat{l})}, & \text{ if  $\vee$ is \eqref{eq: B_n} or \eqref{eq: F_4}}, \\
\displaystyle \prod_{ \frac{t}{\mathsf{d}} \in \frac{1}{\mathsf{d}}\mathbb{N}} (z- (-q^{1/\mathsf{d}})^{t})^{\mathtt{o}^{[\ii_0]}_t(\hat{k},\hat{l})}, & \text{ if $\vee$ is \eqref{eq: C_n} or \eqref{eq: G_2} }.
 \end{cases}
$$
Here $X$ denotes the type of $\lf \Qd \rf$ or $\lf \mathfrak{Q} \rf$.
\end{definition}

\subsubsection{Type $A_{2n+1}$} Recall that the indices of $\widehat{I}$ are given as follows
$$ \widehat{I}=\{ 1,2,\ldots, n,n+1\}.$$

\begin{proof}[Proof of Lemma \ref{lem: folded exponent} for $A_{2n+1}$ case]
{\rm (1)} Assume that $\hat{k},\hat{l} \in \widehat{I} \setminus \{ n+1 \}$. By Theorem \ref{thm: known for Q}, the set $\Phi_{[\ii_0]}(\hat{k},\hat{l})[t]$
is induced from one of
\begin{align} \label{eq: sets}
\Phi_{Q}(k,l)[t] \sqcup  \Phi_{Q}(k^*,l^*)[t] \text{ and } \Phi_{Q}(k^*,l)[t] \sqcup  \Phi_{Q}(k,l^*)[t].
\end{align}
where $t \in \mathbb{N}$, $\PPi([\ii_0])=[Q]$ and $i  \overset{*}{\leftrightarrow} 2n+1-i$. Note that one of the sets in \eqref{eq: sets} must be empty by the parity of $t$.

In each case, if there exists a path from $\beta^{(1)}$ to $\al^{(1)}$ passing through two non-induced vertices, then so is $(\alpha^{(2)},\beta^{(2)})$.
Then our assertion for this case follows from Corollary \ref{cor:label1}, Theorem \ref{thm: dist upper bound Qd} and Lemma \ref{lem: o well}.

{\rm (2)} Assume that $\hat{k}=\hat{l}=\{ n+1 \}$. By Lemma \ref{cor: label for non-induced}
either {\rm (i)} $\al^{(j)}+\be^{(j)} \in \PR$ or {\rm (ii)} $\al^{(j)}+\be^{(j)} \not \in \PR$ and they shares one component.
Then, for all $j$, Lemma \ref{lem: less than eq 2} tells that we have
$$
\gdist_{[\ii_0]}(\al^{(j)},\be^{(j)})=
\begin{cases}
1 & \text{  if {\rm (i)}}, \\
0 & \text{  if {\rm (ii)}}.
\end{cases}
$$

{\rm (3)} Assume that one of $\hat{k}$ and $\hat{l}$ is $n+1$ and the other is not. Then our assertion follows from Lemma \ref{lem: non-induced} and (2) in the proof of
Proposition \ref{prop: less than eq to 2}, since they just consider the local condition determined by the pair $(\al,\be)$.
\end{proof}

\begin{proof}[Proof of Proposition \ref{prop: DQ DQ' twisted}]
It is enough to consider when $[\ii'_0]=[\ii_0]r_i$. Then our assertion is obvious for $i \ne n+1$
by Algorithm \ref{Rem:surgery A}, Proposition \ref{prop: DQ DQ' 1A2n} and Theorem \ref{thm: dist upper bound Qd};
that is, $  o^{[\ii_0]}_t(\hat{k},\hat{l}) \ne 0$ implies
$$ \text{ (i) $ o^{[\ii_0]}_t(\hat{k},\hat{l})=1$ and (ii) $o^{[\ii_0]}_t(\hat{k},\hat{l}) \ne 0$  if and only if  $o^{[\ii_0]r_i}_t(\hat{k},\hat{l}) \ne 0$}.$$
When $i =n+1$ is also obvious from the fact that $[Q^>]r_{n+1}=[Q^<]$.
\end{proof}

\subsubsection{Type $D_{n+1}$}

\begin{proof}[Proof of Lemma \ref{lem: folded exponent} and Proposition \ref{prop: DQ DQ' twisted} for $D_{n+1}$ case]
By Proposition \ref{Prop:gdist_D}, $\gdist_{[\ii_0]}(\alpha,\beta)$ is determined by their relative positions
for any $[\ii_0] \in \lf \Qd \rf$. Hence
our assertion follows the fact that $\gdist(\al,\be)=0,1$ or $2$
\end{proof}

\subsubsection{Remained types} By checking all folded AR-quivers for remained types, one can easily check that
Lemma \ref{lem: folded exponent} and Proposition \ref{prop: DQ DQ' twisted} hold for the cases, also.

\subsection{Minimal pairs on $\lf \Qd \rf$} In this subsection, we shall record
the relative positions of $(\al,\be)$ which is an $[\ii_0]$-minimal pair of some $\ga \in \PR$ for $[\ii_0] \in \lf
\Qd \rf$ of type $A_{2n+1}$ and $D_{n+1}$. Due to the well-definedness of folded distance polynomials, the relative positions do not depend on the choice of $[\ii_0]$.

\subsubsection{$A_{2n+1}$}

\begin{theorem} \cite[Theorem 3.2, Theorem 3.4]{Oh14A} \label{thm upper lower minimal}
For  a Dynkin quiver $Q$ of type $A_{2n}$ and every pair $(\alpha,\beta)$ of $\alpha+\beta=\gamma\in \Phi^+_{A_{2n}}$, we write
$$\Omega_Q(\alpha)=(i,p), \quad \Omega_Q(\be)=(j,q) \quad \text{ and } \quad \Omega_Q(\ga)=(k,z).$$
Then $(\alpha,\beta)$ is $[Q]$-minimal and
\begin{eqnarray} &&
\parbox{79ex}{
\begin{enumerate}
\item[{\rm (i)}] $p-z=|i-k|$ and $q-z=-|j-k|$,
\item[{\rm (ii)}] $i+j=k$ or $(2n+1-i)+(2n+1-j)=(2n+1-k).$
\end{enumerate}
}\label{eq: upper and lower}
\end{eqnarray}
\end{theorem}

Define
$$
i^- =
\begin{cases}
i-1 & \text{ if } i>n+1, \\
i & \text{ if } i \le n+1.
\end{cases}
$$

\begin{lemma} \label{lem: minimal for rds2}
For $\ga \in \PR \setminus \Pi$ with $\Omega_{[\ii_0]}(\ga)=(k,r)$ and $\rds_{[\ii_0]}(\ga)=2$, an $[\ii_0]$-minimal pair $(\al,\be)$ for $\ga$ satisfies
one of the following conditions$:$ Set $\Omega_{[\ii_0]}(\al)=(i,p)$ and $\Omega_{[\ii_0]}(\be)=(j,q)$.
\begin{enumerate}
\item[{\rm (i)}] $i=j=n+1$ such that $p+q=2r$,
\item[{\rm (ii)}] $q=|k^--i^-|+p, \ r=p-|k^--j^-|$, $i,j \ne n+1$ and one of the following holds:
\begin{align} \label{eq: induced minimal}
\begin{cases}
{\rm (a)} \ i+j=k \text{ and } k \le n, \\
{\rm (b)} \ (2n+1-i^-)+(2n+1-j^-)=2n+1-k^-, \  k \le n \text{ and } \min\{ i,j\} \le n, \\
{\rm (c)} \ i^-+j^-=k^-, \  k \ge n+2 \text{ and } \max\{ i,j \} \ge n+2, \\
{\rm (d)} \ (2n+2-i)+(2n+2-j)=2n+2-k \text{ and } k \ge n+2.
\end{cases}
\end{align}
\end{enumerate}
\end{lemma}

\begin{proof}
By Lemma \ref{lem: less than eq 2}, $\ga$ with $\rds_{[\ii_0]}(\ga)=2$ has a unique pair $(\al,\be)$
which consists of non-induced central vertices and is an $[\ii_0]$-minimal pair for $\ga$. Then the first assertion follows. The other $[\ii_0]$-minimal pairs for $\ga$
are induced from $\Gamma_Q$ and incomparable with each others. Then one can easily check that the other $[\ii_0]$-minimal pairs
satisfy one of the four conditions in \eqref{eq: induced minimal} by Theorem \ref{thm upper lower minimal}.
\end{proof}

\begin{lemma} \label{lem: minimal for induced rds1}
For an induced vertex $\ga \in \PR \setminus \Pi$ with $\Omega_{[\ii_0]}(\ga)=(k,r)$ and $\rds_{[\ii_0]}(\ga)=1$,
an $[\ii_0]$-minimal pair $(\al,\be)$ for $\ga$ satisfies the following conditions $:$
Set $\Omega_{[\ii_0]}(\al)=(i,p)$ and  $\Omega_{[\ii_0]}(\be)=(j,q)$. Then
$p-r=|k^- -i^-|$,  $q-r=-|k^--j^-|$ and
\[i^-+j^-=k^- \  \text{ or } \ (2n+1-i^-)+(2n+1-j^-)=2n+1-k^-.\]\
\end{lemma}
\begin{proof}
One can see that all pairs for $\ga$ are induced from $\Gamma_Q$ and they not comparable with each others. Then we can apply the same argument of
the previous lemma.
\end{proof}

\begin{lemma} \label{lem: minimal for non-induced central}
For a non-induced central vertex $\ga \in \PR \setminus \Pi$ with $\Omega_{[\ii_0]}(\ga)=(n+1,r)$,
an $[\ii_0]$-minimal pair $(\al,\be)$ for $\ga$ satisfies one of the following conditions$:$
Set $\Omega_{[\ii_0]}(\al)=(i,p)$ and  $\Omega_{[\ii_0]}(\be)=(j,q)$.
\begin{enumerate}
\item[{\rm (i)}] $(i,p)=(\ell,r+\frac{1}{2}+(n-\ell))$ and $(j,q)=(n+1,r-2\ell)$,
\item[{\rm (ii)}] $(i,p)=(2n+2-\ell,r+\frac{1}{2}+(n-\ell))$ and $(j,q)=(n+1,r-2\ell)$,
\item[{\rm (iii)}] $(i,p)=(n+1,r+2\ell)$ and $(j,q)=(\ell,r-\frac{1}{2}-(n-\ell))$,
\item[{\rm (iv)}] $(i,p)=(n+1,r+2\ell)$ and $(j,q)=(2n+2-\ell,r-\frac{1}{2}-(n-\ell))$,
\end{enumerate}
for some $\ell\in \Z_{\ge 1}$.
\end{lemma}

\begin{proof}
Let us assume that $\ga=[a,n+1]$ for some $a \le n$ and is contained in the $N$-path $N[a]$ with $(2n+1-a)$-arrows. Note that
$\{ \al,\be \} = \{ [a,k], [k+1,n+1] \}$ for some $a \le k <n+1$.  We assume further that $\be=[k+1,n+1]$ and $\Omega_{[\ii_0]}(\be)=(n+1,r-2\ell)$
for some $\ell \in \Z_{\ge 1}$ (see Corollary \ref{cor: label for non-induced}).  Note that there exists an $S$-path $S[k]$ with $(k-1)$-arrows
\begin{itemize}
\item  whose vertices shares $k$ as their second component,
\item  which intersects with $N[a]$.
\end{itemize}
Furthermore, the vertex located at the intersection of $N[a]$ and $S[k]$ is $[a,k]$. By the assumption that $[a,k] \prec_{[\ii_0]} [a,n+1]$,
the $[\ii_0]$-residue $i$ of $[a,k]$ is strictly less than $n+1$ by Theorem \ref{them: comp for length k ge 0}. By applying \cite[Corollary 1.15]{Oh14A}
and Theorem \ref{them: comp for length k ge 0}, we have the following in $\Upsilon_{[\ii_0]}$:
$$
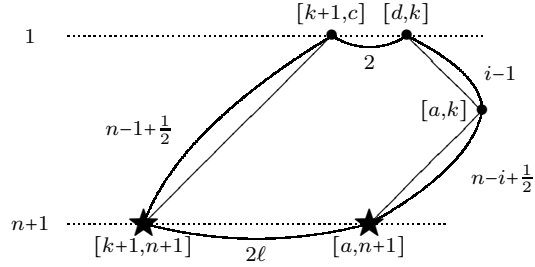
{\xy (0,0)*{}="T1"; (-25,-25)*{}="L1"; (10,0)*{}="T2"; (20,-10)*{}="R1";  (5,-25)*{}="R2";
"T1"; "L1" **\dir{-};"T2"; "R1" **\dir{-};"R1"; "R2" **\dir{-};
"L1"+(-10,0); "R2"+(10,0) **\dir{.}; "L1"+(-15,0)*{_{n+1}};
"T1"+(-35,0); "T2"+(10,0) **\dir{.}; "T1"+(-40,0)*{_{1}};
"T1"*{\bullet};"T2"*{\bullet};"R1"*{\bullet};
"L1"*{\bigstar};"R2"*{\bigstar};
"T1"+(0,3)*{_{[k+1,c]}};"T2"+(0,3)*{_{[d,k]}};
"R1"+(-5,0)*{_{[a,k]}};
"L1"+(0,-3)*{_{[k+1,n+1]}};"R2"+(0,-3)*{_{[a,n+1]}};
"T1"; "T2"; **\crv{(5,-3)}?(.5)+(0,-2)*{\scriptstyle 2};
"L1"; "R2"; **\crv{(-10,-29)}?(.5)+(0,-2)*{_{2\ell}};
"T1"; "L1"; **\crv{(-20.5,-12.5)}?(.5)+(-9,0)*{_{n-1+\frac{1}{2}}};
"R2"; "R1"; **\crv{(18.5,-17.5)}?(.6)+(9,0)*{_{n-i+\frac{1}{2}}};
"T2"; "R1"; **\crv{(20,-5)}?(.5)+(5,0)*{_{i-1}};
\endxy}
$$
Hence we can obtain that $i=\ell$
which yields our first assertion. For the remained cases, one can prove by applying the similar argument.
\end{proof}

Now, we record coordinates of minimal pairs for $\ga \in \PR \setminus \Pi$ in $\widehat{\Upsilon}_{[\ii_0]}$. The following proposition is an immediate
consequence of Lemma \ref{lem: minimal for rds2}, Lemma \ref{lem: minimal for induced rds1} and Lemma \ref{lem: minimal for non-induced central}:

\begin{proposition}
For $\al,\be,\ga \in \PR$ with $\widehat{\Omega}_{[\ii_0]}(\al)=(i,p)$, $\widehat{\Omega}_{[\ii_0]}(\be)=(j,q)$
$\widehat{\Omega}_{[\ii_0]}(\ga)=(k,r)$ and  $\al+\be =\ga$ $(i,j,k \in \widehat{I})$, $(\al,\be)$ is an $[\ii_0]$-minimal pair of $\ga$
if and only if one of the following conditions holds$:$
\begin{eqnarray}&&
\left\{\hspace{1ex}\parbox{80ex}{
\begin{enumerate}
\item[{\rm (i)}]
$\ell \seteq \max(i,j,k) \le n$, $i+j+k=2\ell$
and
$$\hspace{-20ex}   (q-r, p-r) =
\begin{cases}
\big( -i,j \big), & \text{ if } \ell = k,\\
\big( i-(2n+1),j \big), & \text{ if } \ell = i,\\
\big( -i,2n+1-j  \big), & \text{ if } \ell = j.
\end{cases}
$$
\item[{\rm (ii)}]
$s \seteq \min(i,j,k) \le n$, the others are the same as $n+1$ and
$$\hspace{-20ex}  (q-r,p-r) =
\begin{cases}
\big( -(n-k)+1/2,(n-k)-1/2 ), & \text{ if } s = k,\\
\big( -2i-2,(n-i)-1/2  ), & \text{ if } s = i,\\
\big( -(n-j)+1/2, 2j+2), & \text{ if } s = j.
\end{cases}
$$
\end{enumerate}
}\right. \label{eq: Dorey folded coordinate Bn+1}
\end{eqnarray}
\end{proposition}

\subsubsection{$D_{n+1}$} The relative positions for an $[\ii_0]$-minimal pair $(\al,\be)$ for $\ga \in\PR$ follow from
Lemma \ref{Lem:sigle_pair}.
\begin{proposition}
For $\al,\be,\ga \in \PR$ with $\widehat{\Omega}_{[\ii_0]}(\al)=(i,p)$, $\widehat{\Omega}_{[\ii_0]}(\be)=(j,q)$,
$\widehat{\Omega}_{[\ii_0]}(\ga)=(k,r)$ $(i,j,k \in \widehat{I})$ such that  $\al+\be =\ga$, the pair $(\al,\be)$ is an $[\ii_0]$-minimal pair of $\ga$
if and only if one of the following conditions holds$:$
\begin{eqnarray}&&
\left\{\hspace{1ex}\parbox{75ex}{
$\ell \seteq \max(i,j,k) \le n$, $i+j+k=2\ell$
and
$$ \hspace{-20ex} \left( q-r,p-r \right) = \dfrac{1}{2} \times
\begin{cases}
\big( -i,j \big), & \text{ if } \ell = k,\\
\big( i-(2n+2),j \big), & \text{ if } \ell = i,\\
\big( -i,2n+2-j  \big), & \text{ if } \ell = j.
\end{cases}
$$
}\right. \label{eq: Dorey folded coordinate Cn}
\end{eqnarray}
\end{proposition}

\subsection{Twisted additive property} In this subsection, we briefly show that the folded AR-quivers have some property which can be understood as a generalization of the additive property in~\eqref{eq: additive},
by using the results in previous sections.

\begin{proposition}
Let $[\ii_0]\in \lf \Qd \rf$ or $\lf \mathfrak{Q} \rf$ be a (triply) twisted adapted class of type $A_{2n+1}$, $D_{n+1}$, $E_6$ or $D_4$. Suppose $\alpha\in\Phi^+$ has residue $i\in I$ in $\Upsilon_{[\ii_0]}$ and, in the folded AR-quiver, $\widehat{\Omega}_{[\ii_0]}(\al)=(\hat{i},p)$. Let $|\hat{i}|$ be the number of indices in the orbit $\hat{i}$.
If there is $\beta \in \Phi^+$ such that $\widehat{\Omega}_{[\ii_0]}(\beta)= (\hat{i}, p-2 \dfrac{|\hat{i}|}{\mathsf{d}})$  then we have
\begin{align} \label{eq: twisted additive}
\alpha+ \beta=  \sum_{\gamma\in {}_\be [\ii_0]_\al }  \gamma,
\end{align}
where
\begin{align} \label{eq: al be ii0}
{}_\be [\ii_0]_\al=\left\{ \gamma\in \PR \left|  \  \widehat{\Omega}_{[\ii_0]}(\ga) = (  \hat{j},r)  \text{ such that }  \begin{array}{l}  (a) \  |\hat{i}- \hat{j}| =1, \\ (b) \  p-2 \dfrac{|\hat{i}|}{\mathsf{d}}  \le  r \le p  \end{array} \right.  \right \}.
\end{align}
\end{proposition}

\begin{proof}
For types $A_{2n+1}$,  $D_{n+1}$ and $E_6$, the condition
in~\eqref{eq: al be ii0} can be re-interpreted as follows:
$$
{}_\be [\ii_0]_\al=\left\{ \gamma\in \PR  |  \  \text{there exists an arrow } \ga\to\al \text{ or } \be\to\ga \text{ in } \widehat{\Upsilon}_{[\ii_0]} \right\}.
$$

For type $A_{2n+1}$, our assertion is a direct consequence of Theorem~\ref{them: comp for length k ge 0}. For type $D_{n+1}$,
our assertion follows from Lemma~\ref{Lem:add_D} unless $\hat{i}=n$. If $\hat{i}=n$, then our assertion follows from the property of swing and Lemma~\ref{Lem:add_D} together.
For exceptional cases, one can check by direct computations.
\end{proof}

\begin{remark}
\hfill
\begin{enumerate}
\item The set ${}_\be Q_\al$  in ~\eqref{eq: al be Q} coincides with ${}_\be [\ii_0]_\al$ in~\eqref{eq: al be ii0} when $[\ii_0]=[Q]$ for a Dynkin quiver $Q$.
\item Let us take $[\ii_0]\in \lf \Qd \rf$ or $\lf \mathfrak{Q} \rf$ which is associated to a (triply) twisted Coxeter element $\phi_{[\ii_0]} \vee$. Then, the $\beta$ in~\eqref{eq: twisted additive} can be written as follows:
\[
\beta = (\phi_{[\ii_0]} \vee)^{|\hat{i}|}(\al).
\]
\end{enumerate}
Thus,~\eqref{eq: twisted additive} can be said to be the \defn{twisted additive property} of $\Upsilon_{[\ii_0]}$, comparing with~\eqref{eq: additive}.
\end{remark}

\section{Applications on denominators and Dorey's rule for $U_q'(B^{(1)}_{n+1})$ and $U_q'(C^{(1)}_{n})$} \label{Sec: Application}

In this section, we shall show that the denominator formulas and Dorey's rule for $U_q'(B^{(1)}_{n+1})$ and $U_q'(C^{(1)}_{n})$ are well-reflected onto
a folded AR-quiver $\widehat{\Upsilon}_{[\ii_0]}$ for any $[\ii_0] \in \lf \Qd \rf$ of type $A_{2n+1}$ and $D_{n+1}$, respectively. More precisely, we shall prove the
twisted analogues of Theorem \ref{thm: dist denom} and Theorem \ref{thm: Dorey ADE}, by collecting results in previous sections.  We also prove that the additive property of $\widehat{\Upsilon}_{[\ii_0]}$ is
related the T-system of $U_q'(\widehat{X}^{(1)})$.

\begin{theorem} \label{thm: folded dist denom}
For any $[\ii_0] \in \lf \Qd \rf$ of type $X$, the denominator formulas for the quantum affine algebra $U'_q(\widehat{X}^{(1)})$ can be read from $\widehat{\Upsilon}_{[\ii_0]}$
 $(X=A_{2n+1}$ or $D_{n+1})$ and $(\widehat{X}=B_{n+1} \text{ or } C_{n}):$
\begin{align*}
d^{\widehat{X}^{(1)}}_{k,l}(z) & =\widehat{D}^{X}_{k,l}(z) \times (z-q^{\mathsf{h}^\vee})^{\delta_{l,k}},
\end{align*}
where $\mathsf{h}^\vee$ is the dual Coxeter number of type $\widehat{X}$.
\end{theorem}

\begin{proof}[Proof of Theorem \ref{thm: folded dist denom} for $X=A_{2n+1}$ and $\widehat{X}=B_{n+1}$]
Fix $Q$ such that $\PPi([\ii_0])=[Q]$. Recall the denominator formulas for type $A_n^{(1)}$ and $B_{n+1}^{(1)}$ in Theorem \ref{thm: denom 1A2n} (a) and (b). By considering the denominator formulas $d^{A^{(1)}_{2n}}_{k,l}(z)$, $d^{B^{(1)}_{n+1}}_{k,l}(z)$ and the distance polynomials $D^{A_{2n}}_{k,l}(z)$ for $1 \le k,l \le n$,
we have an interesting interpretation as follows:
\begin{equation} \label{eq: interesting interpretation}
\begin{aligned}
\dfrac{d^{B^{(1)}_{n+1}}_{k,l}(z)}{(z-q^{\mathsf{h}^\vee})^{\delta_{kl}}} = D^{A_{2n}}_{k,l}(z) \times D^{A_{2n}}_{k,l^*}(-z)
\end{aligned}
\end{equation}
where $D^{A_{2n}}_{k,l}(z)=D^{A_{2n}}_{l,k}(z)=D^{A_{2n}}_{k^*,l^*}(z)=D^{A_{2n}}_{l^*,k^*}(z)$ are distance polynomials on $\lf \Delta \rf$ of type $A_{2n}$, $*$ is in Definition \ref{Def:inv}
and $\mathsf{h}^\vee$ denotes the dual Coxeter number of $B_{n+1}$.

(1) Assume that $\hat{k}, \hat{l} \in \widehat{I}\setminus \{n+1\}$ where $\widehat{I}=\{1,2,\ldots,n+1\}$. Then one of
$o^{Q}_t(k,l)$ and $o^{Q}_t(k^*,l)$ is positive implies that $o^{[\ii_0]}_t(\hat{k},\hat{l})>0$ and hence
$\mathtt{o}^{[\ii_0]}_t(\hat{k},\hat{l})=1$ by Proposition \ref{prop: less than eq to 2}. Thus our assertion for this case
follows from \eqref{eq: interesting interpretation}.

(2) Assume that $\hat{k} = \hat{l}=n+1$, and write $\Omega_{[\ii_0]}(\al)=(n+1,p)$ and $\Omega_{[\ii_0]}(\be)=(n+1,q)$.
 Then our assertion is obvious since either $\PR \ni \al+\be \prec_{[\ii_0]}^\tb (\al,\be)$ when $|p-q|= 2s-1$ for some $s \ge 2$
and $\gdist_{[\ii_0]}(\al,\be)=0$ otherwise.

(3) In general, it suffices to consider only one folded AR-quiver $\wUp_{[\ii_0]}$ by Proposition \ref{prop: DQ DQ' twisted}. We take the Dynkin quiver
$$Q \ : \ \xymatrix@R=3ex{ *{ \circ }<3pt> \ar@{<-}[r]_<{1}  &*{\circ}<3pt>
\ar@{<-}[r]_<{2} &\cdots\ar@{<-}[r] &*{\circ}<3pt>
\ar@{<-}[r]_<{2n-1} &*{\circ}<3pt>
\ar@{-}[l]^<{\ \ 2n}} \quad \text{ of type $A_{2n}$}$$
and $[\ii_0]$ as $[Q^<]$. Then \cite[(6.20)]{Oh15E} and Corollary \ref{cor:label1}, we can draw $\wUp_{[\ii_0]}$ with its labels. Then one can check that the assertion for
$\hat{k}=n+1$ and $\hat{l} \ne n+1$ holds by reading $\wUp_{[\ii_0]}$. We skip the proof and provide a particular example for $[Q^<]$ of type $A_5$ instead.
\end{proof}

\begin{example} Here are $\Gamma_Q$, $\Upsilon_{[Q^<]}$ and $\wUp_{[Q^<]}$ for
$Q \ : \ \xymatrix@R=3ex{ *{ \circ }<3pt> \ar@{<-}[r]_<{1}  &*{\circ}<3pt>
\ar@{<-}[r]_<{2} &*{\circ}<3pt>
\ar@{<-}[r]_<{3} &*{\circ}<3pt>
\ar@{-}[l]^<{\ \ 4}}$.
\begin{align}
& {\scriptstyle\Gamma_Q} = \raisebox{3em}{\scalebox{0.6}{\xymatrix@C=1ex@R=1ex{
1 & [4]\ar@{->}[ddrr] &&&& [3]\ar@{->}[ddrr] &&  && [2]\ar@{->}[ddrr] &&&& [1] \\
&&&&&& &&\\
2 &&&[3,4]\ar@{->}[ddrr]\ar@{->}[uurr] &&&&[2,3]\ar@{->}[ddrr]\ar@{->}[uurr] &&&& [1,2] \ar@{->}[uurr]\\
&&&&&&&&\\
3 &&&&&[2,4] \ar@{->}[ddrr]\ar@{->}[uurr]&&&& [1,3] \ar@{->}[uurr]\\
&&&&&&&\\
4 &&&&&&&[1,4]\ar@{->}[uurr]
}}}
\hspace{-3.5ex}{\scriptstyle\Upsilon_{[Q^<]}} = \raisebox{3em}{\scalebox{0.6}
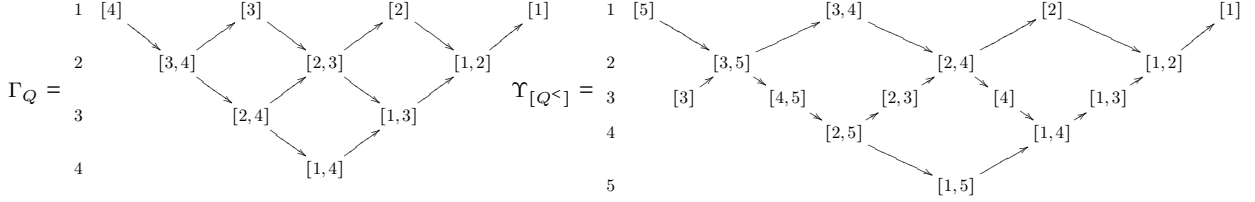
{\xymatrix@C=1ex@R=1ex{
1 & [5] \ar@{->}[ddrr]&&&& [3,4]\ar@{->}[ddrr] && && [2]\ar@{->}[ddrr] &&&& [1] \\
&&&&&& &&\\
2 &&&[3,5]\ar@{->}[dr]\ar@{->}[uurr] &&&&[2,4]\ar@{->}[dr]\ar@{->}[uurr] &&&& [1,2]\ar@{->}[uurr] \\
3 &&[3]\ar@{->}[ur]&&[4,5]\ar@{->}[dr]&&[2,3]\ar@{->}[ur]&&[4]\ar@{->}[dr]&&[1,3]\ar@{->}[ur] \\
4 &&&&&[2,5]\ar@{->}[ur]\ar@{->}[ddrr] &&&& [1,4]\ar@{->}[ur] \\
&&&&&&&\\
5 &&&&&&&[1,5]\ar@{->}[uurr]
}}} \allowdisplaybreaks \nonumber \\
& \qquad\qquad\qquad {\scriptstyle\wUp_{[Q^<]}} =\raisebox{2em}{\scalebox{0.8}
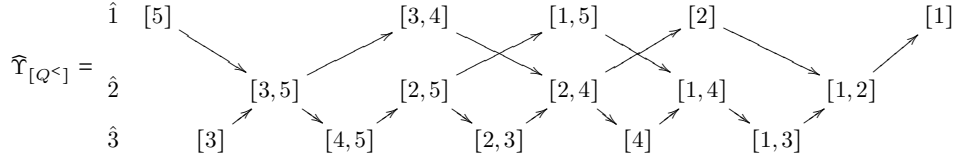
{\xymatrix@C=1ex@R=1ex{
\hat{1} & [5] \ar@{->}[ddrr]&& && [3,4]\ar@{->}[ddrr] &&[1,5]\ar@{->}[ddrr]&& [2]\ar@{->}[ddrr] &&&& [1] \\
&&&&&& &&\\
\hat{2} &&&[3,5]\ar@{->}[dr]\ar@{->}[uurr] &&[2,5]\ar@{->}[uurr]\ar@{->}[dr]&&[2,4]\ar@{->}[dr]\ar@{->}[uurr] &&[1,4]\ar@{->}[dr]&& [1,2]\ar@{->}[uurr] \\
\hat{3} &&[3]\ar@{->}[ur]&&[4,5]\ar@{->}[ur]&&[2,3]\ar@{->}[ur]&&[4]\ar@{->}[ur]&&[1,3]\ar@{->}[ur]
}}} \label{eq: Ar-quiver A5 final}
\end{align}
\end{example}

\begin{proof}[Proof of Theorem \ref{thm: folded dist denom} for $X=D_{n+1}$ and $\widehat{X}=C_{n}$] Recall the denominator formulas for $U_q'(C_n^{(1)})$:
$$d^{C^{(1)}_{n}}_{k,l}(z)= \displaystyle  \prod_{s=1}^{ \min(k,l,n-k,n-l)}
\big(z-(-q^{1/2})^{|k-l|+2s}\big)   \prod_{s=1}^{ \min(k,l)} \big(z-(-q^{1/2})^{2n+2-k-l+2s}\big) \quad 1 \le k,l\le n$$
Then, for $1 \le k,l \le n$, one can observe that
\begin{eqnarray} &&
\parbox{95ex}{
\begin{itemize}
\item[{\rm (i)}] the first factor of $d^{C^{(1)}_{n}}_{k,l}(z)$ is the same as $d^{A^{(1)}_{n-1}}_{k,l}(z)$ $(1 \le k,l \le n-1)$,
\item[{\rm (ii)}] the second factor of $d^{C^{(1)}_{n}}_{k,l}(z)$ is the same as the second factor of $d^{D^{(1)}_{n+2}}_{k,l}(z)$.
\end{itemize}
}\label{eq: CDCD}
\end{eqnarray}
Thus we can apply the same argument of \cite[Theorem 6.18]{Oh15E}. More precisely, {\rm (i)} is induced from (II-1) and (II-2) in \eqref{eq: complacted socle of D},
and {\rm (ii)} is induced from (I-1), (I-2), (I-3) and (I-6) in \eqref{Pic:gdist_D_1}.
\end{proof}

To make a twisted analogue of Theorem \ref{thm: Dorey ADE}, we first define
certain category $\mathscr{C}_{[\ii_0]}$ of modules over quantum affine algebra $U_q'(\widehat{X}^{(1)})$ associated to a (triply) twisted adapted class $[\ii_0]$:

\begin{definition} \label{def: [ii0] module category}
Let $[\ii_0]$ be a (triply) twisted adapted class of type $X$.
For any positive root $\beta$ contained in $\Phi_X^+$, we set the $U_q'(\widehat{X}^{(1)})$-module $V_{[\ii_0]}(\beta)$ as follows:
\begin{align} \label{eq: Vii0(beta)}
V_{[\ii_0]}(\beta) \seteq
\begin{cases}
V(\varpi_{i})_{(-1)^{i}(q^{1/\mathsf{d}})^{p}} & \text{ if $\vee$ is \eqref{eq: B_n} or \eqref{eq: F_4}}, \\
V(\varpi_{i})_{(-q^{1/\mathsf{d}})^{p}} & \text{ if $\vee$ is \eqref{eq: C_n} or \eqref{eq: G_2}},
\end{cases}
\quad \text{where $\widehat{\Omega}_{[\ii_0]}(\beta)=({i},p/\mathsf{d})$}.
\end{align}

Denote by $\mathscr{C}_{[\ii_0]}$ the smallest abelian full subcategory of $\Ca_{\widehat{X}^{(1)}}$ such that
\begin{itemize}
\item[{\rm (a)}] it is stable under subquotient, tensor product and extension,
\item[{\rm (b)}] it contains $V_{[\ii_0]}(\beta)$ for all $\beta \in \Phi_X^+$  and the trivial module.
\end{itemize}
\end{definition}

\begin{theorem}\label{thm: Dorey BC}
Let $(i,x)$, $(j,y)$, $(k,z) \in I \times \mathbf{k}^\times$. Then
$$ {\rm Hom}_{U_q'(\widehat{X}^{(1)})}\big( V(\varpi_{k})_z, V(\varpi_{i})_x \otimes V(\varpi_{j})_y \big) \ne 0 \qquad (\widehat{X}=B_{n+1}\text{ or }C_{n})$$
if and only if there exists a twisted adapted class $[\ii_0] \in \lf \Qd \rf$ of type $X$ $(X=A_{2n+1}$ or $D_{n+1})$ and $\al,\be,\ga \in \Phi_{X}^+$ such that
\begin{itemize}
\item[{\rm (i)}] $(\al,\be)$ is an $[\ii_0]$-minimal pair of $\ga$,
\item[{\rm (ii)}] $V(\varpi_{j})_y  = V_{[\ii_0]}(\be)_t, \ V(\varpi_{i})_x  = V_{[\ii_0]}(\al)_t, \ V(\varpi_{k})_z  = V_{[\ii_0]}(\ga)_t$
for some $t \in \mathbf{k}^\times$.
\end{itemize}
\end{theorem}

\begin{proof}
By comparing \eqref{eq: Dorey folded coordinate Bn+1} and \eqref{eq: Dorey B} (resp. \eqref{eq: Dorey folded coordinate Cn} and \eqref{eq: Dorey C})
our assertion is a consequence of \eqref{eq: Vii0(beta)} in Definition \ref{def: [ii0] module category}.
\end{proof}

Theorem \ref{thm: Dorey BC} can be understood as a generalization of Theorem \ref{thm: Dorey ADE}, in the sense that we interpreted Dorey's rule for $U_q'(\widehat{X}^{(1)})$ using combinatorial AR-quivers. However, in Theorem \ref{thm: Dorey ADE}, the minimality of a pair $(\alpha,\beta)$ of $\gamma$ was not needed. For example, in \eqref{eq: Ar-quiver A5 final}, one can compute that
there does not exist a homomorphism from $V_{[\ii_0]}([1,4])$ to $V_{[\ii_0]}([1,2]) \otimes V_{[\ii_0]}([3,4])$, even though $[1,4]=[1,2]+[3,4]$.

\begin{remark}
By considering the particular case ($k=1$) of the {\it T-system} (see for example \cite[3.2.1]{HL16}), we can interpret the twisted additive property described in \eqref{eq: twisted additive} as follows:
Let $[\ii_0]$ be an adapted or a (triply) twisted adapted class of type $X$.
In the Grothendieck ring of $\mathscr{C}_{[\ii_0]}$, we have
\begin{align} \label{eq: T-sys}
\bigl[V_{[\ii_0]}(\al)\bigl]  \cdot \bigl[V_{[\ii_0]}(\be)\bigr] = \left[{\rm hd}\bigl( V_{[\ii_0]}(\al) \otimes V_{[\ii_0]}(\be) \bigr) \right]+  \prod_{\gamma\in {}_\be [\ii_0]_\al }\bigl[V_{[\ii_0]}(\ga)\bigr],
\end{align}
for $\al,\be$ in \eqref{eq: twisted additive}. Here, $[M]$ denotes the isomorphism class of a $U_q'(\g)$-module $M$ in $\Ca_{U_q'(\widehat{X}^{(1)})}$.
In particular, \eqref{eq: T-sys} implies that the socle of $ V_{[\ii_0]}(\al) \otimes V_{[\ii_0]}(\be)$ is the same as
$ \displaystyle \tens_{i=1}^r V_{[\ii_0]}(\ga_i)$, when ${}_\be [\ii_0]_\al = \{ \ga_1 ,\ldots,\ga_r \}$.
\end{remark}

\appendix

\section{Conjectures on types $F_4^{(1)}$ and $G_2^{(1)}$} \label{Appendix}

As in Section \ref{Sec: Application}, we can read the denominator
formulas and Dorey's rule for $U_q'(B^{(1)}_{n+1})$ (resp. $U_q'(C^{(1)}_{n})$)
from any folded AR-quiver $\widehat{\Upsilon}_{[\ii_0]}$ of type $A_{2n+1}$ (resp. $D_{n+1}$). Thus, we expect
the denominator formulas and Dorey's rule for $U_q'(F^{(1)}_{4})$ and  $U_q'(G^{(1)}_{2})$
from any (triply) folded AR-quiver $\widehat{\Upsilon}_{[\ii_0]}$ of type $E_{6}$ and $D_{4}$.

\smallskip

First, we give a conjectural denominator formula
for $U_q'(G^{(1)}_{2})$.

{\bf Conjectural denominator formulas
for $U_q'(G^{(1)}_{2})$:} Set $q_s$ such that $q_s^3=q$. The conjectural $d_{k,l}(z)$ are given as follows:
$$d_{k,l}(z)=\widehat{D}^{\lf \mathfrak{Q} \rf}_{k,l}(z) \times (z-q^{\mathsf{h}^\vee})^{\delta_{l,k}}.$$
\begin{subequations}
\begin{align}
& d_{1,1}(z)=(z-q_s^6)(z-q_s^8)(z-q_s^{10})(z-q_s^{12}), \label{Eqn:Deno_G d11}   \\
& d_{1,2}(z)=(z+q_s^7)(z+q_s^{11}), \label{Eqn:Deno_G d121} \\
& d_{2,2}(z)=(z-q_s^2)(z-q_s^8)(z-q_s^{12}), \label{Eqn:Deno_G d22}
\end{align}
\end{subequations}
where  $\lf \mathfrak{Q}  \rf$ denotes the triply twisted adapted $r$-cluster points in Definition \ref{def: triply twisted cluster} and the Dynkin diagram of $G_2^{(1)}$ is given as follows:
$$G^{(1)}_2=\xymatrix@R=3ex{ *{ \circ }<3pt> \ar@{-}[r]_<{0} & *{ \circ }<3pt> \ar@{-}[l]^<{1}  &*{\circ}<3pt>
\ar@3{<-}[l]^<{2} } $$

Similarly, one can guess  $U_q'(F^{(1)}_{4})$ can be obtained from $\widehat{D}_{k,l}(z) \times (z-q^{\mathsf{h}^\vee})^{\delta_{l,k}}$. 




{\bf Conjectural denominator formulas
for $U_q'(F^{(1)}_{4})$ :} Set $q_s$ such that $q_s^2=q$. The conjectural formulas for $d_{k,l}(z)$ are given as follows:
\begin{subequations}
\begin{align}
d_{1,1}(z)& =(z-q_s^{4})(z-q_s^{10})(z-q_s^{12})(z-q_s^{18}), \label{Eqn:Deno_F d11} \allowdisplaybreaks \\
d_{1,2}(z)& =(z+q_s^{6})(z+q_s^{8})(z+q_s^{10})(z+q_s^{12})(z+q_s^{14})(z+q_s^{16}), \label{Eqn:Deno_F d12} \allowdisplaybreaks\\
d_{1,3}(z)& =(z-q_s^{7})(z-q_s^{9})(z-q_s^{13})(z-q_s^{15}), \label{Eqn:Deno_F d13} \allowdisplaybreaks \\
d_{1,4}(z)& =(z+q_s^{8})(z+q_s^{14}), \label{Eqn:Deno_F d14} \allowdisplaybreaks \\
d_{2,2}(z)& =(z-q_s^{4})(z-q_s^{6})(z-q_s^{8})^2(z-q_s^{10})^2(z-q_s^{12})^{2} (z-q_s^{14})^2(z-q_s^{16})(z-q_s^{18}), \label{Eqn:Deno_F d22} \allowdisplaybreaks\\
d_{2,3}(z)& =(z+q_s^{5})(z+q_s^{7})(z+q_s^{9})(z+q_s^{11})^2(z+q_s^{13})(z+q_s^{15})(z+q_s^{17}), \label{Eqn:Deno_F d23} \allowdisplaybreaks\\
d_{2,4}(z)& =(z-q_s^{6})(z-q_s^{10})(z-q_s^{12})(z-q_s^{16}),  \label{Eqn:Deno_F d24} \allowdisplaybreaks \\
d_{3,3}(z)& =(z-q_s^{2})(z-q_s^{6})(z-q_s^{8})(z-q_s^{10}) (z-q_s^{12})^2(z-q_s^{16})(z-q_s^{18}), \label{Eqn:Deno_F d33} \allowdisplaybreaks \\
d_{3,4}(z)& =(z+q_s^{3})(z+q_s^{7})(z+q_s^{11})(z+q_s^{13})(z+q_s^{17}), \label{Eqn:Deno_F d34} \allowdisplaybreaks\\
d_{4,4}(z)& =(z-q_s^{2})(z-q_s^{8})(z-q_s^{12})(z-q_s^{18}). \label{Eqn:Deno_F d44} \allowdisplaybreaks
\end{align}
\end{subequations}
where the Dynkin diagram of $F_4^{(1)}$ is given as follows:
$$F^{(1)}_4=\xymatrix@R=3ex{ *{ \circ }<3pt> \ar@{-}[r]_<{0} & *{ \circ }<3pt> \ar@{-}[r]_<{1}  &*{\circ}<3pt>
\ar@{-}[l]^<{2}  &*{\circ}<3pt>
\ar@{<=}[l]^<{3} &*{\circ}<3pt>
\ar@{-}[l]^<{4}  } $$


Now we suggest conjectural Dorey's rule for $U'_q(F_4^{(1)})$ and $U'_q(G_2^{(1)})$ in a canonical way:

\begin{conjecture}\label{conj: Dorey}
Let $(i,x)$, $(j,y)$, $(k,z) \in I \times \ko^\times$. Then
$$ {\rm Hom}_{U_q'(\widehat{X}^{(1)})}\big( V(\varpi_{j})_y \otimes V(\varpi_{i})_x , V(\varpi_{k})_z  \big) \ne 0 \qquad  (\widehat{X}=F_4, G_2) $$
if and only if there exists an $[\ii_0] \in \lf \Qd \rf$ (resp. $[\ii_0] \in \lf \mathfrak{Q} \rf$) and $\al,\be,\ga \in \Phi_{X}^+$ $(X=E_6, D_4)$ such that
\begin{itemize}
\item[{\rm (i)}] $(\al,\be)$ is an $[\ii_0]$-minimal pair of $\ga$,
\item[{\rm (ii)}] $V(\varpi_{j})_y  = V_{[\ii_0]}(\be)_t, \ V(\varpi_{i})_x  = V_{[\ii_0]}(\al)_t, \ V(\varpi_{k})_z  = V_{[\ii_0]}(\ga)_t$
for some $t \in \ko^\times$.
\end{itemize}
\end{conjecture}

The Dorey's rule for $U_q(G_2^{(1)})$ can be conjectured as (see (4) in Example \ref{eq: folded 4})
\begin{subequations}
\begin{align}
V(\varpi_{2})_{-q_s^{-1}} \otimes V(\varpi_{2})_{-q_s} &  \twoheadrightarrow V(\varpi_{1}),  \label{eq: Droey G2(1) Old} \\
V(\varpi_{2})_{q_s^{-4}}\otimes V(\varpi_{2})_{q_s^{4}} &  \twoheadrightarrow V(\varpi_{2}). \label{eq: Droey G2(1) New}
\end{align}
\end{subequations}

The Dorey's rule for $U_q(F_4^{(1)})$ can be conjectured as (see (3) in Example \ref{eq: folded 4})
\begin{subequations}
\begin{align}
 V(\varpi_{4})_{-q_s^{-1}} \otimes V(\varpi_{4})_{-q_s}  &\twoheadrightarrow V(\varpi_{3}), \quad V(\varpi_{1})_{-q_s^{-2}} \otimes V(\varpi_{1})_{-q^2_s}  \twoheadrightarrow V(\varpi_{2}),
\label{eq: Droey F4(1) Old} \allowdisplaybreaks  \\
 V(\varpi_4)_{q_{s}^{-6}} \otimes V(\varpi_4)_{q_{s}^{6}} &\twoheadrightarrow V(\varpi_4),  \quad \ \ \  V(\varpi_4)_{q_{s}^{-4}} \otimes V(\varpi_4)_{q_{s}^{4}} \twoheadrightarrow V(\varpi_1)_{-1},
\label{eq: Droey F4(1) New 1} \allowdisplaybreaks  \\
 V(\varpi_3)_{-q_{s}^{-1}} \otimes V(\varpi_4)_{q_{s}^{2}} &\twoheadrightarrow V(\varpi_2), \quad \ V(\varpi_3)_{-q_{s}^{-5}} \otimes V(\varpi_1)_{q_{s}^{10}} \twoheadrightarrow V(\varpi_4),
\label{eq: Droey F4(1) New 2} \allowdisplaybreaks  \\
 V(\varpi_1)_{q_{s}^{-6}}  \otimes V(\varpi_1)_{q_{s}^{6}} &\twoheadrightarrow V(\varpi_1). \label{eq: Droey F4(1) New 3}
\end{align}
\end{subequations}
We remark that \eqref{eq: Droey G2(1) Old} and \eqref{eq: Droey F4(1) Old} are given in \cite[Page 86]{Her06}.

\begin{remark}
After completing this paper, Travis and the first named author proved all conjectures in Appendix~\ref{Appendix} in \cite{OT18}.
\end{remark}

\section{Twisted Dynkin quiver} \label{Appendix:twisted Dynkin quiver}
In this appendix, we introduce a twisted Dynkin quiver associated to a twisted Coxeter element.
\begin{definition} \label{Def:twistedDynkinQuiver} \hfill
\begin{enumerate}
\item A \defn{twisted Dynkin quiver} $Q\vee$ for $\vee$ in \eqref{eq: B_n},\eqref{eq: C_n} or \eqref{eq: F_4} has the following properties.
\begin{enumerate}[(i)]
\item $Q\vee$ consists of vertices of the form $\left( \substack{ i_{k} \\ i^\vee_{k}}\right)$ for $k=1,2, \cdots, n+1$, and $i_k= k \text{ or } k^\vee$. When $k= k^\vee$, we use $k$ instead of $\left(\substack{ k\\ k^\vee} \right)$ and $\left(\substack{ k^\vee\\ k} \right)$.
\item Two vertices $\left( \substack{ k_1 \\ k_1^\vee } \right)$ and $\left( \substack{ k_2 \\ k_2^\vee} \right)$ are connected by an arrow if $k_1$ and $k_2$ are connected in $\Delta$.
\item Two vertices $\left( \substack{ k_1 \\  k_1^\vee}  \right)$ and $\left( \substack{  k_2 \\ k_2^\vee} \right)$ are connected by an edge (without orientation)
if $k_i \ne k_i^\vee$ $(i=1,2)$, and $k_1$ and $k_2^\vee$ are connected in $\Delta$.
\end{enumerate}
\item A \defn{triply twisted Dynkin quiver} $Q\vee$ (resp. $Q\vee^2$) of type $D_4$ associated to $\vee$ (resp. $\vee^2$) in  \eqref{eq: G_2} has the following properties.
\begin{enumerate}[(i)]
\item $Q\vee$ (resp. $Q\vee^2$) consists of two vertices, $2$ and  $\left(\substack{ 1\\3\\4}\right)$, $\left(\substack{ 3\\4\\1}\right)$ or $\left(\substack{ 4\\1\\3}\right)$  (resp.  $\left(\substack{ 1\\4\\3}\right)$, $\left(\substack{ 3\\1\\4}\right)$ or $\left(\substack{ 4\\3\\1}\right)$).
\item Two vertices are connected by an arrow.
\end{enumerate}
\end{enumerate}
\end{definition}

Note that a type $D_{n+1}$ twisted Dynkin quiver does not have edges, since there are no vertices satisfying (iii) in Definition \ref{Def:twistedDynkinQuiver} (1).

\begin{example} \hfill
\begin{enumerate}
\item The following two twisted Dynkin quivers of type $A_5$ are associated to $143\vee$ and $123\vee$:
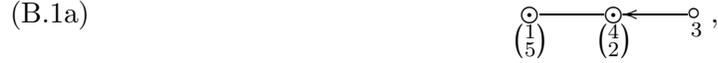
\begin{subequations}
\begin{align}
& \xymatrix@R=3ex{ *{\odot}<3pt>    \ar@{-}[r]_<{ \left(\substack{1\\5}\right) \ \  }&*{\odot}<3pt>
\ar@{<-}[r]_<{\left(\substack{4\\2}\right) \ \  }  &*{\circ}<3pt> \ar@{ }[l]^<{\ \ 3}
}, \label{twisted Dynkin A-1} \\
& 
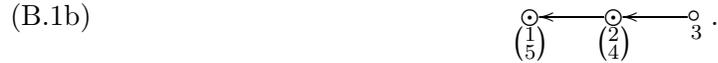
\xymatrix@R=3ex{ *{\odot}<3pt>   \ar@{<-}[r]_<{ \left(\substack{1\\5}\right) \ \  }&*{\odot}<3pt>
\ar@{<-}[r]_<{\left(\substack{2\\4}\right) \ \ }  &*{\circ}<3pt> \ar@{ }[l]^<{\ \ 3}
}. \label{twisted Dynkin A-2}
\end{align}
\end{subequations}
\item The following two twisted Dynkin quivers of type $D_6$ are associated to $13245\vee$ and $13246\vee$:
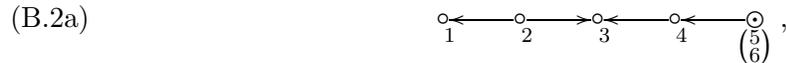
\begin{subequations}
\begin{align}
&
\xymatrix@R=3ex{ *{ \circ }<3pt> \ar@{<-}[r]_<{1}  &*{\circ}<3pt>
\ar@{->}[r]_<{2}  &*{\circ}<3pt>
\ar@{<-}[r]_<{3} &*{\circ}<3pt>
\ar@{<-}[r]_<{4}  & *{\odot}<3pt> \ar@{-}[l]^<{\ \ \left(\substack{5\\6}\right)}
}, \label{twisted Dynkin D-1} \\
& 
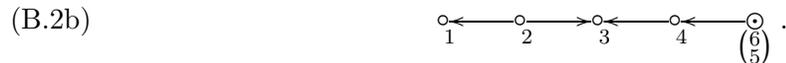
\xymatrix@R=3ex{ *{ \circ }<3pt> \ar@{<-}[r]_<{1}  &*{\circ}<3pt>
\ar@{->}[r]_<{2}  &*{\circ}<3pt>
\ar@{<-}[r]_<{3} &*{\circ}<3pt>
\ar@{<-}[r]_<{4}  & *{\odot}<3pt> \ar@{-}[l]^<{\ \ \left(\substack{6\\5}\right)}
}. \label{twisted Dynkin D-2}
\end{align}
\end{subequations}
\end{enumerate}
\end{example}

\begin{remark} \label{rem: twisted adapted meaning} Considering the number of indices for each vertex, the underlying graph of a (triply) twisted Dynkin quiver can be understood as the Dynkin diagram of type $B_n$,
$C_n$ and $F_4$ (resp. $G_2$), respectively.
\end{remark}

\begin{definition}
A vertex $v$ of a twisted Dynkin quiver is called a \defn{sink} if every arrow (oriented edge) connected to $v$ points towards $v$.
\end{definition}

\begin{example}\hfill
\begin{enumerate}
\item In \eqref{twisted Dynkin A-1}, $ \left(\substack{1\\5}\right)$ and  $\left(\substack{4\\2}\right)$ are sinks, while $3$ is not. In \eqref{twisted Dynkin A-2}, $\left(\substack{1\\5}\right)$ is a sink, while the others are not.
\item In \eqref{twisted Dynkin D-1}, $1$ and $3$ are sinks, while the others are not. In \eqref{twisted Dynkin D-2}, $1$ and $3$ are sinks, while the others are not.
\end{enumerate}
\end{example}

\begin{definition} For a twisted Dynkin quiver $Q\vee$ and $i \in I$, we define $r_i(Q\vee)$ as follows:
\begin{enumerate}
\item  For a twisted Dynkin quiver, if $\left(\substack{i\\i^\vee}\right)$ is a sink of $Q\vee$, we define $r_i (Q\vee)$ by the following steps.
\begin{enumerate}[(i)]
\item Reverse all arrows incident to $\left(\substack{i\\i^\vee}\right)$.
\item Replace the $\left(\substack{i\\i^\vee}\right)$ by $\left(\substack{i\\i^\vee}\right)^\vee \seteq \left(\substack{i^\vee\\i}\right)$.
\item For $i \ne i^\vee$ and $j \ne j^\vee$, if there exists an arrow between $\left(\substack{i^\vee\\i}\right)$ and $\left(\substack{j\\j^\vee}\right)$, and
$i^\vee$ and $j^\vee$ are connected in $\Delta$, remove the orientation of the arrow.
\item For $i \ne i^\vee$ and $j \ne j^\vee$, if there exists an edge between $\left(\substack{i^\vee\\i}\right)$ and $\left(\substack{j\\j^\vee}\right)$, and
$i^\vee$ and $j $ are connected in $\Delta$, give an orientation of the edge from $\left(\substack{i^\vee\\i}\right)$ to $\left(\substack{j\\j^\vee}\right)$.
\end{enumerate}
If $\left(\substack{i\\i^\vee}\right)$ is not a sink, $r_i(Q\vee) \seteq Q\vee$.
\item  If $\left(\substack{i\\i^\vee\\ i^{2\vee}}\right)$ is a sink of a triply twisted Dynkin quiver $Q\vee$, reverse all arrows incident with $\left(\substack{i\\i^\vee \\i^{2\vee}}\right)$ and replace $\left(\substack{i\\i^\vee\\i^{2\vee}}\right)$ by $\left(\substack{i\\i^\vee\\i^{2\vee}}\right)^\vee \seteq \left(\substack{i^\vee\\i^{2\vee} \\i}\right)$. Otherwise, $r_i(Q\vee)=Q\vee$.
\end{enumerate}
\end{definition}

\begin{example} Note that, for a twisted Dynkin quiver
\[ \xymatrix@R=3ex{ *{\odot}<3pt>   \ar@{-}[r]_<{ \left(\substack{1\\5}\right) }&*{\odot}<3pt>
\ar@{<-}[r]_<{\left(\substack{4\\2}\right)}  &*{\circ}<3pt> \ar@{ }[l]^<{\ \ 3}},  \]
$\left( \substack{1 \\ 5}\right) ,\left( \substack{4 \\ 2}\right)$ are sink, while $\left( \substack{2 \\ 4}\right),\left( \substack{3 \\ 3}\right),\left( \substack{5 \\ 1}\right)$ are not. Then we have
\begin{enumerate}
\item $r_1 \left(\xymatrix@R=3ex{ *{\odot}<3pt>   \ar@{-}[r]_<{ \left(\substack{1\\5}\right) }&*{\odot}<3pt>
\ar@{<-}[r]_<{\left(\substack{4\\2}\right)}  &*{\circ}<3pt> \ar@{ }[l]^<{\ \ 3}
}\right)=\xymatrix@R=3ex{ *{\odot}<3pt>   \ar@{->}[r]_<{ \left(\substack{5\\1}\right) }&*{\odot}<3pt>
\ar@{<-}[r]_<{\left(\substack{4\\2}\right)}  &*{\circ}<3pt> \ar@{ }[l]^<{\ \ 3}
}$ and $r_4 \left(\xymatrix@R=3ex{ *{\odot}<3pt>   \ar@{-}[r]_<{ \left(\substack{1\\5}\right) }&*{\odot}<3pt>
\ar@{<-}[r]_<{\left(\substack{4\\2}\right)}  &*{\circ}<3pt> \ar@{ }[l]^<{\ \ 3}
}\right)=\xymatrix@R=3ex{ *{\odot}<3pt>   \ar@{<-}[r]_<{ \left(\substack{1\\5}\right) }&*{\odot}<3pt>
\ar@{->}[r]_<{\left(\substack{2\\4}\right)}  &*{\circ}<3pt> \ar@{ }[l]^<{\ \ 3}
},$
\item $r_2 \left(\xymatrix@R=3ex{ *{\odot}<3pt>   \ar@{-}[r]_<{ \left(\substack{1\\5}\right) }&*{\odot}<3pt>
\ar@{<-}[r]_<{\left(\substack{4\\2}\right)}  &*{\circ}<3pt> \ar@{ }[l]^<{\ \ 3}
}\right)=\xymatrix@R=3ex{ *{\odot}<3pt>   \ar@{-}[r]_<{ \left(\substack{1\\5}\right) }&*{\odot}<3pt>
\ar@{<-}[r]_<{\left(\substack{4\\2}\right)}  &*{\circ}<3pt> \ar@{ }[l]^<{\ \ 3}
}$ and  $r_3 \left(\xymatrix@R=3ex{ *{\odot}<3pt>   \ar@{-}[r]_<{ \left(\substack{1\\5}\right) }&*{\odot}<3pt>
\ar@{<-}[r]_<{\left(\substack{4\\2}\right)}  &*{\circ}<3pt> \ar@{ }[l]^<{\ \ 3}
}\right)=\xymatrix@R=3ex{ *{\odot}<3pt>   \ar@{-}[r]_<{ \left(\substack{1\\5}\right) }&*{\odot}<3pt>
\ar@{<-}[r]_<{\left(\substack{4\\2}\right)}  &*{\circ}<3pt> \ar@{ }[l]^<{\ \ 3}
} .$
\end{enumerate}
\end{example}

\begin{definition}
A reduced expression $\redex= i_1\, i_2\, \cdots \, i_m$ of type $A_{2n-1}$, $D_{n+1}$ or $E_6$ (resp. $D_4$) is said to be \defn{adapted} to $Q\vee$
\[ \text{ if $\left( \substack{i_k\\i^\vee_k} \right)$ (resp. $\left( \substack{i_k\\i^\vee_k\\ i^{2\vee}_k} \right)$)  is a sink of the quiver }  r_{i_{k-1}}r_{i_{k-1}}\cdots r_{i_2}r_{i_1} Q\vee \]
for $k=1,2,\cdots, m.$
\end{definition}

\begin{proposition}
There is a natural one-to-one correspondence between (triply) twisted Dynkin quivers and  (triply) twisted Coxeter elements  defined as follows:
\begin{equation} \label{Eqn:twisted DQuiver-Coxeter}
 Q\vee  \longmapsto   \phi_{Q\vee}  \vee.
\end{equation}
where $\phi_{Q\vee}$ is an element in $W$ all of whose reduced expressions are adapted to $Q\vee$.
\end{proposition}

\begin{proof}
By definitions, there exists only one $\phi_{Q\vee}\vee$  which is  adapted to $Q\vee$.  Conversely, one can find $Q\vee$ from $\phi_{Q\vee}\vee= i_1\, i_2\, \cdots \, i_n \vee$ as follows.
\begin{itemize}
\item If $i_{k_i}\ne i^\vee_{k_i}$ $(i=1,2)$, and  $i_{k_1}$ and $i^\vee_{k_2}$ are connected in $\Delta$, then two vertices $\left( \substack{i_{k_1} \\ i^\vee_{k_1}} \right)$ $\left( \substack{i_{k_2} \\ i^\vee_{k_2}} \right)$ in $Q\vee$ are connected by an edge (without direction).
\item  If $i_{k_1}$ and $i_{k_2}$ are connected in $\Delta$ and $k_1<k_2$, then there is an arrow in $Q\vee$ from $\left( \substack{i_{k_2} \\ i^\vee_{k_2}} \right)$ to $\left( \substack{i_{k_1} \\ i^\vee_{k_1}} \right)$.
\end{itemize}
Hence \eqref{Eqn:twisted DQuiver-Coxeter} is a one-to-one correspondence. The assertion for triply twisted case can be proved directly.
\end{proof}

Recall \eqref{Eqn: 2n+1 twsited}, \eqref{Eqn:twisted D reduced expression} and \eqref{Eqn:E_6 twisted reduced} (resp. \eqref{Eqn:D_4 twisted reduced}), a (triply) twisted Coxeter element $\phi_{Q\vee}\vee$ induces a reduced expression of $w_0$. Let us denote the class associated to the (triply) twisted Coxeter element $\phi_{Q\vee}\vee = i_1 \ i_2 \cdots i_\ell \vee$ by $[Q\vee]$:
\begin{align} \label{eq: natural}
[Q\vee] \seteq \left[ \prod_{k=0}^{(|\PR|/\ell)-1} (i_1 \ i_2 \cdots i_\ell)^{k\vee} \right].
\end{align}

Considering Remark~\ref{rem: twisted adapted meaning}, the following theorem tells that a reduced expression adapted to $Q\vee$ can be understood as a reduced expression
which is {\it related} to a Dynkin diagram of type  $B_n$, $C_n$ or $F_4$ (resp. $G_2$), respectively.

\begin{proposition}
If $\ii_0$ is adapted to $Q\vee$ then $\ii_0\in [Q\vee]$.
\end{proposition}

\begin{proof}
Here, we only give the proof for twisted cases.
Take the adapted reduced expression \eqref{eq: natural} and denote it by $\ii'_0=i_1\, i_2\, \cdots \, i_\N$. Note that $\ii'_0$ is adapted to $Q\vee$.

For a reduced expression $\jj_0=j_1\, j_2\, \cdots \, j_\N$ adapted to $Q\vee$, let us assume
\begin{enumerate}[(i)]
\item $j_t=i_t$ for all $t=1,2,\cdots, k-1$,
\item $j_k\neq i_k$,
\item $k_1$ is the smallest integer such that $k_1>k$ and $j_{k_1}=i_k$.
\end{enumerate}
By the assumption, both ${j_k\choose j^\vee_k }$ and ${i_k \choose i_k^\vee}$ are sinks of the twisted quiver $i_{k-1}\, \cdots \, i_1 Q\vee$. Hence $j_k$ and $i_k$ are not connected in $\Delta$. Moreover, for any $k_2$ such that $k<k_2<k_1$, $j_{k_2}$ is not connected to $i_k$ in $\Delta$. Hence there is $\jj'_0 =j'_1\, j'_2 \, \cdots\, j'_\N  \in [\jj_0]$ such that $j'_t=i_t$ for all $t=1,2,\cdots, k$. By induction, we can show $\jj_0\in [\ii'_0]=[Q\vee].$
\end{proof}

\begin{proposition} \label{Thm:typeD_twsted Dynkin-expression}
Consider a reduced expression $\ii_0$ of type $D_{n+1}$ (resp. $D_4$) and a twisted (resp. triply twisted) adapted class $[Q\vee]$. If $\ii_0 \in [Q\vee]$ then $\ii_0$ is adapted to $Q\vee$.
\end{proposition}

\begin{proof} We shall prove only for twisted cases since the triply twisted cases can be proved directly.

Let us show if $\ii'_0$ is adapted to $Q\vee$ then any $\ii''_0\in [\ii'_0]$ is adapted to $Q\vee$. To see this claim, we aim to show if  $\ii'_0=i_1\, i_2\, \cdots i_{m-1}\,  i_m\, i_{m+1} \, i_{m+2}\cdots \, i_\N$ is adapted to $Q\vee$ and $[i_m \, i_{m+1}]= [i_{m+1} \, i_m]$, then $\ii''_0=i_1\, i_2\, \cdots i_{m-1}\,  i_{m+1}\, i_{m} \, i_{m+2}\, \cdots \, i_\N$  is also adapted to $Q\vee$. Note that $i_m$ and $ i_{m+1} $ are in distinct orbits and $i_m$ and $i_{m+1}$ are not connected in $\Delta$. Moreover, $i_m$ and $i^\vee_{m+1}$ (resp. $i^\vee_m$ and $i_{m+1}$) are not connected in $\Delta$ since we only consider the type $D_{n+1}$ case.

  By the observations, $\left( \substack{i_m \\i^\vee_m}\right)$ and $\left( \substack{i_{m+1} \\i^\vee_{m+1}}\right)$  are both sinks    in $r_{i_{m-1}}\cdots r_{i_2} r_{i_1} Q\vee$ and
$r_{i_m}$ (resp. $r_{i_{m+1}}$) does not change any arrows  incident to $\left( \substack{i_{m+1} \\i^\vee_{m+1}}\right)$ (resp. $\left( \substack{i_{m} \\i^\vee_{m}}\right)$). Hence
\[ i_{m+1}  \text{ (resp. $i_m$) is  a sink in } r_{i_m} r_{i_{m-1}}\cdots r_{i_2} r_{i_1} Q\vee \text{ (resp. $r_{i_{m+1}} r_{i_{m-1}}\cdots r_{i_2} r_{i_1} Q\vee$)}\]
and
\[ r_{i_{m+1}} r_{i_m} r_{i_{m-1}}\cdots r_{i_2} r_{i_1} Q\vee= r_{i_{m}}r_{i_{m+1}} r_{i_{m-1}}\cdots r_{i_2} r_{i_1} Q\vee.\]
 Hence $\ii'_0$ is adapted to $Q\vee$ if and only if $\ii''_0$ is adapted to $Q\vee$.

Now, since $\displaystyle\prod_{k=0}^{(|\PR|/\ell)-1} (i_1  \ i_2 \cdots i_\ell)^{k\vee} $ in \eqref{eq: natural} is adapted to $Q\vee$, we proved the proposition.
\end{proof}

\begin{remark}
Proposition \ref{Thm:typeD_twsted Dynkin-expression} does not hold for type $A_{2n+1}$ and $E_6$.  For example, consider the twisted Dynkin quiver of type $A_5$:
\[ Q\vee=\xymatrix@R=3ex{ *{\odot}<3pt>   \ar@{<-}[r]_<{ \left(\substack{1\\5}\right) }&*{\odot}<3pt>
\ar@{<-}[r]_<{\left(\substack{2\\4}\right)}  &*{\circ}<3pt> \ar@{ }[l]^<{\ \ 3}}.  \]
Then $\ii_0= 1\, 2\, 3\, 5\, 4\, 3\, 1\, 2\,3\, 5\, 4\,3\, 1\, 2\, 3$ is adapted to $Q\vee$ and $\ii'_0=5\, 1\, 2\, 3\, 4\, 3\, 1\, 2\,3\, 5\, 4\,3\, 1\, 2\, 3\in [\ii_0]=[Q\vee]$. However, $\ii'_0$ is not adapted to $Q\vee$.
\end{remark}

The following remark can be understood as the twisted analogue of \eqref{eq:correspondence}.

\begin{remark} \label{rem: twisted correspondence}
Every twisted adapted class of type $D_{n+1}$ is induced from a twisted Coxeter element. Hence, in this case, we have the following one-to-one correspondences:
\begin{equation*}
\{ \text{twisted Dynkin quivers} \} \leftrightarrow \{ \text{twisted Coxeter elements} \} \leftrightarrow  \{ \text{twisted adapted classes} \}.
\end{equation*}
Also, since the number of triply adapted classes $\lf \mathfrak{Q} \rf \seteq \lf \mathfrak{Q}^{\dagger} \rf \bigsqcup \lf \mathfrak{Q}^{\ddagger} \rf$ of type $D_4$ is the same as the number of triply twisted Coxeter elements, we have
\begin{equation*}
\{ \text{triply twisted Dynkin quivers} \} \leftrightarrow \{ \text{triply twisted Coxeter elements} \} \leftrightarrow  \{ \text{triply twisted adapted classes} \}.
\end{equation*}
On the other hand, for type $A_{2n+1}$ and $E_6$ twisted adapted classes, we have
\begin{equation*}
\{ \text{twisted Dynkin quivers} \} \leftrightarrow \{ \text{twisted Coxeter elements} \} \hookrightarrow  \{ \text{twisted adapted classes} \}.
\end{equation*}
\end{remark}


\end{document}